\documentclass{article}
\usepackage[utf8]{inputenc}
\usepackage[T1]{fontenc}
\usepackage{amsmath}
\usepackage{amsthm}
\usepackage{amssymb}
\usepackage{amsfonts}
\usepackage{tikz}
\usepackage{microtype}
\usepackage{hyperref}
\usepackage{pde-translations-paper-code/lean-links/leanlinks}
\usepackage{natbib}
\usepackage{mathtools}
\usepackage{verbatim}
\usepackage[margin=1.2in]{geometry}
\usepackage[textsize=small]{todonotes}
\usepackage{subcaption}
\usepackage{booktabs}
\usepackage{accents}

\usepackage{pgfplots}
\usepgfplotslibrary{groupplots,dateplot}
\usetikzlibrary{patterns,shapes.arrows,fit,arrows.meta}
\pgfplotsset{compat=newest}

\let\b=\boldsymbol

\newtheorem{theorem}{Theorem}
\newtheorem{definition}[theorem]{Definition}
\newtheorem{lemma}[theorem]{Lemma}
\newtheorem{corollary}[theorem]{Corollary}
\newtheorem{prop}[theorem]{Proposition}
\newtheorem{property}[theorem]{Property}
\newtheorem{remark}[theorem]{Remark}
\newtheorem{assumption}[theorem]{Assumption}

\def\bigO#1{\mathrm{O}\!\left(#1\right)}
\def\norm#1{\left\|#1\right\|}
\def\D#1{\partial_{\b #1}}
\def\DD#1{\mathcal{D}^{#1}}

\def\greenf{\mathcal G}
\def\bnu{\boldsymbol{\nu}}
\def\dlmfref#1{\citep[(\href{https://dlmf.nist.gov/#1}{#1})]{NIST:DLMF}}

\def\betam{\b \beta}
\def\betamp{\betam'}
\def\alphal{\b \gamma}
\def\alphalp{\alphal'}
\def\alpham{\b \alpha}
\def\alphamp{\alpham'}

\def\jcomp{\bar{\b j}}

\newcommand\munderbar[1]{%
  \underaccent{\bar}{#1}}
\def\ubw{\munderbar{w}}

\makeatletter

\let\c@table\c@figure
\makeatother

\newcommand{\doi}[1]{\href{https://doi.org/#1}{DOI: #1}}

\makeatletter
\def\input@path{{figures/}}

\title{Automatic Synthesis of Low-Complexity Translation~Operators for the Fast~Multipole~Method}
\author{Isuru Fernando, Chaoqi Lin and Andreas Kl{\"o}ckner}

\begin{document}

\maketitle

\begin{abstract}
The FMM is the leading approach
for attaining linear complexity in the evaluation of long-range
(e.g. Coulomb) many-body interactions.
We demonstrate a new, hybrid symbolic-numerical method for the
automatic synthesis of all families of translation operators required
for the execution of the Fast Multipole Method (FMM). Our method is
dimension-independent (demonstrated in 2D and 3D) and applicable to any
translation-invariant kernel.
Low complexity in translation operators for the FMM
is usually achieved by algorithms specialized for a potential obeying a specific
partial differential equation (PDE). Absent a PDE or specialized
algorithms, Taylor series based FMMs or kernel-independent FMM have
been used, at asymptotically higher expense.

When symbolically provided with a constant-coefficient elliptic
PDE obeyed by the potential, our algorithm can automatically
synthesize translation operators requiring $\bigO{p^d}$ operations,
where $p$ is the expansion order and $d$ is dimension, compared with
$\bigO{p^{2d}}$ operations in a naive approach carried out on (Cartesian) Taylor expansions.  This is achieved by using a
compression scheme that asymptotically reduces the number of terms in the Taylor expansion
and then operating directly on this ``compressed'' representation.
Judicious exploitation of shared subexpressions permits
formation, translation, and evaluation of local and multipole
expansions to be performed in $\bigO{p^{d}}$ operations, while an FFT-based
scheme permits multipole-to-local translations in $\bigO{p^{d-1}\log(p)}$ operations.
We demonstrate computational scaling of code generation and
evaluation as well as numerical accuracy through numerical
experiments on a number of potentials from classical physics.

\end{abstract}

\section{Introduction}

The Fast Multipole Method (FMM)~\citep{greengard1986fast} is an algorithm to attain linear complexity (in the
number of ``source'' and ``target'' or  ``observation'' points) in the
evaluation of a many-body potential
\begin{equation}
  \phi(\b x) = \sum_{\b y \in S}G(\b x- \b y) w(\b y), \qquad (\b x\in T),
  \label{eq:particle-convolution}
\end{equation}
where $S\subset \mathbb R^d$ is a set of source particles and
$T\subset \mathbb R^d$ a set of target particles, $G: \mathbb R^d \to \mathbb C$ is a
kernel, and $w$ is a density function. As a computational primitive,
the FMM has proven crucial in molecular dynamics, particle-based
gravitational simulation, and as a building block of solvers for
boundary value problems of elliptic partial differential equations
based on boundary and volume integral equations, cf.~\citep{chew_integral_2008}
for example.

The FMM proceeds by expanding the far-field of the potential in
``local'' and ``multipole'' expansions that are nested recursively along a quad-
or octree in a box geometry. To ensure accuracy of the expansion
(resulting from the smoothness of the far-field),
all expansions are only used away from their source geometries.
Multipole expansions attain accuracy away from their expansion center, and
local expansions (much like conventional Taylor expansions) attain
accuracy close to their center. Common types of expansions for
multipoles and locals include spherical harmonic series,
complex Laurent series, or Taylor series.

\citet{greengard1986fast}~introduced a (non-adaptive) FMM algorithm for the Laplace equation in
two dimensions, using a complex-valued Laurent series for expressing the Coulomb potential
$\log(R)$ where $R$ is the distance between the target and the source point.
The expansions for this scheme have $\bigO{p}$ terms and translations cost
$\bigO{p^2}$ operations where $p$ is the order of the expansion.
\citet{greengard1988rapid} extended the FMM to the three-dimensional Laplace equation
using a spherical harmonic series which has $\bigO{p^2}$ terms in the expansion
and allows translations that cost $\bigO{p^4}$ operations.
\citet{greengard1988efficient} improved on the FMM by reducing the cost of translations
to $\bigO{p \log(p)}$ for two dimensions and to $\bigO{p^2 \log(p)}$
for three dimensions using a Fast Fourier Transform (FFT).
\citet{greengard2002new} introduced an $\bigO{p^3}$ algorithm for the three-dimensional
Yukawa equation based on spherical harmonic expansions by converting
to plane wave expansions before translations to achieve a better time complexity.

To remedy the kernel-specific nature of prior developments,
\citet{ying2004kernel} introduced a kernel-independent FMM algorithm. The kernel-independent FMM
works for a broad class of potentials without needing kernel-specific
translation operators to be supplied, unlike
prior approaches based on spherical harmonics or plane-wave-based expansions.
The algorithm likewise uses a Fast Fourier
Transform, but uses asymptotically more terms in each expansion than the algorithm in
\citep{greengard1988efficient} (by a factor of $\bigO{p}$), and
therefore is a $\bigO{p^3 \log(p)}$ algorithm in three
dimensions.

FMM variants based on Taylor series have been developed in two
and three dimensions. \citet{zhao1987an} gives an algorithm for
the three-dimensional Laplace equation, where each expansion
has $\bigO{p^3}$ terms and a translation costs $\bigO{p^6}$.
The cost of the translations for Taylor series-based expansions is an obstacle
for widespread usage. \citet{oppelstrup2013matrix} uses common
subexpression elimination (via a computer algebra system) to reduce
the cost from $\bigO{p^6}$ to an (empirically obtained) $\bigO{p^{4.5}}$.
\citet{shanker2007accelerated} give a Taylor series-based algorithm
with a multipole-to-local translation cost of $\bigO{p^4}$
for kernels of the form $R^{-\nu}$ ($\nu \in \mathbb R$).
This includes the three-dimensional Laplace problem at $\nu = 1$.
This algorithm re-writes translation operations as tensor
computations and uses tensor contraction algorithms.
\citet{coles2020optimizing} extend \citeauthor{shanker2007accelerated}'s work
coupled with common subexpression elimination to reduce the cost of
multipole-to-local translation to $\bigO{p^{3.6}}$ empirically.

One advantage of methods based on Taylor series is that they can be easily
generalized to arbitrary kernels. In order to do that, partial
derivatives of the kernel have to be obtained, which can be
both numerically subtle and computationally inconvenient.
\citet{tausch2003fast} proposes
an algorithm to calculate the derivatives using a recurrence formula for
all radially symmetric kernels given that all radial derivatives are
known and gives formulae for the radial derivatives for (free-space) Green's functions
of the Laplace equation, the Helmholtz equation, as well as the equations of
linear elasticity and Stokes flow.
\citet{zhang2011fast} use a differential algebra framework
for calculating the derivatives. \citet{coles2020optimizing}
use a symbolic algebra framework to calculate the derivatives.
While each of these latter methods is aimed at a generalization of the FMM,
their approach to expansion and translation does not yield low
complexity competitive with the manually-derived approaches described
above. Specifically, they use expansions with $\bigO{p^d}$
coefficients, and some of the translations require $\bigO{p^{2d}}$ work.

In this paper, we present, to the best of our knowledge, the first
comprehensive black-box toolchain for the synthesis of Fast Multipole Methods
for general kernels that attains complexities competitive with
manually-derived approaches, from limited amounts of user-provided
information. Our method provably attains the same order of truncation
error as other state-of-the-art methods in the literature. The
complexity of our approach decreases as more information is made
available by the user.
\begin{itemize}
  \item
    With only a symbolic expression of the
    kernel, our method requires $\bigO{p^d}$ terms in the expansion.
    It attains $\bigO{p^d \log p}$ cost of a multipole-to-local translation
    (with $\bigO{p^{d+1}}$ offline precomputation cost),
    and no higher than $\bigO{p^{d+1}}$ cost for other
    translations, assuming a derivative of any given order can be
    computed at amortized $\bigO{p}$ cost.
  \item
    Given a symbolic form of an elliptic constant-coefficient
    partial differential equation satisfied by the potential,
    our method requires $\bigO{p^{d-1}}$ terms
    in the expansion. It attains $\bigO{p^{d-1}\log p}$ cost of a
    multipole-to-local translation
    (with $\bigO{p^{d}}$ offline precomputation cost),
    and no higher than $\bigO{p^{d}}$ cost for other
    translations, again assuming a derivative of any given order can be
    computed at amortized $\bigO{p}$ cost.
  \item
    Given the PDE \emph{and} an evaluation formula for the derivatives of
    the kernel that yields $\bigO{p^{d-1}}$ specific derivatives
    at $\bigO{p^{d-1}}$ cost, our method requires $\bigO{p^{d-1}}$ terms
    in the expansion. It attains $\bigO{p^{d-1}\log p}$ cost of a
    multipole-to-local translation
    (with $\bigO{p^{d-1}}$ offline precomputation cost),
    and no higher than $\bigO{p^{d}}$ cost for other
    translations.
\end{itemize}
The final form of our method is competitive with methods having
state-of-the-art complexity (e.g.~\citep{greengard1988rapid}) outside of
multipole-to-multipole and local-to-local translations requiring
$\bigO{p^d}$ operations. Notably, neither of the two represents a
dominant cost in typical applications of the FMM. In the interest of
brevity, we use the term `translation' so as to include initial
formation and final evaluation of multipole and local expansions in
the above description.

For PDEs with only terms of equal order, we prove that our approach
introduces no additional error relative to conventional Cartesian
Taylor expansions. For PDEs with terms of varying orders, the difference
obeys the same asymptotic bound as Taylor truncation at fixed $p$, under
mild assumptions.
We additionally provide numerical verification supporting this claim. Finally,
we apply our methods to solve a boundary value problem for the
biharmonic equation and verify the accuracy.

Many claims in this paper offer formal verification via the Lean
proof assistant~\citep{deMouralean2021}. Formally verified results
show the corresponding Lean identifier, viewable in the 
companion repository at \url{https://github.com/isuruf/pde-translations-paper-code}.
These links are also clickable and lead to the corresponding Lean code.
Each link represents a claim by the authors that the
formalization appropriately reflects the mathematical meaning
of the natural-language statement. For readers inexperienced
with Lean, we recommend the use of an editor with appropriate
language support, allowing the use of `go-to-definition', which
permits efficient examination of related code relied on by
each statement.

\section{Preliminaries}

In this section, we introduce the fundamental objects under
consideration, namely multipole and local expansions based
on Taylor series as well as their respective translation
operators. These serve as the point of departure for our
compressed and cost-optimized methods.

Let $\b x, \b y \in \mathbb R^d$ be two points. We will refer to these
as the \emph{target point} and the \emph{source point}, respectively.
Throughout, the norm of a spatial vector $\b z\in\mathbb R^d$ is the
Euclidean $2$-norm. For $p\in\mathbb N_0$, let
\[
  \mathcal M(p):=\left\{\b q=(q_1,\dots,q_d) \in \mathbb N_0^d: \sum_{i=1}^d q_i
    \le p\right \}
\]
be the set of all $d$-dimensional multi-indices with total order
less than or equal to $p$, and let $N(p):=|\mathcal M(p)|$.
Observe that
\begin{equation}
  N(p)=\binom {p+d} d=\bigO{p^d}.
  \label{eq:np-count}
\end{equation}

We use the multi-index notation where for $\b q = (q_1, q_2, \ldots, q_d)$
and $\b x = (x_1, x_2, \ldots, x_d)^T$, we have
\begin{align*}
  \b q! &:= q_1! q_2! \ldots q_d !, \\
  \b x^{\b q} &:= x_1^{q_1} x_2^{q_2} \ldots x_d ^ {q_d}, \\
  \D x^{\b q} &:= \frac{\partial^{q_1}}{ \partial x_1^{q_1}}\frac{\partial^{q_2}}{ \partial x_2^{q_2}}
                 \ldots \frac{\partial^{q_d}}{ \partial x_d^{q_d}}, \\
  \b x \ge \b q &:\Leftrightarrow x_1 \ge q_1 \land x_2 \ge q_2\land \cdots\land x_d \ge q_d, \\
  |\b q| &:= q_1 + q_2 + \ldots + q_d, \\
  \binom{\b x}{\b q} &:= \binom{x_1}{q_1} \binom{x_2}{q_2} \ldots
    \binom{x_d}{q_d}.
\end{align*}

Let $G:\mathbb R^d\setminus\{\b 0\}\to\mathbb C$ be
a (translation-invariant and real-analytic) kernel as in
\eqref{eq:particle-convolution}.
Note that we do not yet impose the restriction that $G$ satisfy a PDE
away from the origin.
We assume that there exists $p'\in\mathbb R$, depending only on $G$, such
that, for every $m\in\mathbb N_0$, there are $s_m\in\mathbb N_0$ and
$C_{G,m}>0$ for which
\begin{equation}
  \left|
    \sum_{|\b q|=k}
    \frac{\D z^{\b q+\b r}G(\b z)}{\b q!}\b h^{\b q}
  \right|
  \leq
  \frac{C_{G,m}(1+k)^{s_m}}{\norm{\b z}^{|\b r|+p'}}
  \left(\frac{\norm{\b h}}{\norm{\b z}}\right)^k,
  \label{eq:assumption_g}
\end{equation}
\lean{PDETranslations.SatisfiesDerivativeEstimate}
for every $\b h \in \mathbb R^d$, $\b r\in\mathcal M(m)$ and $k\in\mathbb N$.
The constants $s_m$ and $C_{G,m}$ are independent of $k$, $\b z$, and $\b h$.
Thus, \eqref{eq:assumption_g} controls homogeneous Taylor terms, rather
than individual derivatives; the polynomial factor permits moderate growth
of these terms with their order.  One can show that this form accommodates scale-homogeneous
kernels such as the Laplace kernel and the components of the Stokeslet
(for the 3D Laplace kernel:
\lean{PDETranslations.LaplaceEstimate.laplaceKernel_satisfiesDerivativeEstimate},
with Legendre generating functions~\dlmfref{14.7.19} 
and the Poisson formula~\cite[Thm.~15 in Sec.~2.2.4]{evansPartialDifferentialEquations1998}
as assumptions).

The exponent $p'$ may be negative. For instance, the 3D biharmonic kernel needs $p'=-1$. For the 2D biharmonic kernel, the estimate
in~\eqref{eq:assumption_g} holds with $p'=-2$ only when $k+|\b r|\geq 3$,
due to the lower-order logarithmic terms \lean{PDETranslations.Biharmonic2DEstimate.biharmonic_derivative_estimate}.
For a non-scale-invariant kernel, the estimate may hold only in a restricted spatial regime.
In particular, for the Helmholtz kernel a bound on the dimensionless quantity
$\kappa\norm{\b z}$ is required; the global form of~\eqref{eq:assumption_g}
does not apply without such a restriction.

We let
\[
  \greenf(\b x, \b y) = G(\b x - \b y)
\]
for now to more clearly separate the dependencies on source
and target variables.
Using a Taylor series for (local) expansion of $\greenf$ around a
center $\b c$ from sources $\b y$ outside an expansion radius
$R$ ($\norm{\b y - \b c} > R$),
\begin{align}
  \greenf(\b x, \b y) &= \sum_{\b q \in \mathcal{M}(\infty)
    } \frac{\D x^{\b q} \greenf(\b x, \b y)
    |_{\b x = \b c} }{\b q!} (\b x - \b c)^{\b q}\nonumber \\ 
  &= \sum_{\b q \in \mathcal{M}(\infty)
    } \underbrace{\frac{\D z^{\b q} G (\b z)|_{\b z = \b c -  \b y} }{\b q!}}_{
      \substack{\text{depends on}\\ \text{source/center}}} 
      \underbrace{(\b x - \b c)^{\b q}}_{
        \substack{\text{depends on}\\ \text{center/target}}} \label{eq:local}. 
\end{align}
Then, truncating the series to order $p\in\mathbb N_0$, we have for
$\norm{\b x - \b c} < R < \norm{\b y-\b c}$,
\begin{equation}
\label{eq:local_truncated}
  \left|\greenf(\b x, \b y) - \sum_{\b q \in \mathcal{M}(p)
    } \frac{\D z^{\b q} G (\b z)|_{\b z = \b c -  \b y} }{\b q!}
    (\b x - \b c)^{\b q}\right| \le
    \frac{s_0!(p+2)^{s_0} C_{G,0}}{\|\b y - \b c\|^{p'}}
    \frac{(R/\|\b y - \b c\|)^{p+1}}
      {(1-R/\|\b y - \b c\|)^{s_0+1}}.
\end{equation}
\lean{PDETranslations.local_truncation_bound}
Indeed, apply~\eqref{eq:assumption_g} with $m=0$, $\b r=\b 0$,
$\b z=\b c-\b y$, and $\b h=\b x-\b c$.
Thus, $s_0$ and $C_{G,0}$ are independent of the truncation order $p$.  Setting
$t=R/\norm{\b y-\b c}<1$, the Taylor tail is bounded by
\[
  \frac{C_{G,0}}{\norm{\b y-\b c}^{p'}}
  \sum_{k=p+1}^{\infty}(1+k)^{s_0} t^k
  \leq
  \frac{s_0!(p+2)^{s_0} C_{G,0}}{\norm{\b y-\b c}^{p'}}
  \frac{t^{p+1}}{(1-t)^{s_0+1}}.
\]
Here we used $(p+2+j)^{s_0}\leq(p+2)^{s_0}(1+j)^{s_0}$ and
$\sum_{j=0}^{\infty}(1+j)^{s_0}t^j\leq s_0!/(1-t)^{s_0+1}$.
This proves~\eqref{eq:local_truncated}.
\citet[Section 3.5, (44)]{shanker2007accelerated} prove the
corresponding estimate for the kernel $\norm{\b x-\b y}^{-\nu}$ ($\nu > 0$).
Given a user-specified error tolerance $\epsilon$, we choose an order
such that the remainder term in the above expansion is smaller than the
tolerance whenever $\norm{\b x - \b c}<R$.
In reference to a local expansion,
we consider $(\D z^{\b q} G(\b z)|_{\b z = \b c - \b y}/\b q!)_{\b
q\in\mathcal M(p)}$ 
the \emph{coefficients} of the expansion, and the multivariate monomials
$((\b x-\b c)^{\b q})_{\b q\in\mathcal M(p)}$ the \emph{expansion basis}.
The factorial term in the expansion can be considered part of the
coefficient or the expansion basis. We choose to consider it
part of the coefficient.
Likewise, for a multipole expansion of $\greenf$ around a
center $\b c$ from sources $\b y$ inside an expansion ball with radius
$R$ ($\norm{\b y - \b c} < R$), we have
\begin{align}
\label{eq:mpole}
  \greenf(\b x, \b y) &= \sum_{\b q \in \mathcal{M}(\infty)
    } \D y^{\b q} \greenf(\b x, \b y) |_{\b y = \b c}
    \frac{(\b y - \b c)^{\b q}}{\b q!} \nonumber \\
    &= \sum_{\b q \in \mathcal{M}(\infty)} 
    \D z^{\b q}  G(\b z) |_{\b z = \b x - \b c}
    \frac{ (-1)^{|\b q|} (\b y - \b c)^{\b q}}{\b q!} \nonumber \\    
      &= \sum_{\b q \in \mathcal{M}(\infty)} 
    \underbrace{\D z^{\b q}  G(\b z) |_{\b z = \b x - \b c}}_{
      \substack{\text{depends on} \\ \text{center/target}}}
    \underbrace{\frac{(\b c - \b y)^{\b q}}{\b q!}}_{
      \substack{\text{depends on} \\ \text{source/center}}}  
\end{align}
Truncating the series to order $p$, we have for
$\norm{\b y-\b c}<R<\norm{\b x - \b c}$,
\begin{equation}
  \label{eq:mpole_truncated}
  \left|\greenf(\b x, \b y) - \sum_{\b q \in \mathcal{M}(p)
    } \D z^{\b q} G(\b z) |_{\b z = \b x - \b c}
    \frac{(\b c - \b y)^{\b q}}{\b q!}\right|
    \leq \frac{s_0!(p+2)^{s_0} C_{G,0}}{\|\b x - \b c\|^{p'}}
    \frac{(R/\|\b x - \b c\|)^{p+1}}
      {(1-R/\|\b x - \b c\|)^{s_0+1}}.
\end{equation}
\lean{PDETranslations.multipole_truncation_bound}
Indeed, apply~\eqref{eq:assumption_g} with $m=0$, $\b r=\b 0$,
$\b z=\b x-\b c$, and $\b h=\b c-\b y$, using the same
constants $s_0$ and $C_{G,0}$ as in~\eqref{eq:local_truncated}.  Setting
$t=R/\norm{\b x-\b c}<1$, the same geometric-tail estimate as above gives
\[
  \frac{C_{G,0}}{\norm{\b x-\b c}^{p'}}
  \sum_{k=p+1}^{\infty}(1+k)^{s_0} t^k
  \leq
  \frac{s_0!(p+2)^{s_0} C_{G,0}}{\norm{\b x-\b c}^{p'}}
  \frac{t^{p+1}}{(1-t)^{s_0+1}},
\]
which proves~\eqref{eq:mpole_truncated}.
\citet[Section 3.5, (42)]{shanker2007accelerated} prove the corresponding
estimate for the kernel $\norm{\b x-\b y}^{-\nu}$ ($\nu > 0$).
Given a user-specified error tolerance $\epsilon$, we once again choose an order $p$ such that the
remainder term in the above expansion is smaller than the tolerance whenever
$\norm{\b x-\b c}>R$.
In reference to a multipole expansion, we consider
the multivariate monomials $((\b c - \b y)^{\b q}/\b q!)_{\b q\in\mathcal M(p)}$ the
\emph{coefficients} of the expansions, and the derivatives
$(\D z^{\b q} G(\b z)|_{\b z = \b x - \b c})_{\b q\in\mathcal M(p)}$ the \emph{expansion basis}.

In both cases, the $d$-dimensional Taylor series are truncated at order $p$,
and therefore there are $\bigO{p^d}$ terms in the sum.

To obtain a \emph{multipole-to-local translation operator}
given the coefficients for a multipole expansion
$\alpham_{\b q} = (\b c_1 - \b y)^{\b q}/\b q!$ at a
given center $\b c_1$, we can calculate the coefficients of a local expansion
at a different center $\b c_2$ by formally inserting multipole expansion
\eqref{eq:mpole} into the local expansion \eqref{eq:local}
and matching terms, i.e.
\[
  \frac{\D z^{\b r} G(\b z)|_{\b z = \b c_2 - \b y}}{\b r!} \approx
    \sum_{|\b q| \leqslant p} \D x^{\b r}
    ( \D z^{\b q} G(\b z)|_{\b z = \b x - \b c_1}
    )|_{\b x = \b c_2}
    \frac{\alpham_{\b q}}{\b r !}.
\]
Since there are $\bigO{p^d}$ terms in the expansion around a given center,
and there are the same number of coefficients in the translated expansion,
this multipole-to-local translation operator requires $\bigO{p^{2d}}$ work
when carried out naively.

To obtain a \emph{multipole-to-multipole translation operator}
given the multipole expansion coefficients
$\alpham_{\b q} = (\b c_1 - \b y)^{\b q}/\b q!$ at center $\b c_1$, 
we can calculate the multipole expansion coefficients
at center $\b c_2$ using the multi-index binomial theorem.
\[
  \frac{(\b c_2- \b y)^{\b r}}{\b r!}
    = \frac{1}{\b r!}\sum_{\b q \leqslant \b r}
      \binom{\b r}{\b q}
      \underbrace{\b q! \alpham_{\b q} }_{(\b c_1 - \b y)^{\b q}}
      (\b c_2 - \b c_1)^{\b r-\b q} \\
    = \sum_{\b q \leqslant \b r} \alpham_{\b q} \frac{(\b c_2 - \b c_1)^{\b r-\b q}}{(\b r - \b q)!}.
\]
This multipole-to-multipole translation operator also requires $\bigO{p^{2d}}$
work when carried out naively.

To obtain a \emph{local-to-local translation operator} given the local expansion
coefficients $\alphal_{\b q} = \frac{\D z^{\b q} G(\b z)|_{\b z = \b c_1 - \b
y}}{\b q!} $ at $\b c_1$, we can calculate the (translated) local expansion
coefficients 
at center $\b c_2$ by forming the Taylor expansion of $(\b x -\b c_1)^{\b q}$ 
around $\b c_2$ and matching terms, i.e.
\[
  \frac{\D z^{\b r} G(\b z)|_{\b z = \b c_2 - \b y}}{\b r!} \approx \frac{1}{\b r!}\sum_{| \b q | \leqslant p
    } \alphal_{\b q}
    \D x^{\b r} (\b x - \b c_1)^{\b q} |_{\b x = \b c_2}.
\]
This local-to-local operator also requires $\bigO{p^{2d}}$ work
when carried out naively.

Formation of a multipole expansion following \eqref{eq:mpole}
and evaluation of a local expansion following \eqref{eq:local} both
require only $\bigO{p^d}$ work.
If $\D z^{\b q} G(\b z)$ for $\b q\in \mathcal M(p)$ is available
in amortized $\bigO{p}$ time, evaluation of a multipole expansion and
formation of a local expansion (both requiring these values) require
$\bigO{p^{d+1}}$ work.
If, however, cheaper derivatives that cost amortized $\bigO{1}$ time
are available,
then the evaluation of a multipole expansion and formation of a local
expansion also require only $\bigO{p^d}$ work.

The storage and operation count required for
expansion and evaluation in the Taylor setting is higher by a factor
of $\bigO{p}$ than for, e.g., conventional expansions based on
spherical harmonics.
When carried out naively, translation operators in Taylor form
require more work by a factor of $\bigO{p^2}$. This
difference results from the use of the PDE constraint (or lack
thereof).

\section{Algebraic Compression of Multipole and Local Expansions}

When the potential satisfies a constant-coefficient linear partial differential
equation, the derivatives of the potential satisfy the same PDE. Exploiting the
resulting linear relationships among the derivatives $(\D z^{\b q} G)_{\b
q\in\mathcal M(p)}$ can reduce the amount of storage and computation needed for
formation, translation, and evaluation of multipole and local expansions.

In this section, we derive an automated symbolic procedure to compress local and
multipole expansions based on Taylor series with the help of an assumption that
the potential satisfies a linear, constant-coefficient PDE. The resulting
expansions have $\bigO{p^{d-1}}$ terms (compared to $\bigO{p^d}$ terms in
\eqref{eq:local} and \eqref{eq:mpole}). We begin by encoding the relationships
between derivatives using the language of linear algebra in
Section~\ref{sec:la_deriv}.

\subsection{Linear-Algebraic Relationships Among Derivatives}
\label{sec:la_deriv}

Let $\mathcal{L}$ be a constant-coefficient, linear, scalar $d$-dimensional
partial differential operator, with $d\ge 1$,
\begin{equation}
  \label{eq:define_pde_op}
  \mathcal{L} = \sum_{\b m \in \mathcal{M}(c)} a_{\b m} \frac{\partial^{\b m}}{\partial \b x^{\b m}},
\end{equation}
\lean{PDETranslations.linearPDEOperator}
where $a_{\b m} \in \mathbb{C}$ are coefficients and $c \in \mathbb N$ is the
order of the PDE.

Since $c$ denotes the order, at least one nonzero coefficient has $|\b m|=c$. 
\lean{PDETranslations.LinearPDEOperator}

Let $p\ge c$ be the expansion order, and let
$G:\mathbb R^d \setminus\{\b 0\} \to \mathbb C$ be of class $C^\infty$ on its
domain and satisfy $\mathcal L G(\b x)=0$ there. This regularity allows us to
differentiate the PDE up to order $p-c$ and commute the mixed derivatives.

We define a `graded monomial ordering' for multi-indices in
Definition~\ref{def:bnu} using
the definition in \citep{coxIdealsVarietiesAlgorithms2007}.
\begin{definition}
  \lean{PDETranslations.GradedMonomialOrdering}
  A \emph{graded monomial ordering} $\bnu$ is an invertible mapping
  $\bnu:\mathbb N \to \mathbb N_0^d$ that enumerates the $d$-multi-indices in
  such a way that
  \begin{enumerate}
    \item $\bnu(1) = (0,\dots,0)\in\mathbb N_0^d$,
    \item for $i, j \in \mathbb N$ we have
    \begin{equation}
    |\bnu(i)| < |\bnu(j)| \implies i < j
    ,\label{eq:grading}\end{equation}
    and
    \item for $i, j,k \in \mathbb N$ we have
      \begin{equation}
        i < j \implies
        \bnu^{-1}(\bnu(i) + \bnu(k)) < \bnu^{-1}(\bnu(j) + \bnu(k)).
        \label{itm:numbering-ineq}
      \end{equation}
  \end{enumerate}
  \label{def:bnu}
\end{definition}
\begin{remark}
  \lean{PDETranslations.initial_segment_eq_multiindicesUpTo}
  Definition~\ref{def:bnu} implies that each graded monomial ordering induces a
  numbering of the elements of $\mathcal M(p)$, i.e. $\bnu(\{1,\dots, N(p)\})
  = \mathcal M(p)$.
\end{remark}
\begin{remark}
  \lean{PDETranslations.add_lt_add_iff}
  \eqref{itm:numbering-ineq} is equivalent to
  \begin{equation}
    i < j \Leftrightarrow
    \bnu^{-1}(\bnu(i) + \bnu(k)) < \bnu^{-1}(\bnu(j) + \bnu(k)).
    \label{eq:numbering-ineq-iff}
  \end{equation}
\end{remark}
\begin{remark}
  \lean{PDETranslations.magnitude_monotone}
  The contrapositive of \eqref{eq:grading} yields
  \begin{equation}
    i \le j \implies |\bnu(i)| \le |\bnu(j)|.
    \label{eq:bnu-order-monotone}
  \end{equation}
\end{remark}

For the remainder of this article, $\bnu$ will always be a graded monomial
ordering. Let $\DD q $ be a vector of derivative operators with order less than
or equal to $q$:
\begin{equation}
\label{eq:define_v}
\DD q =[\D x^{\bnu(i)}]_{i=1}^{N(q)}.
\end{equation}
\lean{PDETranslations.derivativeVector}
$\mathcal L G=0$ will result in linear relationships between elements of $\DD p
G$. In the following, we will construct a \emph{decompression matrix} $M\in
\mathbb C^{N(p) \times (N(p) - N(p-c))}$ such that
\begin{equation}
\label{eq:define_m}
  \DD p G= M [\DD p G]_{\b j},
\end{equation}
where $\b j \subset \{1,\dots, N(p)\}$ and $|\boldsymbol{j}| = N(p) - N(p-c)$.
One interpretation of \eqref{eq:define_m} is that all the elements of $\DD p G$
can be reconstructed from just the subset given by $\b j$.

Since $\mathcal L G(\b x)=0$ for $\b x\ne 0$, $\DD{p-c} \mathcal{L} G(\b x) = 0$
for $\boldsymbol{x} \neq 0$. In the following, we use this nullspace relation to
construct a basis of the coefficient space from which all values can be
reconstructed by way of a \emph{decompression operator $M$}.

Let $i\in\{1,\dots, N(p-c)\}$ and $\b n = \bnu(i)\in\mathcal M(p-c)$. Then
\begin{align}
  \label{eq:define_b}
  \left[\DD{p-c} \mathcal L G(\b x)\right]_i
  &= \D x^{\b n} \mathcal L G(\b x) \nonumber \\
  &=  \D x^{\b n} \left(  \sum_{\b \xi \in \mathcal{M}(c)} a_{\b \xi} \D x^{\b \xi} G(\b x) \right ) \\
  &=  \sum_{\b \xi \in \mathcal{M}(c)} a_{\b \xi} \D x^{\b \xi + \b n} G(\b x) \nonumber \\
  &= \sum_{\b m \in \mathcal{M}(p)} b_{\b n, \b m} \D x^{\b m} G(\b x),
\end{align}
with
\begin{equation}
  b_{\b n, \b m} = \begin{cases}
    a_{\b m - \b n} & \b m \ge \b n \text{ and } \b m - \b n \in \mathcal{M}(c), \\
    0 &\text{otherwise}\end{cases}
    \quad
    (\b m \in \mathcal M(p)).
  \label{eq:coeff-shift}
\end{equation}
\lean{PDETranslations.shiftedPDECoefficient}
Next, define a `PDE coefficients' matrix $P \in \mathbb C^{N(p-c)\times N(p)}$
such that
\begin{equation}
\label{eq:define_P}
  P_{i_1, i_2} = b_{\bnu(i_1), \bnu(i_2)}\qquad(i_1\in \{1,\dots,N(p-c)\}, i_2 \in \{1,\dots,N(p)\})
\end{equation}
\lean{PDETranslations.pdeCoefficientMatrix}
and observe that
\begin{equation}
  \label{eq:Pv}
  P \DD p G(\b x) = \b 0\qquad (\b x\ne\b 0).
\end{equation}
In a sense, each subsequent row of $P$ contains a ``shifted'' version of the
coefficients in the first row, cf. \eqref{eq:coeff-shift}.

Every row of $P$ has a nonzero entry: if $a_{\b\xi}\ne 0$ with
$\b\xi\in\mathcal M(c)$, then for $\b n\in\mathcal M(p-c)$ we have
$\b\xi+\b n\in\mathcal M(p)$ and
$b_{\b n,\b\xi+\b n}=a_{\b\xi}\ne 0$.
\lean{PDETranslations.lastNonzeroColumn_spec}
Thus the following maximum exists.
Let $h(i): \{1,\dots,N(p-c)\} \to \{1,\dots,N(p)\}$ be the index of the last
nonzero column in the $i$th row of $P$, i.e.
\begin{equation}
  \label{eq:define_h}
  h(i) = \max \{l \in \{1,\dots,N(p)\}: b_{\bnu(i), \bnu(l)} \neq 0\},
\end{equation}
\lean{PDETranslations.lastNonzeroColumn}
and let
\begin{equation}
  \label{eq:define_bar_j}
  \jcomp =(h(1), \dots, h(N(p-c)))^T.
\end{equation}
\lean{PDETranslations.pivotColumn}
Here and in the following, we abuse notation slightly by using index vectors
like $\jcomp$ as sets, with the obvious interpretation. Next, we argue that
$P_{[:, \jcomp]}$ is an invertible (particularly, square) submatrix of $P$
in Lemmas~\ref{lemma:h_eq}, \ref{lemma:h_increasing}
and~\ref{lemma:lower_triangular}. We next use this invertibility to recover the
full coefficient vector $\DD{p} G(\b x)$ from its subset $[\DD{p} G(\b x)]_{\b
j}$ using the fact that $\DD{p} G(\b x)$ is in the nullspace of $P$.

\begin{lemma}
\label{lemma:h_eq}
  \lean{PDETranslations.lastNonzeroColumn_shift}
  For $i \in \{1, 2, \ldots, N(p-c)\}$,
  \(
    \bnu(h(i)) = \bnu(h(1)) + \bnu(i)
  \).
\end{lemma}
\begin{proof}
Let $i \in \{1, 2, \ldots, N(p-c)\}$. Using the definition of $b$
in~\eqref{eq:coeff-shift}, we see that
\begin{equation}
  h(i) = \max \{l : a_{\bnu(l) - \bnu(i)} \neq 0,
            l \in \{1,\dots,N(p)\}, \bnu(l) \ge \bnu(i), \bnu(l) - \bnu(i) \in \mathcal{M}(c) \},
  \label{eq:h_i}
\end{equation}
and
\begin{equation}
  h(1) = \max \{l : a_{\bnu(l)} \neq 0, l \in \{1,\dots,N(p)\}, \bnu(l) \in \mathcal{M}(c) \}.
  \label{eq:h_1}
\end{equation}
Let
\[
  z = \bnu^{-1}(\bnu(h(1)) + \bnu(i)).
\]
Since $|\boldsymbol{\nu}(i)|  \le p-c$ and $|\bnu(h(1))| \le c$, $z$ is
well-defined and $|\bnu(z)| \le p$. Using \eqref{eq:h_1}, $a_{\bnu(z)-\bnu(i)} =
a_{\bnu(h(1))} \ne 0$. Therefore, using \eqref{eq:h_i}, we have $h(i) \ge z$.

In order to prove the lemma, we need to prove that $h(i) = z$. To arrive at a
contradiction, assume $l = h(i) > z$. By \eqref{eq:h_i}, we may find an $m \in
\{1, \ldots, N(c)\}$ such that $\bnu(m) = \bnu(l) - \bnu(i) \in \mathcal{M}(c)$.
Then, using \eqref{eq:numbering-ineq-iff}, we have
\[
  m = \bnu^{-1}(\bnu(l)-\bnu(i)) > \bnu^{-1}(\bnu(z) - \bnu(i)) = h(1)
\] and $a_{\bnu(m)} = a_{\bnu(l) - \bnu(i)} \neq 0$.
This contradicts the maximality of $h(1)$.
\end{proof}

\begin{lemma}
\label{lemma:h_increasing}
  \lean{PDETranslations.lastNonzeroColumn_strictMono}
  $h$ is a strictly increasing function, i.e. for $r_1, r_2 \in \{1, 2, \dots,
  N(p-c)\}$ with $r_1 < r_2$, we have
  \[
    h(r_1) < h(r_2).
  \]
\end{lemma}
\begin{proof}
Using Definition~\ref{def:bnu},
\[
   \bnu^{-1}(\bnu(r_1) + \bnu(k)) < \bnu^{-1}(\bnu(r_2) + \bnu(k))
\] for $k \in \{1, \ldots, N(p)\}$ since $r_1 < r_2$.
Replacing $k$ by $h(1)$ and using Lemma~\ref{lemma:h_eq} we have,
\begin{align*}
  h(r_1) &= \bnu^{-1}(\bnu(h(r_1))) \\
     &= \bnu^{-1}(\bnu(h(1)) + \bnu(r_1)) \\
     &< \bnu^{-1}(\bnu(h(1)) + \bnu(r_2)) \\
     &= \bnu^{-1}(\bnu(h(r_2))) \\
     &= h(r_2).
\end{align*}
\end{proof}
In the remaining sections, we will use Matlab-style indexing notation, e.g.,
$T_{[:k,\boldsymbol{j}]}$, to denote the extraction of the first $k$ rows from
the columns of a matrix $T$ indexed by $\boldsymbol{j}$. In particular,
a `slice' $j:k$ means indices \emph{inclusive} of $j$ and $k$.
\begin{lemma}
  \label{lemma:lower_triangular}
  \lean{PDETranslations.pivotMatrix_lower_triangular}
  $P_{[:, \jcomp]}$ is an invertible lower triangular matrix.
\end{lemma}
\begin{proof}
  This follows from Lemma~\ref{lemma:h_increasing} because, in the $i$th row of
  $P$, the $h(i)$th column is the last nonzero entry. Therefore, $P_{[:, \jcomp]}$ is lower triangular with nonzero diagonal entries.
\end{proof}
Define $\b j :=\{j \in \{1, ... , N(p)\}: j \not \in \jcomp\}$ to be an
(unconstrained) numbering of the complement of $\jcomp$.
\lean{PDETranslations.FreeColumn}
\begin{theorem}
  \lean{PDETranslations.existsUnique_decompressionMatrix}
  There exists a unique \emph{decompression operator} $M$ such that for any
  vector $\b s \in \mathbb{C}^{N(p)}$,
  \[
    P \b s = \b 0 \iff \b s = M \b s_{\b j}.
  \]
  \label{thm:define_M}
\end{theorem}
\begin{proof}
Using Lemma~\ref{lemma:lower_triangular},
\begin{align*}
  P \b s &= 0 \\
  \Leftrightarrow \begin{bmatrix}P_{[:, \b j]} & P_{[:, \jcomp]} \end{bmatrix}
  \begin{bmatrix}\b s_{\b j} \\ \b s_{\jcomp} \end{bmatrix} &= 0 \\
  \Leftrightarrow P_{[:, \b j]} \b s_{\b j} + P_{[:, \jcomp]} \b s_{\jcomp}  &= 0 \\
  \Leftrightarrow \b s_{\jcomp} &= -(P_{[:, \jcomp]})^{-1} P_{[:, \b j]} \b s_{\b j} \\
  \Leftrightarrow \begin{bmatrix}\b s_{\b j} \\ \b s_{\jcomp} \end{bmatrix} &= \
    \underbrace{
        \begin{bmatrix} I \\ -(P_{[:, \jcomp]})^{-1} P_{[:, \b j]} \end{bmatrix}
    }_{M':=}
    \b s_{\b j}.
\end{align*}
As a row permutation of the matrix $M'$ above, we define the \emph{decompression
operator} $M \in \mathbb{C}^{N(p) \times |\boldsymbol{j}|}$ as
\begin{align}
  M_{[\b j, :]} &= I \label{eq:m_1} \\
  M_{[\jcomp, :]} &= -(P_{[:, \jcomp]})^{-1} P_{[:, \b j]} \label{eq:m_2}
\end{align}
\lean{PDETranslations.decompressionMatrix}
resulting in
\[
  P \b s = 0 \iff \b s = M \b s_{\b j}.
\]
For uniqueness, note that for every $\b t\in\mathbb C^{|\b j|}$, the vector
$\b s := M\b t$ satisfies $P\b s=\b 0$ and $\b s_{\b j}=\b t$. Hence, if $\tilde M$
also has the stated property, then $\tilde M\b t=\tilde M\b s_{\b j}=\b s=M\b t$
for all $\b t$, i.e., $\tilde M=M$.
\end{proof}

Using Theorem~\ref{thm:define_M} and the nullspace relation~\eqref{eq:Pv}, we
have that \begin{equation}
  \label{eq:decompression}
  \DD{p} G(\b x) = M [\DD{p} G(\b x)]_{\b j},
\end{equation}
corroborating the naming of the operator $M$.

Next, we establish some more properties of the decompression operator $M$.
\begin{theorem}
  \label{thm:mult_decompression}
  For a vector $\b s \in \mathbb{C}^{|\b j|}$, $M \b s$ can be computed in
  $\bigO{p^d}$ operations. Similarly, for a vector $\b s \in \mathbb{C}^{N(p)}$,
  $M^T \b s$ can be computed in $\bigO{p^d}$ operations.
\end{theorem}

\begin{proof}
  Multiplying $M$ by $\b s$ amounts to multiplication by the identity $I$ and
  $-(P_{[:, \jcomp]})^{-1} P_{[:, \b j]}$. The former is free, and, for the
  latter, let $k$ be the number of nonzero elements in the first row of $P$,
  where $k=\bigO{1}$ (in the expansion order) as it depends only on the PDE.

  Observe that each row in $P$ has the same number of nonzero elements as the
  first row. Therefore, the number of nonzero entries in $P$ is $k\cdot N(p-c) =
  \bigO{p^d}$. So a matrix-vector product with $P_{[:, \b j]}$ requires only
  $\bigO{p^d}$ operations. Since $P_{[:,\jcomp]}$ is lower triangular
  (Lemma~\ref{lemma:lower_triangular}) and also has $\bigO{p^d}$ nonzero
  entries, a matrix-vector product with $(P_{[:, \jcomp]})^{-1}$ can be
  carried out by forward substitution in $\bigO{p^d}$ operations. As a result, a
  matrix-vector product with $(P_{[:, \jcomp]})^{-1} P_{[:, \b j]}$ can be
  carried out in $\bigO{p^d}$ operations. The argument for $M^T$ is analogous.
\end{proof}

By the construction of $M$ in the proof of
Theorem~\ref{thm:define_M}, we immediately have the following
property.
\begin{prop} \label{prop:m_property}
 \lean{PDETranslations.decompressionMatrix_property}
 For $i \in \{1, 2, \ldots, N(p)\}$ and $k \in \{1,2, \dots  ,
 |\b j| \} $,
 \begin{enumerate}
 \item $M_{i, k} = \delta_{i, \b j_{k}}$ if $i \in \b j$, \label{M-property1}
 \item  $M_{i, k} = 0$ if $\b j_{k} > i$. \label{M-property2}
 \end{enumerate}
 \end{prop}
 \begin{proof}
  The first property follows directly from the definition of $M$. To prove the
	second property, it suffices to assume $i \in \jcomp.$
	Then, there exists $w \in \{1, \ldots , N(p-c)\}$  such that
	$\jcomp_{w} = h(w) = i$ and $P_{[w,i+1:]}=0$. Recall
	the definition of $M$ in the proof of Theorem~\ref{thm:define_M}, $M_{[\jcomp, :]} = - (P_{[:,
	\jcomp]})^{-1} P_{[:,\b j]}$. The
	$i^{\textup{th}}$ row of $M$ is
\begin{align*}
	M_{[i,:]} = M_{[\jcomp_{w},:]} =
              - [(P_{[:, \jcomp]})^{-1}]_{[w,:]} P_{[:,\b j]}
.\end{align*}
Let $k \in \{1, 2, \ldots, |\b j|\}$. As $P_{[:,
\jcomp]}$ is a lower triangular matrix, $( P_{[:,
\jcomp]})^{-1}$ is also lower triangular.
\begin{align*}
	M_{i,k} &= - [(P_{[:, \jcomp]})^{-1}]_{[w,:]} [P_{[:, \b j]}]_{[:, k]} \\
	& = - [(P_{[:, \jcomp]})^{-1}]_{[w,:w]} [P_{[:, \b j]}]_{[:w, k]}\\
	& = - [(P_{[:, \jcomp]})^{-1}]_{[w,:w]}P_{[:w, \b j_{k}]}
.\end{align*}
Since $h$ is (strictly) monotonically increasing and $P_{[:w, i+1:]} =
0$, $P_{[:w, \b j_{k}]} = 0$ if
$\b j_{k} > i$. Therefore, $M_{i,k} = 0$ if $\b j_{k} >
i$.
 \end{proof}

\begin{theorem}
  \label{thm:m_equal_order_property}
  \lean{PDETranslations.decompressionMatrix_equal_order_property}
  Let $i \in \{1, 2, \ldots, N(p)\}$ and $k \in \{1,2, \dots  ,
  |\b j| \} $. For PDEs with only terms of equal order \lean{PDETranslations.LinearPDEOperator.HasOnlyEqualOrder},
      $$|\bnu(i)| \neq |\bnu(\b j_{k})| \implies M_{i,k} = 0.$$
\end{theorem}
\begin{proof}
The case $i \in \b j$ follows from Proposition~\ref{prop:m_property}.
Now suppose $i \in \jcomp.$ If
$|\bnu( \b j_{k} ) | \neq |\bnu(i)|$, then either
\begin{enumerate}
	\item $|\bnu( \b j_{k} ) | > |\bnu(i)|$ (implying $\boldsymbol{j}_{k} > i$), or
	\item $|\bnu( \b j_{k} ) | < |\bnu(i)|$ (implying $\boldsymbol{j}_{k} < i$).
\end{enumerate}
The former case follows from Proposition~\ref{prop:m_property}.  In the
following, we assume $\b j_{k} < i$ and $|\bnu( \b j_{k} ) | < |\bnu(i)|$. As $i
\in \jcomp$, there exists $w \in \{1, \ldots , N(p-c)\}$ such that
$\jcomp_{w} = h(w) = i$. We will prove the cases of $w=1$ and $w>1$
separately. Before proceeding, we note that, as a direct consequence of
\eqref{eq:coeff-shift},
\begin{align}
  P_{s,s'} = 0 \text{ if } |\bnu(s)| + c \neq |\bnu(s')|,\quad (s \in \{1, \ldots, N(p-c)\}, s' \in \{ 1, \dots  , N(p)\})
  \label{eq:p-equal-order}
\end{align}
for PDEs with only terms of equal order.

\underline{Case 1: $w=1$.}
If $w=1$, then
\begin{align*}
 M_{i,k} & = - [(P_{[:, \jcomp]})^{-1}]_{[w,:]} [P_{[:, \b j]}]_{[:, k]} \\
 & = - [(P_{[:, \jcomp]})^{-1}]_{[1,1]} P_{[1, \b j_{k}]},
\end{align*}
as $(P_{[:, \jcomp]})^{-1}$ is a lower triangular matrix.
Since $\b j_{k} < i$ and $|\bnu(\b j_{k})| < |\bnu(i)| $, $|\bnu(1) +
\bnu(\b j_{k})| = |\bnu(\b j_{k})|
< |\bnu(i)| = |\bnu(h(1)) |  = c.$ By \eqref{eq:p-equal-order}, it follows that
$P_{[1,\b j_{k}]} = 0$ and  $M_{i,k} = 0$.

\underline{Case 2: $w>1$.}

First, observe that, for PDEs with only terms of equal order, for $s \in \{ 1, \ldots,
N (p - c) \}$,
\begin{equation}
 | \bnu (h (s)) | = | \bnu (s) | + c \label{eq:h-increment-eq-order}
,\end{equation}
i.e., the order of all
coefficients is $c \,$ higher than the base coefficient. This applies
specifically to $h (s)$, the highest-numbered coefficient.

Next, let $\ubw \in \{ 2, \ldots, w \}$ be such that
\[ | \bnu (h (\ubw)) | = | \bnu (h (w)) | = | \bnu (i) |,
\]
and
\begin{equation}
   | \bnu (\ubw - 1) | < | \bnu (\ubw) | \label{eq:ubw-assumption}
.\end{equation}
Such a $\ubw$ must exist because $w=1$ is the sole value of $w$ for
which a lower-order multi-index is not available.

Now let $s \in \{ \ubw, \ldots, N (p - c) \}$ and $\eta \in \{ 1,
\ldots, \ubw - 1 \}$. We will prove that $[P_{[:,
\jcomp]}]_{s, \eta} = 0$. There exists an $s'$ such that $h (\eta) = s'$, so that $\jcomp_{\eta} = h (\eta) = s'$. By \eqref{eq:p-equal-order}, we obtain
the claim if $c \neq | \bnu (s') | - | \bnu (s) |$.
\begin{align*}
  | \bnu (s') | - | \bnu (s) |
  & = | \bnu (h (\eta)) | - | \bnu (s) |
  \overset{\eqref{eq:h-increment-eq-order}}{=} | \bnu (\eta) | + c - | \bnu (s) |\\
  &
  \overset{\eqref{eq:bnu-order-monotone}}\le |\bnu (\ubw-1) | + c - | \bnu (s) |
  \overset{\eqref{eq:ubw-assumption}}<  | \bnu (\ubw) | + c - | \bnu (s) |
  \overset{\eqref{eq:bnu-order-monotone}}\le| \bnu (\ubw) | + c - | \bnu (\ubw) |
  = c
\end{align*}
shows that this is the case.

This implies that $P_{[:, \jcomp]}$ has the following shape:
\[
\begin{array}{c|cccccccc}
    & h(1) = \jcomp_{1}       &        & h(\ubw) = \jcomp_{\ubw} &        & \cdots & h(w) = \jcomp_{w} &        &        \\[1ex]
    & 1       &        & \ubw &        & \cdots & w &        &        \\ \hline
    & *      &        &        &        &        &        &        &        \\
    & *      & *      &        &        &        &        &        &        \\
\ubw & 0      & 0      & *      &        &        &        &        &        \\
    & 0      & 0      & *      & *      &        &        &        &        \\
    & \vdots & \vdots & \vdots & \vdots & \ddots &        &        &        \\
w   & 0      & 0      & *      & *      & \cdots & *      &        &        \\
    & 0      & 0      & *      & *      & \cdots & *      & *      &        \\
    & 0      & 0      & *      & *      & \cdots & *      & *      & *      \\
\end{array}
\]
As a result,
\[
	[(P_{[:, \jcomp]})^{-1}]_{[\ubw:,:\ubw - 1]} = 0
.\]
Finally, let $s \in \{ \ubw, \ldots, w \}$.
We will show that $[P_{[:, \b j]}]_{s,k}=P_{s,\b j_k}=0$.
By \eqref{eq:h-increment-eq-order}, $|\bnu(s)| + c =|\bnu(i)|.$
By assumption, $ |\bnu( \b j_{k} ) | < |\bnu(i)|$. Again we obtain the claim via \eqref{eq:p-equal-order}
if we can show that $c\ne |\bnu(\b j_k)|-|\bnu(s)|$. To this end, consider
\[
  | \bnu (\b j_k) | - | \bnu (s) |  = \underbrace{| \bnu (\b j_k) | - | \bnu (i) |}_{<0} + c < c.
\]

To complete the argument,
\begin{align*}
	M_{i,k} &= - [(P_{[:, \jcomp]})^{-1}]_{[w,:]} [P_{[:, \b j]}]_{[:, k]} \\
	& = - [(P_{[:, \jcomp]})^{-1}]_{[w,:w]} [P_{[:, \b j]}]_{[:w, k]}\\
	& = - \underbrace{[(P_{[:, \jcomp]})^{-1}]_{[w,:\ubw - 1]}}_0 [P_{[:, \b j]}]_{[:\ubw - 1, k]}
  - [(P_{[:, \jcomp]})^{-1}]_{[w,\ubw : w]} \underbrace{[P_{[:, \b j]}]_{[\ubw:w, k]}}_0\\
    &=0
.\end{align*}
\end{proof}

\begin{corollary}
\label{corollary:deriv_less_order}
For $i \in \{1, 2, \ldots, N(p)\}$,
\begin{align}
 \D z ^{\bnu(i)} G(\b z) = \sum_{\substack{s \in \{ 1, \dots  , |\b j| \} \\ |\bnu(\b j_s)| \leq |\bnu(i)|}} 
    M_{i, s} \D z^{\bnu(\b j_s)} G(\b z), \qquad \b z \in \mathbb{R}^d\setminus \{\b 0\}.\label{eq:deriv-representation}
\end{align}
\lean{PDETranslations.derivative_eq_sum_degree_le}
Particularly, for PDEs with only terms of equal order,
\begin{align}
 \D z ^{\bnu(i)} G(\b z) = \sum_{\substack{s \in \{ 1, \dots  , |\b j| \} \\ |\bnu(\b j_s)| = |\bnu(i)|}} 
    M_{i, s} \D z^{\bnu(\b j_s)} G(\b z), \qquad \b z \in \mathbb{R}^d\setminus \{\b 0\}. \label{eq:deriv-representation-equal}
\end{align}
\lean{PDETranslations.derivative_eq_sum_degree_eq}
\end{corollary}
\begin{proof}
By~\eqref{eq:Pv} and Theorem~\ref{thm:define_M}, each derivative is the
sum over all free coordinates with coefficients from its row of $M$.
If $|\bnu(\b j_s)|>|\bnu(i)|$, grading implies $\b j_s>i$, so
Proposition~\ref{prop:m_property} makes the corresponding summand zero.
This proves~\eqref{eq:deriv-representation}.
For equal-order PDEs, Theorem~\ref{thm:m_equal_order_property} also
eliminates every summand of a different degree, giving
\eqref{eq:deriv-representation-equal}.
\end{proof}

Using~\eqref{eq:decompression} for the storage of Taylor series coefficients
means that $|\b j|$ coefficients suffice to recover the $p$-th-order Taylor expansion
of a function $u$ satisfying $\mathcal L u=0$ in a neighborhood
of the expansion center. The
number of entries of $\b j$ satisfies
\[
  |\b j|
    = |\mathcal M(p) \setminus \bnu(\jcomp)| \\
    = N(p) - |\jcomp| \\
    = N(p) - N(p-c) \\
    \stackrel{\eqref{eq:np-count}}= \binom{p+d}{d} - \binom{p+d-c}{d} \\
    = \bigO{p^{d-1}},
\]
so that we only need to store $\bigO{p^{d-1}}$ coefficients to represent a local
Taylor expansion.
\begin{prop}
\label{prop:j-character}
\lean{PDETranslations.freeColumn_multiindices_eq}
Let $\b t = \bnu(h(1))$. Then,
\[
 \bnu(\b j) =  \{ \b m \in \mathcal M(p): 
 \exists i \in \{1, \dots, d\} \text{ such that } \b m_i < \b t_i \}. 
\]
\end{prop}
\begin{proof}
Recall from~\eqref{eq:define_bar_j} that
\[
\bnu(\jcomp) = \{ \bnu(h(i)): i \in \{ 1, \dots  , N(p-c)\} \}.  
\]
Denote $\b t = \bnu(h(1))$. Since the PDE has exact order $c$, some
nonzero coefficient has degree $c$; grading and the maximality of $h(1)$
therefore imply $|\b t|=c$.
By Lemma~\ref{lemma:h_eq},
\begin{align*}
  \bnu(\jcomp) &= \{ \b t + \bnu(i): i \in \{ 1, \dots  , N(p-c)\} \}\\ 
  &= \{ \b m \in \mathcal{M}(p):  \b m_i \geq \b t_i \text{ for all } i \in \{1, \dots, d\} \}.
\end{align*}
For the reverse inclusion in this equality, if $\b m\in\mathcal M(p)$
and $\b t\le\b m$ componentwise, then $\b n=\b m-\b t$ has
$|\b n|=|\b m|-c\le p-c$ and $\b m=\b t+\b n$.
Taking complements, failure of componentwise dominance is equivalent to
having at least one coordinate smaller than that of $\b t$.
As $\bnu(\b j) = \mathcal M(p) \setminus \bnu(\jcomp)$, we have  
\[\bnu(\b j) =  \{ \b m \in \mathcal M(p): \exists i \in \{1, \dots,
d\} \text{ such that } \b m_i < \b t_i \}\]
\end{proof}
Given a fixed graded monomial ordering, the vector $\b t = \bnu(h(1))$
depends only on the highest-order terms of the PDE and is independent of
the expansion order $p$.
We show the multi-index `footprint' of $\bnu(\b j)$ and $\bnu(\jcomp)$ for
several example PDEs in Figure~\ref{fig:j}, using a degree-lexicographic
ordering with $y$ varying slowest.
For instance, in two dimensions, the Laplace equation yields $\b t = (0,2)$.
Similarly, for the equation $\frac{\partial^{2} u}{\partial x \partial y} =0$,
we have $\b t = (1,1)$.

We describe the formation of the local expansions in more detail in
Section~\ref{sec:p2l} and their evaluation in Section~\ref{sec:l2p}. Since the
basis functions of the multipole expansion are the derivatives of the kernel,
they are amenable to compression via~\eqref{eq:decompression} as well. We
describe formation of a compressed multipole expansion in Section~\ref{sec:p2m}
and evaluation of the multipole expansion in Section~\ref{sec:m2p}.

\begin{figure}
\centering
\begin{subfigure}[t]{.28\linewidth}
\begin{tikzpicture}[scale=0.8]
\draw[draw opacity=0.25,->] (0,0)--(4.5,0) node[right]{$x$}; \draw[draw
opacity=0.25,->] (0,0)--(0,4.5) node[above]{$y$}; \foreach \i in
{0.25,0.5,...,4.0}{ \foreach \j in {0,0.25,...,\i} { \draw (\i-\j,\j)
node[fill=black!20,circle,scale=0.3] {}; } } \foreach \i in
{0,0.25,0.5,...,4.0}{ \draw (\i, 0) node[fill=black,circle,scale=0.3] {}; }
\foreach \i in {0,0.25,0.5,...,3.75}{ \draw (\i, 0.25)
node[fill=black,circle,scale=0.3] {}; } \draw[->,shorten >=2mm] (2.5,3)
node[right, text=black!50] {$\bnu(\jcomp)$} -- (1.5, 2.5); 
\draw[->,shorten >=2mm] (3.75,1.25)
node[above] {$\bnu(\b j)$} -- (3.75, 0.25);
\end{tikzpicture}
\caption{ Index space showing $\bnu(\b j)$ and $\bnu(\jcomp)$ for the Laplace equation in
two dimensions }
\label{fig:j_laplace}
\end{subfigure}\hspace{3mm}
\begin{subfigure}[t]{.28\linewidth}
\begin{tikzpicture}[scale=0.8]
\draw[draw opacity=0.25,->] (0,0)--(4.5,0) node[right]{$x$}; \draw[draw
opacity=0.25,->] (0,0)--(0,4.5) node[above]{$y$}; \foreach \i in
{0.25,0.5,...,4.0}{ \foreach \j in {0,0.25,...,\i} { \draw (\i-\j,\j)
node[fill=black!20,circle,scale=0.3] {}; } } \foreach \i in
{0,0.25,0.5,...,4.0}{ \draw (\i, 0) node[fill=black,circle,scale=0.3] {}; }
\foreach \i in {0,0.25,0.5,...,3.75}{ \draw (\i, 0.25)
node[fill=black,circle,scale=0.3] {}; } \foreach \i in {0,0.25,0.5,...,3.5}{
\draw (\i, 0.5) node[fill=black,circle,scale=0.3] {}; } \foreach \i in
{0,0.25,0.5,...,3.25}{ \draw (\i, 0.75) node[fill=black,circle,scale=0.3] {}; }
\draw[->,shorten >=2mm] (2.5,3) node[right, text=black!50] {$\bnu(\jcomp)$} -- (1.5, 2.5);
\draw[->,shorten >=2mm] (3.75,1.25) node[above] {$\bnu(\b j)$} -- (3.75, 0.25);
\end{tikzpicture}
\caption{ Index space showing $\bnu(\b j)$ and $\bnu(\jcomp)$ for the biharmonic equation in
two dimensions }
\label{fig:j_biharmonic}
\end{subfigure}\hspace{3mm}
\begin{subfigure}[t]{.28\linewidth}
\begin{tikzpicture}[scale=0.8]
\draw[draw opacity=0.25,->] (0,0)--(4.5,0) node[right]{$x$}; \draw[draw
opacity=0.25,->] (0,0)--(0,4.5) node[above]{$y$}; \foreach \i in
{0.25,0.5,...,4.0}{ \foreach \j in {0,0.25,...,\i} { \draw (\i-\j,\j)
node[fill=black!20,circle,scale=0.3] {}; } } \foreach \i in
{0,0.25,0.5,...,4.0}{ \draw (\i, 0) node[fill=black,circle,scale=0.3] {}; \draw
(0, \i) node[fill=black,circle,scale=0.3] {}; } \draw[->,shorten >=2mm] (2.5,3)
node[right, text=black!50] {$\bnu(\jcomp)$} -- (1.5, 2.5);
\draw[->,shorten >=2mm] (4.0,1.0)
node[above] {$\bnu(\b j)$} -- (4.0, 0.0);
\end{tikzpicture}
\caption{ Index space showing $\bnu(\b j)$ and $\bnu(\jcomp)$ for the PDE
$\frac{\partial^2 u}{\partial x \partial y}= 0$ }
\label{fig:j_weird}
\end{subfigure}
\caption{Index space showing $\bnu(\b j)$, $\bnu(\jcomp)$ for order 16 Taylor series
satisfying different two-dimensional partial differential equations}
\label{fig:j}
\end{figure}

\subsection{Compressed Local Expansion}
\label{sec:p2l}
Let $\b y \in \mathbb R^d$ be a source point,  $\b c\in \mathbb R^d$ an
expansion center, and $\b x\in\mathbb R^d$ a target point with $\norm{\b x-\b
c}<\norm{\b y-\b c}$. Recall from~\eqref{eq:local} that the $i$-th coefficient of
a local expansion is
\[
  \alphal_i = \frac{\D z^{\bnu(i)} G(\b z) |_{\b z = \b c - \b y}}{\bnu(i)!} =
    \frac{[\DD p G(\b c - \b y)]_{i}}{\bnu(i)!}.
\]
Using the decompression machinery from Section~\ref{sec:la_deriv},
\[
  \alphal_i = \frac{[\DD p G(\b c - \b y)]_{i}}{\bnu(i)!}
  = \frac{[M  [\DD p G(\b c - \b y)]_{\b j}]_{i}}{\bnu(i)!}
        = \frac{[M (\alphal_{\b j} \bnu(\b j)!)]_i}{\bnu(i)!},
\]
where $\alphal_{\b j} \bnu(\b j)!$ is understood as the element-wise
product of the vectors $\alphal_{\b j}$ and $\bnu(\b j)!$. 
\begin{definition}
\lean{PDETranslations.IsCompressibleLocalExpansion}
Let $\b s \in \mathbb{C}^{N(p)}$ be the coefficients of a local
expansion. We call a local expansion compressible if for $i \in \{1, 2,
\ldots N(p)\}$,
\begin{equation}
  \b s_i = \frac{[M (\b s_{\b j}  \bnu(\b j)!)]_i}{\bnu(i)!},
  \label{eq:compressed-local-form}
\end{equation}

where $\b s_{\b j} \bnu(\b j)!$ is understood as the element-wise product of the
vectors $\b s_{\b j}$ and $\bnu(\b j)!$.
\end{definition}
\noindent 
For compressible local expansion coefficients $\alphal$, it suffices to
compute and store $\alphal_{\b j}$ only. Under that condition, the expansion
can be recovered from that subset without any loss of information. 

Next, we consider the cost of computing a compressed local expansion at a center
$\b c$ from a single source point $\b y$.
\begin{definition}
  Let $w(p)$ be the amortized number of floating point operations
  per derivative when computing the subset $[\DD{p}G(\b x)]_{\b j}$.
  \label{def:wp}
\end{definition}
Given the quantity $w(p)$, forming a compressed local expansion requires
$\bigO{p^{d-1}w(p)}$ operations. Section~\ref{sec:calculate_derivatives} gives
procedures for the computation of derivatives and provides estimates of $w(p)$
for different kernels.

\subsection{Compressed Local Evaluation}
\label{sec:l2p}

In local expansion evaluation, the inner product of the full set of Taylor
coefficients and the monomials is calculated per target. Recall
from~\eqref{eq:local} that the local Taylor expansion for a target $\b x \in T$
is,
\[
  G(\b x - \b y)
  \approx
  \sum_{i \in \{1, 2, \ldots N(p)\}} \alphal_i (\b x - \b c)^{\bnu(i)},
\] where
\[
  \alphal_i = \frac{\left[M (\alphal_{\b j}  \bnu(\b j)!)\right]_{i}}{\bnu(i)!} \qquad{i \in \{1, 2, \ldots N(p)\}}
\] and $\alphal_{\b j}$ are stored coefficients.
Using Theorem~\ref{thm:mult_decompression}, calculating the full set of Taylor
coefficients $[\alphal_i]_{i=1}^{N(p)}$ from the stored coefficients
$\alphal_{\b j}$ requires $\bigO{p^d}$ operations. Calculating the monomials
$(\b x - \b c)^{\bnu(i)}$ requires $\bigO{p^d}$ operations, and the inner
product of the monomials and the coefficients requires the same amount of work.
Overall, evaluation of a compressed local expansion requires $\bigO{p^d}$
operations.

\subsection{Compressed Multipole Expansion}
\label{sec:p2m}
While compression and decompression of local expansions amounted to a
straightforward application of \eqref{eq:decompression}, compression and
decompression of multipole expansions requires additional machinery, which we
introduce in this section. We compress a multipole expansion to $\bigO{p^{d-1}}$
terms at a cost of $\bigO{p^d}$ operations.

Let $\alpham \in \mathbb{C}^{N(p)}$ be a vector of coefficients of a
multipole expansion.  From \eqref{eq:mpole}, we know that
\[
  \alpham_i = \frac{(\b c - \b
  y)^{\bnu(i)}}{\bnu(i)!},
\] where $i \in \{1, 2, \ldots, N(p)\}$.
Let $\DD p G(\b x - \b c)$ be a vector of derivatives. Let $M$ be the
decompression operator and $\b j$ be the vector of indices of $\DD p G(\b x)$
that are stored. Recall from \eqref{eq:mpole} that the multipole expansion
around center $\b c$ evaluated at a target point $\b x \in T$ is,
\[
  \DD{p} G(\b x - \b c)^T \alpham.
\]
Inserting~\eqref{eq:decompression} into the above equation and using
associativity, we obtain
\[
  \DD{p} G(\b x - \b c)^T \b \alpha = (M  [\DD{p} G(\b x - \b c)]_{\b j})^T \alpham
    = [\DD{p} G(\b x - \b c)]_{\b j}^T M^T \alpham
    = [\DD{p} G(\b x - \b c)]_{\b j}^T \betam,
\] where
\begin{equation}
  \label{eq:define_beta_multipole}
  \betam := M^T \alpham.
\end{equation}
We choose to store \(\betam \) instead of \(\alpham\), requiring only
$\bigO{p^{d-1}}$ elements to be stored instead of $\bigO{p^d}$. Using
Theorem~\ref{thm:mult_decompression} we know that this computation requires only
$\bigO{p^d}$ operations.

\subsection{Compressed Multipole Evaluation}
\label{sec:m2p}
Recall from Section~\ref{sec:p2m} that the multipole expansion around center $\b
c$ evaluated at a target point $\b x \in T$ is given by,
\begin{equation}
  \label{eq:multipole_eval_beta}
  \greenf(\b x, \b y) \approx [\DD{p} G(\b x - \b c)]_{\b j}^T \betam,
\end{equation}
where $\betam$ are the stored coefficients, and $\b j$ is the vector of stored
indices of $\DD{p} G(\b x - \b c)$. For multipole evaluation, the derivatives
$[\DD{p}G(\b x - \b c)]_{\b j}$ need to be calculated, and there are
$\bigO{p^{d-1}}$ derivatives. Using Definition~\ref{def:wp}, calculating the
derivatives costs $\bigO{p^{d-1} w(p)}$ operations. The last step in the
evaluation is a dot product of the coefficients $\betam$ and the kernel
derivatives $[\DD{p}G(\b x - \b c)]_{\b j}$, which requires $\bigO{p^{d-1}}$
operations. Therefore, multipole evaluation performs $\bigO{p^{d-1}w(p)}$
operations.

\section{Translations on Algebraically Compressed Expansions}
\label{sec:translations}
In this section, we consider algorithms for translating compressed Taylor series
expansions as well as costs and errors incurred by them.
Subsections~\ref{sec:m2l},~\ref{sec:l2l}, and~\ref{sec:m2m}
give algorithms for translating a compressed multipole expansion to a
compressed local expansion, shifting the center of a compressed local expansion
and shifting the center of a compressed multipole expansion, respectively.

\subsection{Multipole-to-Local Translation}
\label{sec:m2l}
We first derive the multipole-to-local translation.
Let $\alpham \in \mathbb{C}^{N(p)}$ be the coefficients of a
$p$-th-order multipole expansion around $\b c_1$ and
$\betam = M^T \alpham$. Then, the multipole expansion can
be written as
\begin{equation}
  [\DD{p} G(\b x - \b c_1)]^T \alpham
  = [\DD{p} G(\b x - \b c_1)]_{\b j}^T \betam,
  \label{eq:m2l_multipole}
\end{equation}
where the latter expression is the compressed form.
We then translate the (compressed) multipole expansion
\eqref{eq:m2l_multipole}
to a local expansion centered at $\b c_2$. The $i$-th translated
local expansion coefficient is defined as
\lean{PDETranslations.m2lLocalCoefficients}
\begin{align}
  \alphal_i &:= \frac{1}{\bnu(i)!} \left. \D x^{\bnu(i)} \left(
    [\DD{p} G(\b x - \b c_1)]_{\b j}^T \betam \right) \right|_{\b x = \b c_2}
    = \frac{1}{\bnu(i)!} \left. \left[ \DD{p} \left( \D z^{\bnu(i)} G(\b z) \right) \right]_{\b j}^T
    \betam \right|_{\b z = \b c_2 - \b c_1},
  \label{eq:m2l_toeplitz}
\end{align}
where $i \in \{1, 2, \ldots, N(p)\}$. In order to show $\alphal$ satisfies
the compressibility condition \eqref{eq:compressed-local-form}, one may
observe that
\begin{align*}
  \alphal_i
  &= \frac{1}{\bnu(i)!} \left. \left[ \DD{p} \left(\D{z}^{\bnu(i)} G(\b z) \right) \right]_{\b j}^T
    \betam \right|_{\b z = \b c_2 - \b c_1} \\
  &= \frac{1}{\bnu(i)!} \sum_{s=1}^{|\b j|} \betam_s
    \left. \left[ \D{z}^{\bnu(\b j_s)} \left( \DD{p} G(\b z) \right) \right]_i
    \right|_{\b z = \b c_2 - \b c_1} \\
  &= \frac{1}{\bnu(i)!} \sum_{s=1}^{|\b j|} \betam_s
    \left. \left[ \D{z}^{\bnu(\b j_s)} \left( M \left[ \DD{p} G(\b z) \right]_{\b j} \right) \right]_i
    \right|_{\b z = \b c_2 - \b c_1} \\
  &= \frac{1}{\bnu(i)!} \sum_{s=1}^{|\b j|} \betam_s
    \left. \left[ M \left[ \D{z}^{\bnu(\b j_s)} \left( \DD{p} G(\b z) \right) \right]_{\b j} \right]_i
    \right|_{\b z = \b c_2 - \b c_1} \\
  &= \frac{1}{\bnu(i)!} \sum_{s=1}^{|\b j|} \betam_s
    \left. \left[ M \left[ \DD{p} \left( \D{z}^{\bnu(\b j_s)} G(\b z) \right) \right]_{\b j} \right]_i
    \right|_{\b z = \b c_2 - \b c_1} \\
  &= \frac{1}{\bnu(i)!} \sum_{s=1}^{|\b j|} \betam_s \sum_{t=1}^{|\b j|} M_{i,t}
    \left. \left[ \DD{p} \left( \D{z}^{\bnu(\b j_s)} G(\b z) \right) \right]_{\b j_t}
    \right|_{\b z = \b c_2 - \b c_1} \\
  &= \frac{1}{\bnu(i)!} \sum_{t=1}^{|\b j|} M_{i,t} \sum_{s=1}^{|\b j|} \betam_s
    \left. \left[ \DD{p} \left( \D{z}^{\bnu(\b j_t)} G(\b z) \right) \right]_{\b j_s}
    \right|_{\b z = \b c_2 - \b c_1} \\
  &= \frac{1}{\bnu(i)!} \sum_{t=1}^{|\b j|} M_{i,t}
    \left. \left[ \DD{p} \left( \D{z}^{\bnu(\b j_t)} G(\b z) \right) \right]_{\b j}^T \betam
    \right|_{\b z = \b c_2 - \b c_1} \\
  &= \frac{1}{\bnu(i)!} \sum_{t=1}^{|\b j|} M_{i,t}\, \alphal_{\b j_t}\, \bnu(\b j_t)! \\
  &= \frac{[M (\alphal_{\b j} \bnu(\b j)!)]_i}{\bnu(i)!}.
\end{align*}
Therefore, in order to obtain the full multipole-to-local coefficients from
$\b c_1$ to $\b c_2$, it suffices to calculate the $\bigO{p^{d-1}}$
coefficients $\b \gamma_{\b j}$ of \eqref{eq:m2l_toeplitz}. We refer to this
as the multipole-to-local translation in compressed representation.
\begin{prop}
\label{prop:m2l_error}
\lean{PDETranslations.m2lLocalCoefficients_isCompressible}
Let $\alpham \in \mathbb{C}^{N(p)}$ be the $p$-th-order multipole
expansion coefficients around $\b c_1$.
Suppose $\alphal$ are the translated local expansion coefficients
\eqref{eq:m2l_toeplitz} from $\b c_1$ to $\b c_2$.
Then, $\alphal$ satisfies the compressibility condition
\eqref{eq:compressed-local-form}. In particular, the
multipole-to-local translation in compressed representation introduces no
additional error beyond the truncation of the multipole expansion
and the truncation of the local expansion.
\end{prop}

To set the stage for the statement of the multipole-to-local
translation algorithm in compressed representation, consider a linear
PDE of order $c$ in $d$ dimensions in the form
\[
  \mathcal L G(\b x) =
  \sum_{|\b m| = c} a_{\b m} \D x^{\b m} G (\b x) +
    \sum_{|\b m| < c} a_{\b m} \D x^{\b m} G (\b x) = 0,
\]
with $a_{\b m}\in \mathbb C$ for $\b m\in \mathcal M(c)$,

where we have separated the highest-order terms from the lower-order terms.

The following property (and its consequences as given by the subsequent
lemmas) yields a reduction in the asymptotic cost of our translation
operators:
\begin{property}
\label{prop:nice}
\lean{PDETranslations.LinearPDEOperator.HasNonzeroPurePower}
\[
  a_{c \b e_k} \neq 0 \text{ for some } k\in\{1,\dots,d\}.
\]
\end{property}
\begin{lemma}
  \lean{PDETranslations.LinearPDEOperator.IsElliptic.hasNonzeroPurePower}
  A constant-coefficient elliptic PDE has property~\ref{prop:nice}.
\end{lemma}
\begin{proof}
  Using the definition for an elliptic PDE given by~\citet{cosner1991definition}
  a constant-coefficient elliptic PDE is a PDE such that the principal symbol
  \[
    \mathcal P(\b \xi)= \sum_{|\b m| = c}a_{\b m} \b\xi^{\b m} \neq 0
  \]
  \lean{PDETranslations.LinearPDEOperator.principalSymbol}
  for all nonzero vectors $\b\xi \in \mathbb{R}^d$.
  \lean{PDETranslations.LinearPDEOperator.IsElliptic}
  Here ellipticity means precisely this nonvanishing condition, not a
  positivity or strong-ellipticity condition.
  Suppose an elliptic PDE does not have Property~\ref{prop:nice},
  then, $a_{c \b e_k} = 0$ for all $k\in\{1,\dots,d\}$.
  Then $\mathcal P(\b e_k)=a_{c\b e_k}=0$ for $k\in\{1,\dots,d\}$,
  contradicting ellipticity.
\end{proof}
Fast Fourier Transforms (FFTs) will play an important role in our
translation algorithm. In this next lemma, we bound an expression that
determines the size of the FFTs carried out as part of the evaluation
of the translation operator.
\begin{lemma}
  \label{lemma:max_prod_nice}
  \lean{PDETranslations.exists_ordering_storedMaxProduct_isBigO}
  There exists a graded monomial ordering $\bnu$ such that, for a PDE with
  Property~\ref{prop:nice},
  \begin{equation}
    \prod_{i=1}^{d}\max_{s \in \b j} \bnu(s)_i=\bigO{p^{d-1}},
    \label{eq:max_prod_nice}
  \end{equation}
  where the vector of stored indices $\b j$ depends on the ordering $\bnu$.
\end{lemma}
\begin{proof}
  Let $k \in \{1,\dots,d\}$ be such that $a_{c \b e_k} \neq 0$.
  Let $\bnu$ be a degree-lexicographic ordering with the $k$th entry varying slowest.
  Then $\bnu^{-1}(c \b e_k)=N(c)$, i.e. it is the last entry in the
  ordering of multi-indices $\mathcal M(c)$.
  In terms of $h$ from \eqref{eq:define_h},
  $h(1) = \max \bnu^{-1}(\mathcal M(c)) = \bnu^{-1}(c \b e_k)$, i.e.
  $\bnu(h(1)) = c \b e_k$.
  From Lemma~\ref{lemma:h_eq}, we have
  \[
    \bnu(h(q)) = \bnu(h(1)) + \bnu(q) = c \b{e}_k + \bnu(q)
    \qquad(q\in \{1, 2, \ldots, N(p-c)\}).
  \]
  Therefore all multi-indices $\b m \in \mathcal{M}(p)$ with
  $\b m_k \geq c$ belong to the unrepresented indices $\bnu(\jcomp)$,
  as defined in \eqref{eq:define_bar_j}.
  On the other hand, for represented indices $\b j$,
  $\bnu(s)_k < c$ for $s \in \b j$.
  Thus
  \begin{align*}
  \max_{s \in \b j} \bnu(s)_k &= c - 1=O(1) \qquad \text{(as $p\to\infty$)},
  \intertext{while}
  \max_{s \in \b j} \bnu(s)_i &= O(p)\qquad (i\ne k).
  \end{align*}
\end{proof}
Consider the three examples in Figure~\ref{fig:j}.
We know that the Laplace equation and biharmonic equation satisfy
Property~\ref{prop:nice} and $\b j$ given in Figures~\ref{fig:j_laplace}
and \ref{fig:j_biharmonic} are such that $\max_{s \in \b j} \bnu(s)_2$
is $1$ and $3$ respectively, confirming Lemma~\ref{lemma:max_prod_nice}.
We also know that the PDE $\partial^2 u/\partial x \partial y = 0$ does not
satisfy the Property~\ref{prop:nice}, and we can see that the
resulting numbering does not satisfy \eqref{eq:max_prod_nice} since
$\max_{s \in \b j} \bnu(s)_2 = \max_{s \in \b j} \bnu(s)_1 = p$
as shown in Figure~\ref{fig:j_weird}.

We are now ready to state an algorithm for multipole-to-local
translation, including for the special case of PDEs with
Property~\ref{prop:nice}.
The derivatives needed in the translation \eqref{eq:m2l_toeplitz} are
$\D z^{\bnu(q)+\bnu(i)}G(\b z)|_{\b z=\b c_2-\b c_1}$
for $q,i\in\b j$.
They can be obtained by computing the stored derivatives of order up to $2p$
(Definition~\ref{def:wp}) and recovering the full set of derivatives via the decompression operator (Theorem~\ref{thm:mult_decompression}), in $\bigO{p^{d}+p^{d-1}w(2p)}$ operations.
For PDEs with Property~\ref{prop:nice}, full decompression is unnecessary:
only $\bigO{p^{d-1}}$ of the needed derivatives are not stored, and each
follows from stored ones via the PDE in $\bigO{1}$ operations, so the
cost is $\bigO{p^{d-1}w(2p)}$.
Since these derivatives depend only on $\b c_2-\b c_1$, and
a typical FMM uses at most $7^d-3^d$ distinct translation vectors
per level, these derivatives can be \emph{precomputed} and reused.

\begin{theorem}
Translating a multipole expansion in compressed representation around center $\b c_1$
to a local expansion in compressed representation around center $\b c_2$ can be
achieved with
\begin{itemize}
\item $\bigO{p^{2d-2}}$ work by direct summation, in particular $\bigO{p^2}$ in two dimensions,
\item $\bigO{p^{d} \log(p)}$ work using the FFT, and
\item $\bigO{p^{d-1} \log(p)}$ work using the FFT for PDEs with Property~\ref{prop:nice}.
\end{itemize}
Here, the derivatives in~\eqref{eq:m2l_toeplitz} are assumed
to be precomputed.
\end{theorem}
\begin{proof}
As shown in the beginning of this section, it suffices to calculate
the $\bigO{p^{d-1}}$ coefficients $\b \gamma_{\b j}$ of \eqref{eq:m2l_toeplitz}.
Since each $\b\gamma_i$ is a sum of $|\b j|=\bigO{p^{d-1}}$ terms
and there are only $\bigO{p^{d-1}}$ coefficients in the compressed
target local expansion, the total work required is $\bigO{p^{2d-2}}$.

The translation operator in~\eqref{eq:m2l_toeplitz} can be transformed to a
$d$-dimensional convolutional form.
Let $\b \eta \in \bnu(\b j)$,
$\b \gamma'_{\bnu(\b j_i)} = \b \gamma_{\b j_i}$ and
$\betamp_{\bnu(\b j_i)} = \betam_i$ for $i \in \{1,\ldots,|\b j|\}$, with
$\betamp_{\b \zeta} = 0$ for
$\b \zeta \in \mathbb{N}_0^d \setminus \bnu(\b j)$.
Re-writing the expression for translated expansion coefficient in three dimensions
with the individual dimensions made explicit, we obtain
\[
  \b \gamma'_{\eta_1, \eta_2, \eta_3} =
  \frac{1}{\eta_1! \eta_2! \eta_3!}
    \sum_{\zeta_1=0}^{m_1} \sum_{\zeta_2=0}^{m_2} \sum_{\zeta_3=0}^{m_3}
    \left. \D z ^{(\eta_1+\zeta_1, \eta_2+\zeta_2, \eta_3+\zeta_3)}
    G(\b z) \right|_{\b z = \b c_2 - \b c_1}
    \betamp_{\zeta_1, \zeta_2, \zeta_3},
\] where $m_1, m_2, m_3$ are the maximum nonzero indices for each dimension, i.e.
$ m_i = \max_{s \in \b j} \bnu(s)_i$.
With $\tilde{\betam}_{-\zeta_1, -\zeta_2, -\zeta_3} := \betamp_{\zeta_1, \zeta_2, \zeta_3}$,
rearranging the summation variables, we obtain
\begin{equation}
  \b \gamma'_{\eta_1, \eta_2, \eta_3} =
  \frac1{\eta_1! \eta_2! \eta_3! }
  \sum_{\zeta_1=-m_1}^{0} \sum_{\zeta_2=-m_2}^{0} \sum_{\zeta_3=-m_3}^{0}
  \left. \D z^{(\eta_1-\zeta_1, \eta_2-\zeta_2, \eta_3-\zeta_3)}
  G(\b z) \right|_{\b z = \b c_2 - \b c_1}
  \tilde{\betam}_{\zeta_1, \zeta_2, \zeta_3},
  \label{eq:convolution_m2l}
\end{equation}
revealing the convolutional form of the operator.

This is similar to the convolution-based algorithm for the M2L translation operator for
3D Laplace described in~\citep{greengard1988efficient}.
In a further step, this convolution can be applied by viewing it as
multiplication of the coefficient vector by a Toeplitz matrix. This
multiplication in turn can be realized via a Fast Fourier Transform
through an embedding in a circulant matrix. We refer to~\citep{greengard1988efficient}
for the details, noting that while the reference only discusses
convolutions in one and two dimensions, the extension to
$d$ dimensions is straightforward (where, in our case, the sizes of the
$d$-dimensional convolution are $(m_1 + 1, \ldots, m_d + 1)$).

A one-dimensional convolution of size $m$ can be represented as
multiplication by a Toeplitz matrix, which in turn
can be embedded in a circulant matrix of size $2m - 1$.
A circulant matrix multiplication can be carried out using an FFT of
length $2m - 1$. Similarly, a $d$-dimensional convolution with sizes
$(m_1 + 1, \ldots, m_d + 1)$ can be carried out using an FFT of sizes
$(2m_1 + 1, \ldots, 2m_d + 1)$, at cost $\bigO{N \log N}$, where
$N = \prod_{i=1}^{d}(2m_i + 1)$ is the total length of the FFT.
Since $m_i = \bigO{p}$, the length of the FFT is $\bigO{p^d}$.
Using Lemma~\ref{lemma:max_prod_nice}, the length of the FFT is
$\bigO{p^{d-1}}$ for PDEs satisfying Property~\ref{prop:nice}.

The overall time complexity of the translation is therefore $\bigO{p^{d-1} \log p}$
for PDEs satisfying Property~\ref{prop:nice} and $\bigO{p^{d} \log p}$ for
all other PDEs. In particular, for three-dimensional Laplace, the FFT
is of sizes $(2p+1, 2p+1, 3)$ which leads to a $\bigO{p^2 \log p}$ algorithm.
\end{proof}

\textbf{Cost comparisons with other M2L translation algorithms in the
literature.}
The specialized algorithm by~\citet{greengard1988efficient}
for the Laplace equation requires an FFT of sizes $(2p, 4p+1)$ which
requires approximately one third fewer operations than our
algorithm. As a further comparison with the literature, the algorithm
without the FFT performs $\bigO{p^{2d-2}}=\bigO{p^4}$ floating point
operations for Laplace 3D, which has the same cost as the algorithm
based on tensor contraction by \citet{shanker2007accelerated} and is
an improvement over the algorithm given in
\citep{oppelstrup2013matrix} where common sub-expression elimination
is used to give an (empirically) $\bigO{p^{4.5}}$ algorithm.

\subsubsection{Numerical stability}
\label{sec:numerical_stability}

While the Fast Fourier Transform is numerically stable, having input values that
differ in magnitude has the potential to result in substantial error resulting
from rounding of intermediate quantities.
\citet{greengard1988efficient} provide a partial mitigation to this problem for the
Laplace equation, where scaling the derivatives and the monomials in the Taylor series
reduces the error by reducing the range of the magnitudes of the derivatives.
Rewriting \eqref{eq:m2l_toeplitz} using a scaling parameter $t$, we obtain
\[
  \b \gamma_i = \frac{t^{|\bnu(i)|}}{\bnu(i)!} \sum_{s=1}^{|\b j|}
    \frac{\left. \D z ^{\bnu(\b j_s) + \bnu(i)} G(\b z) \right|_{\b z = \b c_2 - \b c_1} }
    {t^{|\bnu(\b j_s)+\bnu(i)|}} \betam_{s} t^{|\bnu(\b j_s)|}.
\]
This preserves the convolutional form of the operator, and therefore
the multipole-to-local translation algorithm of Section~\ref{sec:m2l}
applies with only minor modifications.

Choosing the scaling parameter depends on the magnitudes of the derivatives
$\D z ^{\bnu(\b j_s) + \bnu(i)} G(\b z) $. For the three-dimensional Laplace
equation, \citet{greengard1988efficient} suggest using
$t = p/r$, where $p$ is the order and $r$ is the distance between the centers
of two neighboring boxes in the tree, akin to an expansion `radius'.
Empirically, the same parameter succeeds in
controlling the magnitudes of entries of the input vector in our algorithm.
For the two-dimensional biharmonic equation on the other hand, we
observed a smaller range (ratio of largest to smallest
magnitude) in the scaled derivatives when using $p/(4r)$ instead of
$p/r$ for an expansion of order $p=16$.
The biharmonic experiments in Section~\ref{sec:biharmonic_ie} use $t = p/(4r)$;
the resulting errors are reported in Table~\ref{tab:biharmonic_ie_error}.
Optimal selection of the scaling parameter for general
kernels remains a subject for future research.

\subsection{Multipole-to-Multipole Translation}
\label{sec:m2m}
In the uncompressed case, multipole-to-multipole translation using
Taylor series is worked out in some detail below, 
using the binomial theorem. 

Let $\b c_2\in\mathbb{R}^d$ be a multipole expansion center. The
$p$-th-order multipole expansion induced by a source point $\b y$
with $\|\b y-\b c_2\|<R$ (and valid for $\|\b x-\b c_2\|>R$) takes the form
\begin{equation}
  \label{eq:mpole_eval_translated_full}
  \sum_{q=1}^{N(p)} \left. \b \rho_q \, \D z^{\bnu(q)} G(\b z)\right|_{\b z=\b x-\b c_2},
\end{equation}
with coefficients
\[
  \b\rho_q := \frac{(\b c_2-\b y)^{\bnu(q)}}{\bnu(q)!}
  \qquad (q\in\{1,2,\ldots,N(p)\}).
\]
Using the (multi-index) binomial theorem, we may express $\b \rho_q$ in terms of
multipole coefficients about a different center $\b c_1\in\mathbb{R}^d$.
Denote $\b h=\b c_2-\b c_1$. Then,
\begin{align*}
  \b \rho_q
  = \frac{(\b c_2-\b y)^{\bnu(q)}}{\bnu(q)!}
  &= \frac{(\b h + \b c_1-\b y)^{\bnu(q)}}{\bnu(q)!} \\
  &= \frac{1}{\bnu(q)!}
     \sum_{\substack{i\in\{1,\ldots,N(p)\}\\ \bnu(i)\le\bnu(q)}}
     \binom{\bnu(q)}{\bnu(i)}
     \b h^{\bnu(q)-\bnu(i)} (\b c_1-\b y)^{\bnu(i)} \\
  &= \sum_{\substack{i\in\{1,\ldots,N(p)\}\\ \bnu(i)\le\bnu(q)}}
     \frac{\b h^{\bnu(q)-\bnu(i)}}{(\bnu(q)-\bnu(i))!}
     \underbrace{\frac{(\b c_1-\b y)^{\bnu(i)}}{\bnu(i)!}}_{\text{multipole coefficients around $\b c_1$} }.
\end{align*}
\lean{PDETranslations.m2mOperator_multipoleCoefficients}
Define $\alpham_i := \frac{(\b c_1-\b y)^{\bnu(i)}}{\bnu(i)!}$. Then,
\begin{equation}
  \label{eq:mpole_translated_full_coeffs}
  \b \rho_q
  = \sum_{\substack{i\in\{1,\ldots,N(p)\}\\ \bnu(i)\le\bnu(q)}}
    \frac{\b h^{\bnu(q)-\bnu(i)}}{(\bnu(q)-\bnu(i))!}\,\alpham_i,
  \qquad (q\in\{1,2,\ldots,N(p)\}).
\end{equation}
\begin{definition}[multipole-to-multipole translation operator]
Let $\b c_1,\b c_2\in\mathbb{R}^d$ and $\b h=\b c_2-\b c_1$.
The multipole-to-multipole translation operator from $\b c_1$ to $\b c_2$ is
the lower-triangular linear map
\[
  T_{\b c_1\to\b c_2}:\mathbb{C}^{N(p)}\longrightarrow \mathbb{C}^{N(p)},
\]
with entries
\begin{equation}
  [T_{\b c_1\to\b c_2}]_{q,i} =
  \begin{cases}
    \displaystyle
    \frac{\b h^{\bnu(q)-\bnu(i)}}{(\bnu(q)-\bnu(i))!},
      & \text{if } \bnu(i)\le\bnu(q), \\[1.2ex]
    0, & \text{otherwise,}
  \end{cases}
  \qquad (q,i\in\{1,2,\ldots,N(p)\}).
  \label{eq:m2m_translation_operator}
\end{equation}
\end{definition}
To simplify notation, we will write $T$ instead of $T_{\b c_1 \to \b c_2}$ when
the centers are clear from context. Then, $\b \rho = T \alpham$ and
the translated multipole expansion
\[
  [\DD p G(\b x - \b c_2)]^T T \alpham
\]
has exactly the coefficients obtained by directly forming the order-$p$
multipole expansion about $\b c_2$. Thus, the translation introduces no
additional algebraic error; the resulting truncation error and validity region
are those of the expansion about $\b c_2$ in \eqref{eq:mpole_truncated}.

The remainder of this section is divided into three subsections.
In Subsection~\ref{sec:m2m_compressed}, we derive a compressed
representation of the multipole-to-multipole translation.
In Subsection~\ref{sec:m2m_error_analysis}, we show that for PDEs with only
equal-order terms (such as Laplace), the compressed and uncompressed
translations agree. For PDEs with varying-order terms, we bound the
difference between the compressed and uncompressed translations by
$\bigO{R_2^{p+1}}$ as $R_2\to0$, with the order $p$, target distance,
and source-strength bound fixed. This is the same power of $R_2$ as in
the Taylor truncation bound about the new center; it does not establish
convergence as $p\to\infty$ at fixed separation ratio.
Finally, in Subsection~\ref{section:m2m_compute_faster}, we give an algorithm
to perform this computation with low asymptotic cost.

\subsubsection{Multipole-to-Multipole Translation in Compressed Representation}
\label{sec:m2m_compressed}
For the uncompressed case, the translated multipole expansion from center
$\b c_1$ to $\b c_2$ is given by $[\DD p G(\b x-\b c_2)]^{T} T \alpham$.
By linearity, we can decompose the translated
expansion into contributions from individual coefficients:
\begin{equation}
  [\DD p G(\b x-\b c_2)]^{T} T \alpham
  = \sum_{k=1}^{N(p)} [\DD p G(\b x-\b c_2)]^{T} T (\alpham_k \b e_k),
  \label{eq:m2m_decomposed}
\end{equation}
where $\b e_k$ is the $k$-th standard basis vector in
$\mathbb{C}^{N(p)}$.
\begin{prop}
\label{prop:m2m_single_coeff}
\lean{PDETranslations.m2m_single_coeff}
Let $\b c_1,\b c_2\in\mathbb{R}^d$ and $\b h=\b c_2-\b c_1$, and let
$\alpham\in\mathbb{C}^{N(p)}$ denote the multipole expansion coefficients about $\b
c_1$.
Then, for any $k\in\{1,\dots,N(p)\}$,
\begin{align}
  [\DD p G(\b x-\b c_2)]^{T} T (\alpham_k \b e_k)
  =
  \left.
  \sum_{q=1}^{N(p)}
  \left(
    \sum_{\substack{
      s\in\{1,\dots,|\b j|\}\\
      \bnu(\b j_s)\le \bnu(q)\\
      |\bnu(q)|-|\bnu(\b j_s)|\le p-|\bnu(k)|}}
    \frac{\b h^{\bnu(q)-\bnu(\b j_s)}}{(\bnu(q)-\bnu(\b j_s))!}\,
    \bigl[M^{T}(\alpham_k \b e_k)\bigr]_{s}
  \right)
  \D z^{\bnu(q)}G(\b z)
  \right|_{\b z=\b x-\b c_2}.
  \label{eq:m2m-single-coeff}
\end{align}
\end{prop}
\begin{proof}
  \allowdisplaybreaks
 \begin{align*}
 &\quad [\DD p G(\b x - \b c_2)]^T T (\alpham_k \b e_k) \\
    &= [\DD p G(\b x - \b c_2)]^T_{[k:N(p)]}
       (\alpham_k T_{[k:N(p), k]} ) \\ 
    &= \sum_{\ell =1}^{N(p - |\bnu(k)|)} \left.\alpham_k
       \frac{\b h^{\bnu(\ell)}}{\bnu(\ell)!}
       \D z ^{\bnu(\ell) + \bnu(k)} G(\b z)\right|_{\b z = \b x - \b c_2}\\ 
    &= \left.\left[ \left(\sum_{\ell = 1}^{N(p - |\bnu(k)|)}
       \frac{\b h^{\bnu(\ell)}}{\bnu(\ell)!} 
       \D z ^{\bnu(\ell)} \right) \alpham_k
       \D z ^{\bnu(k)} G(\b z) \right]\right|_{\b z = \b x - \b c_2} \\ 
    &= \left.\left[ \left(\sum_{\ell = 1}^{N(p - |\bnu(k)|)}
       \frac{\b h^{\bnu(\ell)}}{\bnu(\ell)!} 
       \D z ^{\bnu(\ell)} \right)
       \sum_{\substack{s \in \{ 1, \dots , |\b j| \}}}
       \alpham_k M_{k, s}
       \D z ^{\bnu(\b j_s)} G(\b z) \right]\right|_{\b z = \b x - \b c_2} \\      
    &= \left.\left[\sum_{\substack{s \in \{ 1, \dots , |\b j| \}}} 
       \b [M^{T}(\alpham_k \b e_k)]_{s}
       \left(\sum_{\ell = 1}^{N(p - |\bnu(k)|)}
       \frac{\b h^{\bnu(\ell)}}{\bnu(\ell)!} 
       \D z ^{\bnu(\ell) + \bnu(\b j_s)} G(\b z)
       \right)\right]\right|_{\b z = \b x - \b c_2} \\  
    &=\left. \left[ \sum_{s \in \{ 1, \dots  , |\b j| \} }
       [M^{T} (\alpham_k \b e_k)]_{s} \left(
       \sum_{\substack{q \in \{ 1, \dots  , N(p) \} \\
       \bnu(\b j_s) \leq \bnu(q)\\
       |\bnu(q)| - |\bnu(\b j_s)| \leq p - |\bnu(k)|}}
       \frac{\b h ^{\bnu(q) - \bnu(\b j_s)}}{(\bnu(q) - \bnu(\b j_s))!}
       \D z ^{\b \nu(q)} G(\b z) \right)  \right]
       \right|_{\b z = \b x - \b c_2} \\ 
    &= \left. \sum_{q=1}^{N(p)}  \left(
       \sum_{\substack{s \in \{ 1, \dots  , |\b j| \} \\
       \bnu(\b j_s) \leq \bnu(q) \\
       |\bnu(q)| - |\bnu(\b j_s)| \leq p - |\bnu(k)|}}
       \frac{\b h ^{\bnu(q) - \bnu(\b j_s)}}{(\bnu(q) - \bnu(\b j_s))!}
       [M^{T} (\alpham_k \b e_k)]_{s}  \right)
       \D z ^{\bnu(q)} G(\b z) \right|_{\b z = \b x - \b c_2}.
\end{align*} 
\end{proof}
For PDEs with only terms of equal order, Theorem~\ref{thm:m_equal_order_property}
implies that $[M^{T}(\alpham_k \b e_k)]_{s} = 0$ if $|\bnu(\b j_s)| \neq |\bnu(k)|$.
In that case, the constraint $|\bnu(q)| - |\bnu(\b j_s)| \le p - |\bnu(k)|$ in
\eqref{eq:m2m-single-coeff} can be removed. Hence, we have the following:
\begin{prop}\label{prop:m2m_single_coeff_equal_order}
\lean{PDETranslations.m2m_single_coeff_equal_order}
Let $\b c_1, \b c_2 \in \mathbb{R}^d$, $\b h = \b c_2 - \b c_1$, and let $\alpham$ be the multipole coefficients
around $\b c_1$. For PDEs with only terms of equal order,  
\begin{align}
 [\DD p G(\b x - \b c_2)]^T T (\alpham_k \b e_k) 
 = \left. \sum_{q=1}^{N(p)}  \left( 
  \sum_{\substack{s \in \{ 1, \dots  , |\b j| \} \\ \bnu(\b j_s) \leq \bnu(q)}} 
  \frac{\b h ^{\bnu(q) - \bnu(\b j_s)}}{(\bnu(q) - \bnu(\b j_s))!} 
  [M^{T} (\alpham_k \b e_k)]_{s}  
  \right) \D z ^{\bnu(q)} G(\b z) 
 \right|_{\b z = \b x - \b c_2}. 
 \label{eq:m2m-single-coeff-equal-order}
\end{align}
\end{prop}
\begin{definition}
  \lean{PDETranslations.multipoleEmbeddingMatrix}
  Let $\pi: \{ 1, \dots, |\b j| \} \to \{ 1, \dots, N(p) \}$ be an enumeration
  of $\b j$ such that $\pi(s) = \b j_{s}$. We define an embedding
  operator $E: \mathbb{C}^{|\b j|} \to \mathbb{C}^{N(p)}$ by
  \begin{equation}
    [E \b v]_i =
    \begin{cases}
      \b v_{\pi^{-1}(i)} & \text{if } i \in \b j, \\
      0 & \text{otherwise},
    \end{cases} \qquad (i \in \{1, 2, \ldots, N(p)\}).
    \label{eq:embedding_operator}
  \end{equation}
\end{definition}
\noindent

The compressed multipole-to-multipole translation from center $\b c_1$ to center $\b c_2$
can be summarized as follows:
\begin{enumerate}
  \item Start with full multipole coefficients $\b\alpha$ around $\b c_1$, 
  \item Compress to obtain $\b\beta = M^T \b\alpha$, 
  \item Embed into uncompressed space: $E\b\beta$,
  \item Apply the translation operator from $\b c_1$ to $\b c_2$: $\b \sigma = T[E\b\beta]$,
  \item Optionally compress again if further translations are needed:
  $M^{T} \b \sigma$. The recompression cost is $\bigO{p^d}$ as shown in Theorem~\ref{thm:mult_decompression}.
\end{enumerate}
Figure~\ref{fig:m2m_diagrams} shows a two-dimensional example in
schematic form, showing the state of the expansion on a path through
compression and translation. In the figure, each lattice point
represents a coefficient in $\mathcal M(p)$. We denote
$\alphamp_{\b q}:=\alpham_{\bnu^{-1}(\b q)}$, 
$\betamp_{\b q}:=[E\betam]_{\bnu^{-1}(\b q)}$ and $\b e'_{\b q}:=\b
e_{\bnu^{-1}(\b q)}$ for $\b q\in\mathcal M(p)$.

Assuming only one coefficient $\alphamp_{\boldsymbol k}$ is nonzero 
in the expansion around the first center as in
Figure~\ref{fig:m2m_diagrams}(a), the nonzero coefficients of the
translated expansion around the new center are shown in
Figure~\ref{fig:m2m_diagrams}(d). If one were to compress first,
however, two cases arise, shown as the second and third column
respectively. The third column shows the more straightforward case of
a PDE with only terms of equal order (e.g. Laplace). In
this case, compress-embed-translate incurs no additional error relative
to the uncompressed translation. The single
coefficient from Figure~\ref{fig:m2m_diagrams}(a) may become one or more nonzero
coefficients as shown in Figure~\ref{fig:m2m_diagrams}(c). After
translation, each of those coefficients spawns a `triangle' of
coefficients in the translated expansion, shown in
Figure~\ref{fig:m2m_diagrams}(f). In the case of a PDE with varying-order
terms (e.g. Helmholtz), cf. Figure~\ref{fig:m2m_diagrams}(b), a difference
can arise, as shown schematically in Figure~\ref{fig:m2m_diagrams}(e).
\begin{figure}
\centering
\makebox[.3\linewidth][c]{\textbf{Uncompressed}}
\hspace{3mm}
\makebox[.3\linewidth][c]{\textbf{Compressed (varying-order)}}
\hspace{3mm}
\makebox[.3\linewidth][c]{\textbf{Compressed (equal-order)}}
\par\vspace{3mm}
\begin{subfigure}[t]{.3\linewidth}
\begin{tikzpicture}[scale=0.7]
\foreach \i in {0.25,0.5,...,4.0}{
  \foreach \j in {0,0.25,...,\i} {
    \draw (\i-\j,\j) node[fill=black!25,circle,scale=0.1] {};
  }
}
\draw (0,0) node[fill=black!25,circle,scale=0.1] {};
\draw[thick,opacity=0.5] (0,0)
  -- (4,0)
  -- (0,4)
  -- cycle;
\draw (1,1.5) node[anchor=center,fill=black,circle,scale=0.4,label=below:$\alphamp_{\b k}$] {};
\draw[draw opacity=0.5,->] (0,0)--(4.75,0) node[right]{$\eta_1$};
\draw[draw opacity=0.5,->] (0,0)--(0,4.75) node[above]{$\eta_2$};
\draw[->] (2.5,3) node[right] {$\alphamp_{\mathcal{M}(p)}$} -- (1.5, 2.5);
\draw (3.5,2.2) -- (5.5,2.2) -- (5.5, 2.025) -- (5.75, 2.325) --
      (5.5, 2.575) -- (5.5,2.45) -- (3.5,2.45) -- (3.5,2.2);
\node at (4.5, 1.8) {\small Compress};
\node at (4.5, 1.4) {\small Embed};
\end{tikzpicture}
\caption{
  Index space showing the coefficients  $\alphamp$ 
  around center $\b c_1$. We assume that only $\alphamp_{\b k}$ is nonzero.
}
\end{subfigure}\hspace{3mm}
\begin{subfigure}[t]{.3\linewidth}
\begin{tikzpicture}[scale=0.7]
\foreach  [evaluate = \i as \k using 4 - \i] \i in {0.25,0.5,...,4.0}{
  \foreach \j in {0,0.25,...,\i} {
    \draw (\i-\j,\j) node[fill=black!25,circle,scale=0.1] {};
  }
}
\draw[draw opacity=0.5,->] (0,0)--(4.75,0) node[right]{$\eta_1$};
\draw[draw opacity=0.5,->] (0,0)--(0,4.75) node[above]{$\eta_2$};
\draw[thick,opacity=0.5] (0,0)
  -- (4,0)
  -- (0,4)
  -- cycle;
\draw[fill=black!50] (0,1)
  -- (3,1);
\draw[dashed,fill=black!25,opacity=0.5] (1,1.5)
  -- (2.25, 0.25);
\draw[->] (2.5,3) node[right] {$\betamp_{\bnu(\jcomp)} = 0$} -- (1.5, 2.5);
\draw[->] (4.5,1) node[right] {$\betamp_{\bnu(\b j)}$} -- (3.5, 0.5);
\draw [dashed,fill=gray] (1,1.5)
  node[fill=black!25,circle,anchor=center,scale=0.4,
       label={[text=black!25]west:$\alphamp_{\b k}$}]{};
\draw [dashed,fill=gray] (1.25,0.75) node[fill=black,circle,anchor=center,scale=0.4,label=west:$\betamp_{\b s^1}$]{};
\draw [dashed,fill=gray] (2.25,0.25) node[fill=black,circle,anchor=center,scale=0.4,label=west:$\betamp_{\b s^0}$]{};
\end{tikzpicture}
  \caption{Index space showing the coefficients $\betamp$ 
  for a PDE with \textit{varying-order} terms. The 
coefficient $\alphamp_{\b k}$ gives rise to coefficients
$\betamp_{\b s^0}$ and $\betamp_{\b s^1}$ with
$|\b s^0| = |\b k|$ and $|\b s^1| < |\b k|$.}
\end{subfigure}\hspace{3mm}
\begin{subfigure}[t]{.3\linewidth}
\begin{tikzpicture}[scale=0.7]
\foreach  [evaluate = \i as \k using 4 - \i] \i in {0.25,0.5,...,4.0}{
  \foreach \j in {0,0.25,...,\i} {
    \draw (\i-\j,\j) node[fill=black!25,circle,scale=0.1] {};
  }
}
\draw (0,0) node[fill=black!25,circle,scale=0.1] {};
\draw[draw opacity=0.5,->] (0,0)--(4.75,0) node[right]{$\eta_1$};
\draw[draw opacity=0.5,->] (0,0)--(0,4.75) node[above]{$\eta_2$};
\draw[thick,opacity=0.5] (0,0)
  -- (4,0)
  -- (0,4)
  -- cycle;
\draw[fill=black!50] (0,1)
  -- (3,1);
\draw[dashed,fill=black!25,opacity=0.5] (1,1.5)
  -- (2.25, 0.25);
\draw[->] (2.5,3) node[right] {$\betamp_{\bnu(\jcomp)} = 0$} -- (1.5, 2.5);
\draw[->] (4.5,1) node[right] {$\betamp_{\bnu(\b j)}$} -- (3.5, 0.5);
\draw [dashed,fill=gray] (1,1.5)
  node[fill=black!25,circle,anchor=center,scale=0.4,
       label={[text=black!25]west:$\alphamp_{\b k}$}]{};
\draw [dashed,fill=gray] (1.75,0.75) node[fill=black,circle,anchor=center,scale=0.4,label=west:$\betamp_{\b s^1}$]{};
\draw [dashed,fill=gray] (2.25,0.25) node[fill=black,circle,anchor=center,scale=0.4,label=west:$\betamp_{\b s^0}$]{};
\end{tikzpicture}
\caption{Index space showing the coefficients $\betamp$ for a PDE with 
only terms of \textit{equal order}.
The coefficient $\alphamp_{\b k}$ gives rise to coefficients $\betamp_{\b s^0}$ and
$\betamp_{\b s^1}$ with $|\b s^0| = |\b s^1| = |\b k|$.}
\end{subfigure}
\begin{subfigure}[t]{.3\linewidth}
\begin{tikzpicture}[scale=0.7]
\foreach  [evaluate = \i as \k using 4 - \i] \i in {0.25,0.5,...,4.0}{
  \foreach \j in {0,0.25,...,\i} {
    \draw (\i-\j,\j) node[fill=black!25,circle,scale=0.1] {};
  }
}
\draw[thick,opacity=0.5] (0,0)
  -- (4,0)
  -- (0,4)
  -- cycle;
\draw[draw opacity=0.5,->] (0,0)--(4.75,0) node[right]{$\eta_1$};
\draw[draw opacity=0.5,->] (0,0)--(0,4.75) node[above]{$\eta_2$};
\draw [dashed,fill=black!25,opacity=0.5] (1,1.5)
  -- (1,3)
  -- (2.5,1.5)
  -- cycle;
\draw [dashed,fill=gray] (1,1.5)
  node[
    fill=black,
    circle,
    anchor=center,
    scale=0.4,
    label={[text=black!40]west:$\alphamp_{\b k}$}
  ] {};
\draw[->] (2.5,2.8) node[right] {$T(\alphamp_{\b k}\,\b e'_{\b k})$}
  -- ++(-1,-0.5);
\draw (2.5, 7) -- (2.5, 5) -- (2.375, 5) -- (2.625, 4.75) --
      (2.875, 5) -- (2.75, 5) -- (2.75, 7) -- (2.5, 7);
\node at (4, 6) {\small Translate};
\draw (3.5,2.2) -- (5.5,2.2) -- (5.5, 2.025) -- (5.75, 2.325) --
      (5.5, 2.575) -- (5.5,2.45) -- (3.5,2.45) -- (3.5,2.2);
\node at (4.5, 1.8) {\small Compress};
\node at (4.5, 1.4) {\small Embed};
\end{tikzpicture}
\caption{Index space showing the coefficients obtained by translating 
$\alphamp$ to center $\b c_2$, with the gray triangular region
representing the coefficients spawned by $\alphamp_{\b k}$.}
\end{subfigure}\hspace{3mm}
\begin{subfigure}[t]{.3\linewidth}
\begin{tikzpicture}[scale=0.7]
\foreach  [evaluate = \i as \k using 4 - \i] \i in {0.25,0.5,...,4.0}{
  \foreach \j in {0,0.25,...,\i} {
    \draw (\i-\j,\j) node[fill=black!25,circle,scale=0.1] {};
  }
}
\draw[fill=black!25,opacity=0.25] (0,1) -- (3,1);
\draw[draw opacity=0.5,->] (0,0)--(4.75,0) node[right]{$\eta_1$};
\draw[draw opacity=0.5,->] (0,0)--(0,4.75) node[above]{$\eta_2$};
\draw[thick,opacity=0.5] (0,0)
  -- (4,0)
  -- (0,4)
  -- cycle;
\draw [dashed,fill=black!25,opacity=0.5] (2.25,0.25)
  -- +(1.5,0)
  -- +(0,1.5)
  -- cycle;
\draw [dashed,fill=black!25,opacity=0.5] (1.25,0.75)
  -- +(1.5,0)
  -- +(0,1.5)
  -- cycle;
\draw [thick,fill=black!45,opacity=0.95] (2.75,0.75)
  -- ++(0.5,0)
  -- ++(-2,2)
  -- ++(0,-0.5)
  -- cycle;
\draw [dashed] (2.25,0.25)
  -- +(1.5,0)
  -- +(0,1.5)
  -- cycle;
\node[
  circle,
  fill=black,
  scale=0.4,
  label={[text=black!40]west:$\betamp_{\b s^1}$}
] at (1.25,0.75) {};

\node[
  circle,
  fill=black,
  scale=0.4,
  label={[text=black!40]west:$\betamp_{\b s^0}$}
] at (2.25,0.25) {};
\draw[->] (2.75,2.5) node[right,align=left] {
  difference \\ induced by \\ $T(\betamp_{\b s^1}\,\b e'_{\b s^1})$
} -- (1.75, 2);
\draw (1.25, 7) -- (1.25, 5) -- (1.125, 5) -- (1.375, 4.75) --
      (1.625, 5) -- (1.5, 5) -- (1.5, 7) -- (1.25, 7);
\node at (2.8, 6) {\small Translate};
\end{tikzpicture}
\caption{Index space showing the translated coefficients in the compressed
representation around center $\b c_2$. Gray triangular regions denote the
coefficients spawned by $\betamp_{\b s^{0}}$ and $\betamp_{\b s^{1}}$,
while the darker shading represents the difference between the
compress-embed-translate and translate-compress-embed paths.}
\end{subfigure}\hspace{3mm}
\begin{subfigure}[t]{.3\linewidth}
\begin{tikzpicture}[scale=0.7]
\foreach  [evaluate = \i as \k using 4 - \i] \i in {0.25,0.5,...,4.0}{
  \foreach \j in {0,0.25,...,\i} {
    \draw (\i-\j,\j) node[fill=black!25,circle,scale=0.1] {};
  }
}
\draw[fill=black!25,opacity=0.25] (0,1) -- (3,1);
\draw[draw opacity=0.5,->] (0,0)--(4.75,0) node[right]{$\eta_1$};
\draw[draw opacity=0.5,->] (0,0)--(0,4.75) node[above]{$\eta_2$};
\draw[thick,opacity=0.5] (0,0)
  -- (4,0)
  -- (0,4)
  -- cycle;
\draw [dashed,fill=black!25,opacity=0.5] (2.25,0.25)
  -- +(1.5,0)
  -- +(0,1.5)
  -- cycle;
\draw [dashed,fill=black!25,opacity=0.5] (1.75,0.75)
  -- +(1.5,0)
  -- +(0,1.5)
  -- cycle;
\draw [dashed] (2.25,0.25)
  -- +(1.5,0)
  -- +(0,1.5)
  -- cycle;
\node[
  circle,
  fill=black,
  scale=0.4,
  label={[text=black!40]west:$\betamp_{\b s^1}$}
] at (1.75,0.75) {};

\node[
  circle,
  fill=black,
  scale=0.4,
  label={[text=black!40]west:$\betamp_{\b s^0}$}
] at (2.25,0.25) {};
\draw[->] (3,2.0) node[right] {$T(\betamp_{\b s^1}\,\b e'_{\b s^1})$} -- (2, 1.5);
\draw[->] (4,1) node[right] {$T(\betamp_{\b s^0}\,\b e'_{\b s^0})$} -- (3, 0.5);
\draw (1.5, 7) -- (1.5, 5) -- (1.375, 5) -- (1.625, 4.75) --
      (1.875, 5) -- (1.75, 5) -- (1.75, 7) -- (1.5, 7);
\node at (3, 6) {\small Translate};
\end{tikzpicture}
\caption{Index space showing the coefficients obtained by translating
$\betamp$ to center $\b c_2$, with gray triangular regions representing
the coefficients spawned by $\betamp_{\b s^{0}}$ and $\betamp_{\b s^{1}}$.}
\end{subfigure}
\caption{Multipole-to-multipole translation for uncompressed representation and compressed representation}
\label{fig:m2m_diagrams}

\end{figure}

\subsubsection{Error analysis of compressed multipole-to-multipole translation}
\label{sec:m2m_error_analysis}
For PDEs with only terms of equal order, it follows from
Proposition~\ref{prop:m2m_single_coeff_equal_order} that the
translated multipole expansion from $\b c_1$ to $\b c_2$ is exact:
\begin{align*}
  [\DD p G(\b x - \b c_2)]^T T \b \alpha
  &= \sum_{k=1}^{N(p)} [\DD p G(\b x - \b c_2)]^T T (\b\alpha_k \b e_k) \\
  &= \sum_{k=1}^{N(p)} [\DD p G(\b x - \b c_2)]^T T [E M^{T}(\b\alpha_k \b e_k)] \\ 
  &= [\DD p G(\b x - \b c_2)]^T T [E M^{T} \alpham]\\
  &= [\DD p G(\b x - \b c_2)]^T T [E \betam],
\end{align*}
where $\betam = M^{T} \alpham$ are the compressed multipole
coefficients around $\b c_1$.

For PDEs with varying-order terms, we define $\epsilon(\b x)$ to be
the difference between the potentials obtained by compressed and
uncompressed multipole-to-multipole translation from center $\b c_1$
to $\b c_2$:
\begin{align*}
 \epsilon(\b x) :=
 [\DD p G(\b x - \b c_2)]^T T [E M^{T} \alpham] - [\DD p G(\b x - \b c_2)]^T T \alpham.
\end{align*}
As a direct consequence of Proposition~\ref{prop:m2m_single_coeff}, the
contribution of a single coefficient $\alpham_k$ to this difference
is given by the following:
\begin{lemma}
\label{lemma:m2m_error_single_coeff_varying_order}
\lean{PDETranslations.m2m_error_single_coeff_varying_order}
Let $\b c_1, \b c_2 \in \mathbb{R}^d$, set $\b h = \b c_2 - \b c_1$, and let
$\alpham$ be the multipole coefficients around center $\b c_1$. Then, for
$k \in \{1, 2, \ldots, N(p)\}$,
\begin{align*}
 \epsilon_k (\b x)
  :=& [\DD p G(\b x - \b c_2)]^T T [E M^{T}(\alpham_k \b e_k)]
     - [\DD p G(\b x - \b c_2)]^T T (\alpham_k \b e_k)\\ 
  =& \sum_{\substack{s \in \{ 1, \dots  , |\b j| \} \\
     |\bnu(\b j_s)| < |\bnu(k)|}} M_{k,s}\alpham_k
    \left. \left( 
     \sum_{\ell = N(p - |\bnu(k)|) + 1}^{N(p - |\bnu(\b j_s)|)}
      \frac{\b h ^{\bnu(\ell)}}{\bnu(\ell)!}
      \D z ^{\bnu(\ell) + \bnu(\b j_s)} G(\b z)
    \right) \right|_{\b z = \b x - \b c_2}.  
\end{align*}
\lean{PDETranslations.m2mTranslationError}
\end{lemma}
A novel difficulty presented by the case of PDEs with terms of varying order
is that nonzero coefficients $M_{k,s}$ with $|\bnu(k)| > |\bnu(\b j_s)|$
will occur (cf. Corollary~\ref{corollary:deriv_less_order}).  Necessarily, the $M_{k,s}$ have spatial scaling
behavior depending on $|\bnu(k)|-|\bnu(\b j_s)|$.  To capture the dependency of the
coefficients $M_{k,s}$ on spatial scaling, we introduce the following assumption.
\begin{assumption}\label{assumption:v}
Let $a > 0$. We assume that
there exists a constant $K'$ such that for all $k \in \{ 1, \dots  , N(p) \}
$ and $s \in \{ 1, \dots  ,|\b j|\}$ with $|\bnu(k)| \geq |\bnu(\b j_s)|$,
\begin{equation}
|M_{k,s} a^{|\bnu(k)| - |\bnu(\b j_s)|}| \leq K'.
\label{eq:assumption_v}
\end{equation}
\end{assumption}
In an FMM, $a$ is typically taken to be the diameter of the root box of
the hierarchy, which ensures $\norm{\b x - \b c_2} \leq a$ for all
expansion centers $\b c_2$ and targets $\b x$.
If $M$ is viewed as a fixed finite matrix for a given PDE and order $p$, it is clear that a
constant $K'$ satisfying~\eqref{eq:assumption_v} exists unconditionally.
A non-vacuous statement is obtained when viewing $M$ as a matrix dependent on parameters
in the PDE (such as the Helmholtz wave number), allowing error bounds
analogous to the following Theorem~\ref{thm:varying} to hold across
a range of wave numbers. Note though that in this contribution, we adopt the
point of view that $G$ and the PDE operator $\mathcal L$
are fixed, making \eqref{eq:assumption_v} true by default.
We leave consideration of a bound uniform across PDEs with varying
constants for future work.

\begin{remark}
For the Helmholtz equation with wave number $\kappa$, it can be shown that the dependency of $M_{k,s}$
on $\kappa$ can be isolated as follows:
\[
  M_{k, s} = \kappa^{|\bnu(k)| - |\bnu(\b j_s)|} D_{k,s}
\] for all $\kappa> 0$, where $D_{k, s}$ is independent of $\kappa$.
Thus, assuming
\begin{equation}
  a\kappa < 1
  \label{eq:assumption_helmholtz}
\end{equation}
leads to
\[
  |M_{k,s} a^{|\bnu(k)| - |\bnu(\b j_s)|}|
  = | (a\kappa)^{|\bnu(k)| - |\bnu(\b j_s)|}  D_{k, s}|
  \le \max_{k,s} |D_{k, s}|= K',
\]
establishing a version of \eqref{eq:assumption_v}. In
Theorem~\ref{thm:varying}, $a$ provides an upper bound on the distance
between the target point and the center of the translated multipole
expansion.
Operationally, \eqref{eq:assumption_helmholtz} limits the root-box size
relative to the Helmholtz wavelength. Wide-band FMMs impose a similar
restriction on the boxes in which their `low-frequency' expansions
are used, cf.~\citep{cheng_wideband_2006}.
In Section~\ref{sec:m2m-experiment}, we further present numerical
evidence suggesting that the accuracy impact of
\eqref{eq:assumption_helmholtz} is no worse than that imposed by the use of
Cartesian expansions in the first place.
\end{remark}
\begin{theorem}
  \label{thm:varying}
  \lean{PDETranslations.m2m_error_varying_order}
  Let $\alpham$ be the $p$-th-order (uncompressed) multipole expansion
  coefficients induced by point sources within a ball of radius $R_1$
  centered at $\b c_1$. Let $T$ be the multipole-to-multipole translation operator from
  center $\b c_1$ to $\b c_2$, and $\betam = M^{T} \alpham$. \\
  Denote $R_2:=R_1 + \|\b c_2 - \b c_1\|$, then for any target point $\b x$
  with $R_2 < \norm{\b x - \b c_2} \leq a$ (cf. Figure~\ref{fig:experiment_points}), 
  \begin{align*} 
    |\epsilon(\b x)|
    &= \left|\sum_{i=1}^{N(p)} [T E \betam]_{i} \left. \D z^{\bnu(i)} G(\b z) \right|_{\b z = \b x - \b c_2}
      - \sum_{i=1}^{N(p)} [T \alpham]_{i} \left. \D z^{\bnu(i)} G(\b z) \right|_{\b z = \b x - \b c_2} \right|\\ 
    &\leq \frac{A''}{\|\b x - \b c_2\|^{p'}} \left( \frac{R_2}{\|\b x - \b c_2\|} \right)^{p+1},
  \end{align*}
  where $A''$ is independent of $\b x$ and $R_2$, but may depend on $p$,
  $G$, the PDE, and the source strengths.
\end{theorem}

\begin{proof}
Let $\mathcal{I}$ be the index set of all source points $\{ \b y_i \}_{i
\in \mathcal{I}}$ in a ball centered at $\b c_1$ with radius $R_1$, and let
$w(\b y_i)$ be the strength of the source at $\b y_i$.  Suppose $\alpham$
is the $p$-th-order (uncompressed) multipole expansion
coefficients generated by $\{(\b y_i, w(\b y_i))\}_{i \in \mathcal{I}}$ around
$\b c_1$, i.e.,
\[
\alpham_k = \sum_{i \in \mathcal{I}} w(\b y_i) \frac{(\b
c_1 - \b y_i)^{\bnu(k)}}{\bnu(k)!}.
\]

For
$k \in \{1, 2, \ldots, N(p)\}$, it follows from
Lemma~\ref{lemma:m2m_error_single_coeff_varying_order} that
\allowdisplaybreaks
\begin{align*}
  |\epsilon_k (\b x)|
  &= \left| \sum_{\substack{s \in \{ 1, \dots  , |\b j| \} \\
    |\bnu(\b j_s)| < |\bnu(k)|}} M_{k,s}\alpham_k \left. \left[
    \sum_{\ell = N(p - |\bnu(k)|) + 1}^{N(p - |\bnu(\b j_s)|)}
    \frac{(\b c_2-\b c_1)^{\bnu(\ell)}}{\bnu(\ell)!}
    \D z ^{\bnu(\ell) + \bnu(\b j_s)} G(\b z)
    \right] \right|_{\b z = \b x - \b c_2} \right| \\
  &\leq
    \sum_{\substack{r \in \{ 1, \dots  , |\b j| \} \\
      |\bnu(\b j_r)| < |\bnu(k)|}}
    |M_{k,r}| |\alpham_k|
    \sum_{n=p-|\bnu(k)|+1}^{p-|\bnu(\b j_r)|}
    \left|
      \sum_{|\b q|=n}
        \left.\frac{\D z^{\b q+\bnu(\b j_r)}G(\b z)}{\b q!}
        \right|_{\b z=\b x-\b c_2}
        (\b c_2-\b c_1)^{\b q}
    \right| \\
  &\stackrel{\eqref{eq:assumption_g}}{\leq}
    \sum_{\substack{r \in \{ 1, \dots  , |\b j| \} \\
      |\bnu(\b j_r)| < |\bnu(k)|}}
    |M_{k,r}| |\alpham_k|
    \frac{C_{G,p}}{\|\b x-\b c_2\|^{|\bnu(\b j_r)|+p'}}
    \sum_{n=p-|\bnu(k)|+1}^{p-|\bnu(\b j_r)|}
      (1+n)^{s_p}
      \left(\frac{\|\b c_2-\b c_1\|}
        {\|\b x-\b c_2\|}\right)^n.
\end{align*}
Let $A:=\sum_{i\in\mathcal I}|w(\b y_i)|$ and
$\rho:=\|\b x-\b c_2\|$.  Since
$|\alpham_k|\leq A R_1^{|\bnu(k)|}/\bnu(k)!$,
$R_1\leq R_2$, and $\|\b c_2-\b c_1\|\leq R_2$, it follows that
\begin{align*}
  |\epsilon_k (\b x)|
  &\leq \frac{A C_{G,p}}{\bnu(k)!}
    \sum_{\substack{r \in \{ 1, \dots  , |\b j| \} \\
      |\bnu(\b j_r)| < |\bnu(k)|}}
    \sum_{n=p-|\bnu(k)|+1}^{p-|\bnu(\b j_r)|}
    (1+n)^{s_p}
    \left(\frac{R_2}{\rho}\right)^{n+|\bnu(k)|}
    |M_{k,r}|\rho^{|\bnu(k)|-|\bnu(\b j_r)|}\rho^{-p'}\\
  &\stackrel{\eqref{eq:assumption_v}}{\leq}
    \frac{p(p+1)^{s_p} |\b j| A K' C_{G,p}}{\bnu(k)!}
    \left(\frac{R_2}{\rho}\right)^{p+1}\rho^{-p'}.
\end{align*}
\lean{PDETranslations.m2m_weighted_tail_bound}
For the last inequality, note that $n+|\bnu(k)|\geq p+1$,
$R_2/\rho<1$, and
$|M_{k,r}|\rho^{|\bnu(k)|-|\bnu(\b j_r)|}\leq K'$ because
$\rho\leq a$.  Hence,
\begin{align*}
  |\epsilon(\b x)|
  &\leq \sum_{k=1}^{N(p)} |\epsilon_k (\b x)| \\
  &\leq e^d p(p+1)^{s_p} |\b j| A K' C_{G,p}
    \left(\frac{R_2}{\|\b x-\b c_2\|}\right)^{p+1}
    \|\b x-\b c_2\|^{-p'},
\end{align*}
where we used $\sum_{k=1}^{N(p)}1/\bnu(k)!\leq e^d$.
\end{proof}
\begin{remark}
In the classical FMM framework \citep{greengard1986fast},
uncompressed multipole-to-multipole translation is equivalent to
directly forming the multipole expansion around the new center.
For PDEs with terms of varying order, our compressed scheme generally
produces a different approximation than uncompressed translation,
and Theorem~\ref{thm:varying} bounds the difference.
Combining this bound with the Taylor truncation estimate
\eqref{eq:mpole_truncated} shows that the error relative to the exact
potential is $\bigO{(R_2/\|\b x-\b c_2\|)^{p+1}}$ as $R_2\to0$,
with $p$ and $\|\b x-\b c_2\|$ fixed.

Often, we are interested in the asymptotic error behavior for a fixed
separation ratio $R_2/\|\b x-\b c_2\|$ as $p\to\infty$. 
Theorem~\ref{thm:varying} does not provide this; we leave
theoretical consideration of this for future work.
Nonetheless, numerical experiments in Figure~\ref{fig:m2m_accuracy_compare} suggest that
the compressed multipole-to-multipole translation is comparable in
accuracy to the uncompressed translation.
\end{remark}
\begin{remark}
In Theorem~\ref{thm:varying}, we assume the multipole coefficients
arise from point sources. After applying the compressed
multipole-to-multipole translation, the resulting coefficients can again be 
realized as arising from some effective discrete source distributions inside the
`parent' box. This is guaranteed by the existence of a multivariate polynomial
interpolant on a unisolvent set of points (see Chapter~2 of
\citep{wendlandScatteredDataApproximation2004a}).
Consequently, if an additional compressed multipole-to-multipole translation is
performed, the estimate of Theorem~\ref{thm:varying}
applies recursively. This implies that as long as the difference at each
level is controlled according to Theorem~\ref{thm:varying}, the cumulative
effect of repeated compressed translations across the hierarchy remains
bounded.
\end{remark}

\subsubsection{Asymptotically Faster Multipole-to-Multipole Translation}
\label{section:m2m_compute_faster}

Using the multipole-to-multipole translation procedure as described
thus far,
shifting the center of a multipole expansion requires calculating
each of the $\bigO{p^d}$ coefficients at the new center from
$\bigO{p^{d-1}}$ coefficients around the old center and then
compressing them. A straightforward approach might use $\bigO{p^{d-1}}$
operations per target coefficient, resulting in an overall algorithm with
$\bigO{p^{2d-1}}$ operations. This cost can be reduced by
applying sum factorization methods, i.e., by reordering the nested summations and
storing intermediate results (for example, see
\cite{orszag1980spectral,melenkFullyDiscreteHpfinite2001,eibnerFastAlgorithmsSetting2005}).
A straightforward generalization of such an approach in
\cite{greengard1988rapid} yields an algorithm with $\bigO{p^{d+1}}$
operations for the \emph{uncompressed} multipole-to-multipole translation.
In the compressed case, one may observe that $[E\betam]_{i} = 0$ when $i \notin \b j$ and the
multi-index set $\bnu(\b j)$ can be separated into a union of $\bigO 1$
hyperplanes (see Proposition~\ref{prop:j-character}). Exploiting this
structure, the compressed multipole-to-multipole translation can be
achieved in $\bigO{p^{d}}$ operations.

\begin{theorem}
Translating $p$-th-order multipole expansion coefficients in
compressed representation around center $\b c_1$ to (uncompressed) multipole expansion
coefficients around $\b c_2$ can be achieved with $\bigO{p^{d}}$
operations.
\end{theorem}
\begin{proof}
Let $\b c_1, \b c_2 \in \mathbb{R}^d$ and set $\b h = \b c_2 - \b c_1$. Let
$\betam$ be the multipole expansion coefficients in compressed
representation around center $\b c_1$. For $\b \eta \in
\mathcal{M}(p)$, the (uncompressed)
translated multipole expansion coefficient $\b\sigma'_{\b \eta}$ from 
$\b c_1$ to $\b c_2$ is given by
\[
\b \sigma'_{\b \eta} =  \sum_{\substack{i \in \{ 1, \dots  ,N(p) \}\\ \bnu(i)\leq \b
\eta }} \frac{\b h^{\b \eta - \bnu(i)}}{(\b \eta - \bnu(i))!} [E \betam ]_i
= \sum_{\substack{i \in \{ 1, \dots  ,N(p) \}\\ i\in \b j,\bnu(i)\leq \b
\eta }} \frac{\b h^{\b \eta - \bnu(i)}}{(\b \eta - \bnu(i))!} [E \betam ]_i 
=  \sum_{\substack{\b \zeta \in \bnu(\b j)\\ \b \zeta \leq \b \eta}}
\frac{\b h^{\b \eta - \b \zeta}}{(\b \eta - \b \zeta)!}\betamp_{\b \zeta}. 
\]
Denote $\b t = \bnu(h(1))$. By Proposition~\ref{prop:j-character}, $\bnu(\b
j)$ can be represented as
\begin{align*}
\bnu(\b j)
&= \{\b \zeta \in \mathcal M(p): \exists \ell \in \{1,\dots,d\}
    \text{ such that } \b \zeta_\ell < \b t_\ell\}\\
&= \bigcup_{\ell=1}^{d}
    \{\b \zeta \in \mathcal M(p): \b \zeta_\ell < \b t_\ell,
    \b \zeta_i \geq \b t_i \text{ for }1\leq i<\ell\}.
\end{align*}
Thus,
\[
 \b \sigma'_{\b \eta}= \sum_{\ell =1}^{d}
\sum_{\substack{ \b \zeta \in \mathcal{M}(p),\ \b \zeta \leq \b \eta \\
\b \zeta_\ell < \b t_\ell \\
\b \zeta_i \geq \b t_i\ (1\leq i<\ell)}}
\frac{\b h^{\b \eta - \b \zeta}}{(\b \eta - \b \zeta)!}\betamp_{\b \zeta}
= \sum_{\ell =1}^{d}
\underbracket{\left( \sum_{\substack{ \b \zeta \in \mathcal{M}(p),\ \b \zeta \leq \b \eta \\
\b \zeta_\ell < \b t_\ell \\
\b \zeta_i \geq \b t_i\ (1\leq i<\ell)}}
\left(\prod_{s = 1}^{d} \frac{\b h_s^{\b \eta_s - \b \zeta_s}}{(\b \eta_s - \b \zeta_s)!}\right)
\betamp_{\b \zeta} \right)}_{:=S_{\b \eta}^{\ell}}.
\]
For each $\ell \in \{1, \dots , d\}$, we will show that the cost of
computing $S_{\b \eta}^{\ell}$ for all $\b \eta \in \mathcal{M}(p)$ is
$\bigO{p^{d}}$. Consequently, the total cost of computing $\b \sigma'$
remains $\bigO{p^{d}}$ as we are summing over $d$ terms.

Without loss of generality, we consider the case of $\ell = 1$. Then, for $\b
\eta \in \mathcal{M}(p)$,
\begin{align*}
 S_{\b \eta}^{1} &= \sum_{\substack{ \b \zeta \in \mathcal{M}(p) \\ 
 \b \zeta_1 < \b t_1 \\ \b \zeta \leq \b \eta }} \left(\prod_{s = 1}^{d} 
 \frac{\b h_s^{\b \eta_s - \b \zeta_s}}{(\b \eta_s - \b \zeta_s)!}\right)\betamp_{\b \zeta}\\  
 &= \sum_{\substack{\b \zeta_1 \in \{ 0, \dots  ,p \} \\ 
 \b \zeta_1 \leq \min\{ \b t_{1} - 1, \b \eta_1\}}} 
 \frac{\b h_1^{\b \eta_1 - \b \zeta_1} }{(\b \eta_1 - \b \zeta_1)!} 
 \sum_{\substack{\b \zeta_2 \in \{ 0, \dots  ,p \} \\ \b \zeta_2 \leq \b \eta_2}} 
 \frac{\b h_2^{\b \eta_2 - \b \zeta_2} }{(\b \eta_2 - \b \zeta_2)!} \cdots 
 \sum_{\substack{\b \zeta_d \in \{ 0, \dots  , p \} \\ \b \zeta_d \leq \b  \eta_d }} 
 \frac{\b h_d^{\b \eta_d - \b \zeta_d}}{(\b \eta_d - \b \zeta_d)!} 
 \betamp_{(\b \zeta_1, \b \zeta_2, \dots , \b \zeta_d)}.
\end{align*}
We will compute $S_{\b \eta}^{1}$ by precomputing and storing intermediate
values starting from the innermost sum and working our way outward.

Denote $T^{d+1}_{\b \zeta}=\betamp_{\b \zeta}$. 
If $d\geq 2$, for $(\b \zeta_1, \dots , \b \zeta_{d-1}, \b \eta_d) \in \mathcal{M}(p)$
and $\b \zeta_1 < \b t_1$, define
\begin{align*}
 T^d_{(\b \zeta_1, \dots , \b \zeta_{d-1}, \b \eta_d)} 
 = \sum_{\substack{ \b \zeta_d \in \{ 0, \dots , p \} \\ \b \zeta_d \leq \b \eta_d}}
 \frac{\b h_d^{\,\b \eta_d - \b \zeta_d}}{(\b \eta_d - \b \zeta_d)!} 
 T^{d+1}_{(\b \zeta_1, \b \zeta_2, \dots , \b \zeta_d)}.
\end{align*}
\noindent
The factorial and power terms in this sum can be computed efficiently using
a recurrence relation.
By initializing $w_{\b \eta_d} =1$ and using the recurrence
\[
w_{\b \eta_{d}-i} = w_{\b \eta_{d}-(i-1)} \frac{\b h_d}{i},
\quad i = 1, \dots , \b \eta_d,
\]
then $w_{\b \zeta_d}
= \b h_d^{\,\b \eta_d-\b \zeta_d}/(\b \eta_d-\b \zeta_d)!$,
and each term can be computed in $\bigO 1$ operations. As a result,
the cost of computing all
$T^d_{(\b \zeta_1, \dots , \b \zeta_{d-1}, \b \eta_d)}$ for a fixed $\b \eta_d$ is
\[
\bigO{\b \eta_d+1} \cdot \bigO{c(p+1)^{d-2}} = \bigO{(\b \eta_d+1) p^{d-2}}
\qquad (p \to \infty).
\]
\noindent
Summing over all $\b \eta_d = 0, \dots, p$, the total cost becomes $\bigO
{p^{d}}$. We then store the precomputed values $T^d_{(\b \zeta_1, \dots , \b
\zeta_{d-1}, \b \eta_d)}$ for use in the next step.

If $d\geq 3$, for
$(\b \zeta_1, \dots , \b \zeta_{d-2}, \b \eta_{d-1}, \b \eta_d) \in \mathcal{M}(p)$
and $\b \zeta_1 < \b t_1$, we define
\begin{align*}
T^{d-1}_{(\b \zeta_1, \dots , \b \zeta_{d-2}, \b \eta_{d-1}, \b \eta_d)} 
= \sum_{\substack{ \b \zeta_{d-1} \in \{ 0, \dots , p \} \\ \b \zeta_{d-1} \leq \b \eta_{d-1}}}
    \frac{\b h_{d-1}^{\,\b \eta_{d-1} - \b \zeta_{d-1}}}{(\b \eta_{d-1} - \b \zeta_{d-1})!} 
    T^d_{(\b \zeta_1, \dots , \b \zeta_{d-1}, \b \eta_d)}.
\end{align*}
\noindent
Similarly, the cost of computing all
$T^{d-1}_{(\b \zeta_1, \dots , \b \zeta_{d-2}, \b \eta_{d-1}, \b \eta_d)}$
for a fixed $\b \eta_{d-1}$ is $\bigO{(\b \eta_{d-1}+1)\, p^{d-2}}$. As a
consequence, the total cost of computing all $T^{d-1}_{(\b \zeta_1, \dots , \b \zeta_{d-2},
\b \eta_{d-1}, \b \eta_d)}$ remains $\bigO{p^d}$. Again,
we store the precomputed values $T^{d-1}_{(\b \zeta_1, \dots , \b \zeta_{d-2},
\b \eta_{d-1}, \b \eta_d)}$. We can repeat this process recursively to 
compute $T^{d-2}, \dots, T^2$, each time
precomputing and storing intermediate values for use in the next step. During
each step, the total cost of computing all intermediate values is
$\bigO{p^d}$.

Finally, for every
$(\b \eta_1, \b \eta_2, \dots , \b \eta_d) \in \mathcal{M}(p)$, we compute
\begin{align*}
S^{1}_{\b \eta} = \sum_{\substack{\b \zeta_1 \in \{ 0, \dots , p \} \\ \b \zeta_1 \leq \min\{\b t_1 -1, \b \eta_1\}}}
\frac{\b h_1^{\b \eta_1 - \b \zeta_1}}{(\b \eta_1 - \b \zeta_1)!}
T^{2}_{(\b \zeta_1, \b \eta_2, \dots , \b \eta_d)}.
\end{align*}
There are $\bigO{p^d}$ such sums, and each sum has at most
$\b t_1 \leq c$ terms.
Combining all steps, the total cost of computing $S_{\b
\eta}^{1}$ for all $\b \eta \in \mathcal{M}(p)$ is $\bigO{p^{d}}$.
\end{proof}

\subsection{Local-to-Local Translation}
\label{sec:l2l}
Recall from \eqref{eq:local} that the (uncompressed) $p$-th-order local
expansion around center $\b c_1$ induced by some source points $\b y$ with $\|\b y - \b
c_1\| > R$ (valid for $\|\b x - \b c_1\| < R$) is given by
\[
 G(\b x - \b y) \approx \sum_{\ell =1}^{N(p)} \b\gamma_\ell (\b x - \b c_1)^{\bnu(\ell)},
\]
where $\b \gamma_{\ell}$ are the local expansion coefficients around $\b
c_1$.
\lean{PDETranslations.localCoefficients}
We can shift the expansion center to $\b c_2 \in \mathbb{R}^{d}$ using the
(multi-index) binomial theorem with no additional error.

Denote $\b h= \b c_2 -
\b c_1$. Then,
\begin{align*}
  \sum_{\ell =1}^{N(p)} \b\gamma_\ell (\b x - \b c_1)^{\bnu(\ell)}
    &=  \sum_{\ell=1}^{N(p)} \b\gamma_\ell (\b x - \b c_2 + \b c_2 - \b c_1)^{\bnu(\ell)}\\
    &= \sum_{\ell=1}^{N(p)} \b\gamma_\ell
       \left(\sum_{\substack{i \in \{1, 2, \ldots, N(p)\} \\ \bnu(i) \leq \bnu(\ell)}}
       \binom{\bnu(\ell)}{\bnu(i)} \b h^{\bnu(\ell) - \bnu(i)}
       (\b x - \b c_2)^{\bnu(i)} \right)\\
    &= \sum_{i = 1}^ {N(p)} \left(
       \sum_{\substack{\ell \in \{1, 2, \ldots, N(p)\} \\ \bnu(i) \leq \bnu(\ell)}}
       \b\gamma_\ell \binom{\bnu(\ell)}{\bnu(i)} \b h^{\bnu(\ell) - \bnu(i)}
       \right)(\b x - \b c_2)^{\bnu(i)}\\
    &= \sum_{i = 1}^ {N(p)}  \b\mu_i (\b x - \b c_2)^{\bnu(i)},
\end{align*}
where $\b\mu_i$ are the (translated) local expansion coefficients
from $\b c_1$ to $\b c_2$ defined as
\begin{align}
  \b \mu_i &= \sum_{\substack{\ell \in \{1, 2, \ldots, N(p)\} \\ \bnu(i) \leq \bnu(\ell)}}
  \b \gamma_\ell \binom{\bnu(\ell)}{\bnu(i)} \b h^{\bnu(\ell) - \bnu(i)},
  \quad i \in \{1, 2, \ldots, N(p)\}.
  \label{eq:l2l-expr}
\end{align}
\lean{PDETranslations.l2lOperator}
As in the multipole-to-multipole translation, we consider two cases. In
Subsection~\ref{sec:l2l_error_analysis}, we show that for PDEs with only
terms of equal order, $\b \mu$ satisfies the compressibility condition
\eqref{eq:compressed-local-form}. However, in general,
\eqref{eq:compressed-local-form} does not hold for PDEs with
varying-order terms. Regardless, we apply the compression scheme and show
that the difference between the translated compressed local expansion and
the translated uncompressed local expansion obeys a bound that is
of the same order as the truncation error of the original local
expansion at a fixed expansion order $p$.
Finally, in Subsection~\ref{sec:l2l_compute_faster}, we provide an
algorithm to perform the computation with low asymptotic cost.

\subsubsection{Error analysis of compressed local-to-local translation}
\label{sec:l2l_error_analysis}

Let $\b \mu$ be the $p$-th-order
(translated) local expansion coefficients from $\b c_1$ to $\b c_2$ as
defined in \eqref{eq:l2l-expr}. We define the \emph{compressed} local expansion
coefficients $\b \mu^{c}$ as
\begin{align}
  \b \mu^{c}_i = \frac{[M (\b \mu_{\b j} \b \nu(\b j)!)]_{i}}{\b \nu(i)!}
  \quad i\in \{1, 2, \dots, N(p)\} \label{eq:l2l-expr-compressed}.
\end{align}
\lean{PDETranslations.compressedLocalCoefficients}
\begin{lemma}
\label{lemma:l2l_difference}
\lean{PDETranslations.l2l_difference}
Let $\b c_1, \b c_2 \in \mathbb{R}^d$ and $\b h = \b c_2 - \b c_1$, and let $\b \gamma$ be the
$p$-th-order local expansion coefficients around center $\b c_1$
induced by a source point $\b y$ with unit strength. Let $\b \mu$ denote
the translated local expansion coefficients from $\b c_1$ to $\b c_2$ as in
\eqref{eq:l2l-expr}, and let $\b \mu^{c}$ be the compressed local
expansion coefficients from $\b c_1$ to $\b c_2$ as in
\eqref{eq:l2l-expr-compressed}. Then, for
$i \in \{1, 2, \ldots, N(p)\}$,
\begin{align}
\b \mu_{i}^{c} - \b \mu_{i}
= \frac{1}{\bnu(i)!}
\sum_{\substack{s \in \{1,\dots,|\b j|\} \\ |\bnu(\b j_s)| < |\bnu(i)|}}
M_{i,s}
\left[
  \sum_{q = N(p-|\bnu(i)|)+1}^{N(p-|\bnu(\b j_s)|)}
  \left.
  \partial_{\b z}^{\bnu(q) + \bnu(\b j_s)} G(\b z)
  \right|_{\b z = \b c_1 - \b y}
  \frac{\b h^{\bnu(q)}}{\bnu(q)!}
\right].
\label{eq:l2l_difference}
\end{align}
\end{lemma}
\begin{proof}
Let $i \in \{1, 2, \ldots, N(p)\}$. Then,
\begin{align*}
\b \mu_i &= \sum_{\substack{\ell \in \{1, 2, \ldots, N(p)\} \\ \bnu(\ell) \geq \bnu(i)}}
  \b \gamma_\ell \binom{\bnu(\ell)}{\bnu(i)} \b h^{\bnu(\ell) - \bnu(i)}  \\
  &=\frac{1}{\bnu(i)!}
    \sum_{\substack{\ell \in \{1, 2, \ldots, N(p)\} \\ \bnu(\ell) \geq \bnu(i)}}
    \left. \left[
    \partial_{\b z}^{\bnu(\ell)} G(\b z) \right]
    \right|_{\b z = \b c_1 - \b y}
    \frac{\b h^{\bnu(\ell) - \bnu(i)}}{(\bnu(\ell) - \bnu(i))!} \\
  &=\frac{1}{\bnu(i)!} \sum_{q = 1}^{N(p -|\bnu(i)|)}
    \left. \left[
    \partial_{\b z}^{\bnu(q) + \bnu(i)} G(\b z) \right]
    \right|_{\b z = \b c_1 - \b y}
    \frac{\b h^{\bnu(q)}}{\bnu(q)!} \\
  &= \frac{1}{\bnu(i)!} \sum_{q = 1}^{N(p- |\bnu(i)|)}
      \left. \left[
      \sum_{\substack{s \in \{ 1, \dots  ,|\b j| \} \\ |\bnu(\b j_s)| \leq |\bnu(i)|}}
      M_{i,s}
      \partial_{\b z}^{\bnu(q) + \bnu(\b j_s)} G(\b z) \right]
      \right|_{\b z = \b c_1 - \b y}
      \frac{\b h^{\bnu(q)}}{\bnu(q)!}\\
  &= \frac{1}{\bnu(i)!}
      \sum_{\substack{s \in \{ 1, \dots  ,|\b j| \} \\ |\bnu(\b j_s)| \leq |\bnu(i)|}}
      M_{i,s}
      \left[
        \sum_{q = 1}^{N(p- |\bnu(i)|)}
        \left.
        \partial_{\b z}^{\bnu(q) + \bnu(\b j_s)} G(\b z)
        \right|_{\b z = \b c_1 - \b y}
        \frac{\b h^{\bnu(q)}}{\bnu(q)!}
      \right].
\end{align*}
On the other hand,
\begin{align*}
 \b \mu_i^{c}
 &= \frac{1}{\bnu(i)!}
    \sum_{\substack{s \in \{1, \dots, |\b j|\} \\ |\bnu(\b j_s)| \le |\bnu(i)|}}
    M_{i,s} \b \mu_{\b j_s} \bnu(\b j_s)! \\
 &= \frac{1}{\bnu(i)!}
    \sum_{\substack{s \in \{1, \dots, |\b j|\} \\ |\bnu(\b j_s)| \le |\bnu(i)|}}
    M_{i,s} \, \bnu(\b j_s)! \left[
      \sum_{\substack{\ell \in \{1,\dots, N(p)\} \\ \bnu(\ell) \ge \bnu(\b j_s)}}
      \b \gamma_\ell \binom{\bnu(\ell)}{\bnu(\b j_s)}
      \b h^{\bnu(\ell) - \bnu(\b j_s)}
    \right] \\
 &= \frac{1}{\bnu(i)!}
    \sum_{\substack{s \in \{1, \dots,|\b j|\} \\ |\bnu(\b j_s)| \le |\bnu(i)|}}
    M_{i,s}
    \left[
      \sum_{q = 1}^{N(p - |\bnu(\b j_s)|)} \gamma_{\bnu^{-1}(\bnu(q) + \bnu(\b j_s))}
      \frac{(\bnu(q) + \bnu(\b j_s))! \b h^{\bnu(q)}}{\bnu(q)!}
    \right] \\
  &= \frac{1}{\bnu(i)!}
    \sum_{\substack{s \in \{1, \dots,|\b j|\} \\ |\bnu(\b j_s)| \le |\bnu(i)|}}
    M_{i,s}
    \left[
      \sum_{q = 1}^{N(p - |\bnu(\b j_s)|)}
      \left.
      \partial_{\b z}^{\bnu(q) + \bnu(\b j_s)} G(\b z)
      \right|_{\b z = \b c_1 - \b y}
      \frac{\b h^{\bnu(q)}}{\bnu(q)!}
    \right].
\end{align*}
Taking the difference of the above two expressions yields
\begin{align*}
\b \mu_i^{c} - \b \mu_i
=& \frac{1}{\bnu(i)!}
\sum_{\substack{s \in \{1,\dots,|\b j|\} \\ |\bnu(\b j_s)| < |\bnu(i)|}}
M_{i,s}
\left[
  \sum_{q = N(p-|\bnu(i)|)+1}^{N(p-|\bnu(\b j_s)|)}
  \left.
  \partial_{\b z}^{\bnu(q) + \bnu(\b j_s)} G(\b z)
  \right|_{\b z = \b c_1 - \b y}
  \frac{\b h^{\bnu(q)}}{\bnu(q)!}
\right].
\end{align*}
\end{proof}
For PDEs with only terms of equal order, by Theorem~\ref{thm:m_equal_order_property}, $M_{i,s} = 0$ if $|\bnu(\b j_s)| <
|\bnu(i)|$. Therefore, the above lemma implies $\b \mu^{c} = \b \mu$. Consequently,
\begin{align*}
 \sum_{\ell=1}^{N(p)} \gamma_{\ell} (\b x - \b c_1)^{\bnu(\ell)}
  =\sum_{i=1}^{N(p)} \b \mu_i (\b x - \b c_2)^{\bnu(i)}
  = \sum_{i=1}^{N(p)}
  \frac{[M (\b \mu_{\b j}   \b \nu(\b j)!)]_{i}}{\b \nu(i)!} (\b x - \b c_2)^{\b \nu(i)}.
\end{align*}
For PDEs with varying-order terms, our compressed local expansion scheme obeys
a bound of the same order as the truncation error of the
original local expansion at a fixed expansion order $p$.

\begin{theorem}
\label{thm:l2l_error}
Let $\b \gamma$ be the $p$-th-order uncompressed local expansion coefficients
induced by point sources $\{(\b y_k, w(\b y_k))\}_{k \in \mathcal{K}}$
outside a ball of radius $R_1$ centered at $\b c_1$ for some finite index set $\mathcal{K}$.
Assume the minimal distance from any source to center
$\b c_1$ is at least $R_{\mathrm{loc}}$ with $a \geq R_{\mathrm{loc}} > R_1$, where $a$ is the constant from Assumption~\ref{assumption:v}.

Let $\b c_2$ be a new expansion center with expansion radius
$R_2$ satisfying $\|\b c_2 - \b c_1\| + R_2 \leq R_1$, and let $\b \mu$ and
$\b \mu^{c}$
denote the translated local expansion coefficients from $\b c_1$ to $\b
c_2$ in uncompressed and compressed representation, respectively.
Let $\b x$ be any target point with $\|\b x - \b c_2\|< R_2$.

Then, if $p'\geq 0$ (cf.~\eqref{eq:assumption_g}),
\[
\left|
  \sum_{i=1}^{N(p)} \b \mu^{c}_i (\b x - \b c_2)^{\bnu(i)}
  - \sum_{i=1}^{N(p)} \b \mu_i (\b x - \b c_2)^{\bnu(i)}
\right| \leq \frac{A}{R_{\mathrm{loc}}^{p'}}
\left(\frac{R_1}{R_{\mathrm{loc}}}\right)^{p+1},
\]
\lean{PDETranslations.l2l_error_nonneg}
where
\[
  A = e^{d} p (p+1)^{s_p} |\b j| K' C_{G,p} \sum_{k \in \mathcal{K}} |w(\b y_k)|,
\]
with $s_p$, $C_{G,p}$ defined in~\eqref{eq:assumption_g} for $m=p$ and $K'$ from~\eqref{eq:assumption_v}.

If, on the other hand, $p'<0$, and if $\|\b c_1-\b y_k\|\leq a$ for all $k\in\mathcal K$
(with $a$ as in Assumption~\ref{assumption:v}), then
\[
\left|
  \sum_{i=1}^{N(p)} \b \mu^{c}_i (\b x - \b c_2)^{\bnu(i)}
  - \sum_{i=1}^{N(p)} \b \mu_i (\b x - \b c_2)^{\bnu(i)}
\right| \leq \frac{A}{a^{p'}} \left(\frac{R_1}{R_{\mathrm{loc}}}\right)^{p+1}.
\]
\lean{PDETranslations.l2l_error_neg}
\end{theorem}
\begin{proof}
Let $\mathcal{K}$ be the index set of all source points $\{\b y_k\}_{k \in \mathcal{K}}$ outside a ball centered at $\b c_1$ with radius $R_1$, and let $w(\b y_k)$ denote the strength of the source at $\b y_k$. Suppose $\b \gamma$ are the $p$-th-order local expansion coefficients generated by
$\{(\b y_k, w(\b y_k))\}_{k \in \mathcal{K}}$ around $\b c_1$, i.e.,
\[
  \b \gamma_i = \sum_{k \in \mathcal{K}} w(\b y_k)
  \frac{\D z^{\bnu(i)} G(\b z)|_{\b z = \b c_1 - \b y_k}}{\bnu(i)!}
  \qquad (i \in \{1, 2, \ldots, N(p)\}).
\]
\lean{PDETranslations.finiteSourceLocalCoefficients}
It follows from Lemma~\ref{lemma:l2l_difference}, with
$\b h = \b c_2 - \b c_1$ and $\|\b x - \b c_2\| < R_2$,
\begin{align*}
 &|(\b \mu_i^c - \b \mu_i)(\b x - \b c_2)^{\bnu(i)}|\\
 &\leq \frac{R_2^{|\bnu(i)|}}{\bnu(i)!}
  \sum_{k \in \mathcal{K}} |w(\b y_k)|
  \sum_{\substack{s \in \{1,\dots,|\b j|\} \\ |\bnu(\b j_s)| < |\bnu(i)|}}
  |M_{i,s}|
  \sum_{n = p-|\bnu(i)|+1}^{p-|\bnu(\b j_s)|}
  \left| \sum_{|\b q|=n}
      \frac{\D z^{\b q+\bnu(\b j_s)}G(\b z)|_{\b z = \b c_1 - \b y_k}}{\b q!}
      \b h^{\b q}\right| \\
 &\stackrel{\eqref{eq:assumption_g}}{\le}
  \sum_{k \in \mathcal{K}} |w(\b y_k)|
  \sum_{\substack{s \in \{1,\dots,|\b j|\} \\ |\bnu(\b j_s)| < |\bnu(i)|}}
  \sum_{n = p-|\bnu(i)|+1}^{p-|\bnu(\b j_s)|}
  \frac{R_2^{|\bnu(i)|}|M_{i,s}|C_{G,p}(1+n)^{s_p}}
       {\bnu(i)!\|\b c_1 - \b y_k\|^{|\bnu(\b j_s)|+p'}}
  \left(\frac{\|\b h\|}{\|\b c_1 - \b y_k\|}\right)^n \\
 &\leq \sum_{k \in \mathcal{K}} \frac{|w(\b y_k)|}{\|\b c_1 - \b y_k\|^{p'}}
  \sum_{\substack{s \in \{1,\dots,|\b j|\} \\ |\bnu(\b j_s)| < |\bnu(i)|}}
  \sum_{n = p-|\bnu(i)|+1}^{p-|\bnu(\b j_s)|}
  \frac{R_2^{|\bnu(i)|}|M_{i,s}|C_{G,p}(1+n)^{s_p}}
       {\bnu(i)!R_{\mathrm{loc}}^{|\bnu(\b j_s)|}}
  \left(\frac{R_1}{R_{\mathrm{loc}}}\right)^n,
\end{align*}
where~\eqref{eq:assumption_g} is applied with $m=p$ and $\b r=\bnu(\b j_s)$,
followed by $\|\b h\| \leq R_1$ and $\|\b c_1 - \b y_k\| \geq R_{\mathrm{loc}}$.

Denote $\tilde A=\sum_{k \in \mathcal{K}}|w(\b y_k)|/\|\b c_1 - \b y_k\|^{p'}$. Then,
\begin{align*}
 &|(\b \mu_i^c - \b \mu_i)(\b x - \b c_2)^{\bnu(i)}|\\
 &\leq \tilde A \sum_{\substack{s \in \{1,\dots,|\b j|\} \\ |\bnu(\b j_s)| < |\bnu(i)|}}
  \sum_{n = p-|\bnu(i)|+1}^{p-|\bnu(\b j_s)|}
  \frac{C_{G,p}(1+n)^{s_p}}{\bnu(i)!}
  \frac{|M_{i,s}|}{R_{\mathrm{loc}}^{|\bnu(\b j_s)| - |\bnu(i)|}}
  \left(\frac{R_2}{R_{\mathrm{loc}}}\right)^{|\bnu(i)|}
  \left(\frac{R_1}{R_{\mathrm{loc}}}\right)^n\\
  &\stackrel{\eqref{eq:assumption_v}}{\leq} \tilde A
  \sum_{\substack{s \in \{1,\dots,|\b j|\} \\ |\bnu(\b j_s)| < |\bnu(i)|}}
  \sum_{n = p-|\bnu(i)|+1}^{p-|\bnu(\b j_s)|}
  \frac{C_{G,p}(1+n)^{s_p}K'}{\bnu(i)!}
  \left(\frac{R_1}{R_{\mathrm{loc}}}\right)^{n +|\bnu(i)|}\\
  &\leq \tilde A
  \sum_{\substack{s \in \{1,\dots,|\b j|\} \\ |\bnu(\b j_s)| < |\bnu(i)|}}
  \frac{pC_{G,p}(1+p)^{s_p}K'}{\bnu(i)!}
  \left(\frac{R_1}{R_{\mathrm{loc}}}\right)^{p+1}\\
  &\leq \frac{p(p+1)^{s_p}|\b j| \tilde A C_{G,p} K'}{\bnu(i)!}
  \left(\frac{R_1}{R_{\mathrm{loc}}}\right)^{p+1}.
\end{align*}
\lean{PDETranslations.l2l_error_single_coeff}
Applying the triangle inequality over all $i$, and using
$\sum_{i=1}^{N(p)}1/\bnu(i)!\leq e^d$,
\begin{align*}
 \left|\sum_{i=1}^{N(p)} (\b \mu_i^c - \b \mu_i)(\b x - \b c_2)^{\bnu(i)}\right|
 &\leq e^{d} p (p+1)^{s_p} |\b j| \tilde A C_{G,p} K'
  \left(\frac{R_1}{R_{\mathrm{loc}}}\right)^{p+1}.
\end{align*}
In the case $p'\geq 0$, we can further bound $\tilde A$, using
$\|\b y_k - \b c_1\| \geq R_{\mathrm{loc}}$,
\[
  \tilde A = \sum_{k \in \mathcal{K}} \frac{|w(\b y_k)|}{\|\b c_1 - \b y_k\|^{p'}}
  \leq  R_{\mathrm{loc}}^{-p'}\sum_{k \in \mathcal{K}} |w(\b y_k)|.
\]
Hence,
\begin{align*}
 \left|\sum_{i=1}^{N(p)} (\b \mu_i^c - \b \mu_i)(\b x - \b c_2)^{\bnu(i)}\right|
 &\leq \frac{e^{d} p (p+1)^{s_p} |\b j| C_{G,p} K' \sum_{k \in \mathcal{K}} |w(\b y_k)|}{R_{\mathrm{loc}}^{p'}}
  \left(\frac{R_1}{R_{\mathrm{loc}}}\right)^{p+1}.
\end{align*}
In the case $p'<0$, using instead $\|\b y_k - \b c_1\| \leq a$ and $-p'>0$,
\[
  \tilde A = \sum_{k \in \mathcal{K}} \frac{|w(\b y_k)|}{\|\b c_1 - \b y_k\|^{p'}}
  \leq  a^{-p'}\sum_{k \in \mathcal{K}} |w(\b y_k)|,
\]
and substituting this bound into the estimate preceding the case distinction
yields the claim with $R_{\mathrm{loc}}^{-p'}$ replaced by $a^{-p'}$.
\end{proof}
\begin{remark}
Like the M2M estimate in Theorem~\ref{thm:varying}, the L2L estimate holds
for each fixed expansion order $p$. Establishing convergence as
$p\to\infty$ at fixed FMM separation ratios (e.g., $R_1/R_{\mathrm{loc}}$)
requires explicit bounds on the growth of the constants in terms of $p$.
We therefore leave a complete convergence analysis of the full FMM
hierarchy to future work.
\end{remark}

\subsubsection{Asymptotically Faster Local-to-Local Translation}
\label{sec:l2l_compute_faster}

From Theorem~\ref{thm:mult_decompression}, we know that calculating
$\b \alphal$ from $\b\alphal_{\b j}$ requires $\bigO{p^{d}}$ operations.
A naive evaluation of each $\b \mu_{i}$ following the formula above 
requires $\bigO{p^{d}}$ operations, resulting in an overall operation 
count of $\bigO{p^{2d-1}}$ for the translation in compressed representation.
Fortunately, the amount of work can be reduced to
$\bigO{p^{d}}$ by sum factorization techniques, similarly to the approach
described in Section~\ref{section:m2m_compute_faster} for the compressed
multipole-to-multipole translation.

\begin{lemma}
\label{lem:l2l_fix_one_dim}
  Let $p$ be the order of a local expansion and let
  \[
    S_{\ell, k} = \{\b m: \b m \in \mathcal{M}(p), \b m_{\ell} = k\}
  \] for $\ell\in\{1,\dots, d\}$, $k\in\{0,\dots, p\}$.
  When calculating the coefficients of a local expansion by translating
  a local expansion at a different center, calculating all the
  $\bigO{p^{d-1}}$ coefficients in $S_{\ell, k}$ for a given $\ell$ and $k$
  requires $\bigO{p^d}$ work.
\end{lemma}
\begin{proof}
Let $i\in \{1,\dots,N(p)\}$, $\b \eta = \bnu(i)$,
$\b \mu'_{\b \eta} = \b \mu_i$,
$\alphalp_{\bnu(q)} = \alphal_q$ for $q \in \{1, 2, \ldots, N(p)\}$.
Rewriting \eqref{eq:l2l-expr} using these substitutions yields
\[
  \b \mu'_{\b \eta} = \sum_{(\b \eta+\b \zeta) \in \mathcal{M}(p),\b \zeta \ge 0} \alphalp_{\b \zeta+\b \eta}
    \binom{\b \zeta+\b \eta}{\b \eta}(\b c_2 - \b c_1)^{\b \zeta}.
\]
As an illustrative example, consider the two-dimensional case with
$\b \mu'_{\b \eta}$ expanded:
\begin{align*}
  \b \mu'_{\eta_1, \eta_2}
    &= \sum_{\zeta_1=0}^{p - \eta_1} \left(
       \sum_{\zeta_2=0}^{p -\eta_1-\eta_2-\zeta_1}
      \alphalp_{\zeta_1+\eta_1, \zeta_2+\eta_2}
      \binom{\zeta_1+\eta_1}{\eta_1}\binom{\zeta_2+\eta_2}{\eta_2}
      d_1^{\zeta_1} d_2^{\zeta_2}\right) \\
    &= \sum_{\zeta_1=0}^{p - \eta_1}
      \Bigg(\underbrace{
        \sum_{\zeta_2=0}^{p -(\zeta_1+\eta_1)-\eta_2}
        \alphalp_{\zeta_1+\eta_1, \zeta_2+\eta_2}
        \binom{\zeta_2+\eta_2}{\eta_2} d_2^{\zeta_2}
      }_{\xi_{\zeta_1 + \eta_1, \eta_2} }\Bigg)
      d_1^{\zeta_1} \binom{\zeta_1+\eta_1}{\eta_1},
\end{align*}where $(d_1, d_2)^T = \b c_2 - \b c_1$.
Let
\begin{equation}
  \label{eq:l2l-dimbydim-subexpr}
  \xi_{\eta_1, \eta_2} = \sum_{\zeta_2=0}^{p - \eta_1 - \eta_2}
    \alphalp_{\eta_1, \zeta_2+\eta_2}
    \binom{\zeta_2+\eta_2}{\eta_2}  d_2 ^ {\zeta_2}.
\end{equation}
Assume that we are calculating the coefficients for multi-indices in $S_{2, k}$.
Then $\eta_2=k$ is fixed and we vary only $\eta_1$.
Evaluating the $\bigO{p}$ coefficients $(\xi_{\b m})_{\b m \in S_{2, k}}$
requires $\bigO{p^2}$ operations since the sum in \eqref{eq:l2l-dimbydim-subexpr}
has $\bigO{p}$ terms. Finally, once $(\xi_{\b m})_{\b m \in S_{2, k}}$ are calculated,
calculating $(\b \mu'_{\b m})_{\b m \in S_{2, k}}$ also requires $\bigO{p^2}$ operations.
A similar formula (factoring on the first instead of the second dimension) gives an algorithm
of the same asymptotic cost when calculating $(\b \mu'_{\b m})_{\b m \in S_{1, k}}$.

In the general case of $d$ dimensions, an analog of the above computation may use $d - 1$
`nested' definitions $\xi^{(1)}, \xi^{(2)}, \ldots \xi^{(d-1)}$, with each
intermediate quantity having $\bigO{p^{d-1}}$ elements and each entry a sum of
$\bigO{p}$ terms. Overall, the computation requires $\bigO{p^{d}}$ operations in total.
\end{proof}
In the uncompressed case, local-to-local translation carried out
naively requires $\bigO{p^{2d}}$ work. Using Lemma~\ref{lem:l2l_fix_one_dim}
and the observation that
\[
  \mathcal{M}(p) = \bigcup\limits_{k=0}^{p} S_{1, k},
\]
we can reduce this cost to $\bigO{p^{d + 1}}$, in a manner mathematically
analogous to~\citep{greengard1988rapid} and the tensor contraction technique of~\citep{shanker2007accelerated}. In the compressed case, we subdivide $\bnu(\b j)$
into slices, and only
$\bigO{1}$ slices need to be computed. 
As an illustrative example, consider the PDE
\begin{equation}
  \frac{\partial^2 u}{\partial x \partial y}
  + \frac{\partial u}{\partial x}
  + \frac{\partial u}{\partial y} = 0
  \label{eq:two-slice-example}
\end{equation}
in two dimensions. Using a degree-lexicographic ordering $\bnu$ where the $x$
dimension varies slowest, it follows from
Proposition~\ref{prop:j-character} that
\begin{equation*}
  \bnu(\b j) = \{\b m\in \mathcal{M}(p): \b m_1 = 0\text{ or }\b m_2 = 0\} = S_{1, 0} \cup S_{2, 0}.
\end{equation*}
As a result, we have a union of two slices with the first slice varying only the
first dimension and the second slice varying only the second dimension.
This leads to a $\bigO{p^{2}}$ algorithm. Lemma~\ref{lem:divide_multi_indices}
generalizes this to any constant-coefficient linear PDE.

\begin{lemma}
\label{lem:divide_multi_indices}
Suppose $c$ is the order of a constant-coefficient linear PDE.
$\bnu(\b j)$ can be written as a union
of $c$ (not necessarily disjoint) sets of the form $S_{\ell, k}$. 
\end{lemma}
\begin{proof}
Let $\b t = \bnu(h(1))$. By Proposition~\ref{prop:j-character}, 
\begin{align*}
  \bnu(\b j) &= \{ \b m \in \mathcal M(p): 
  \exists i \in \{1, \dots, d\} \text{ such that } \b m_i < \b t_i \} \\ 
  &= \bigcup\limits_{\ell=1}^{d} 
      \{ \b m \in \mathcal M(p): \b m_\ell < \b t_\ell \}\\
  &= \bigcup\limits_{\ell=1}^{d} 
      \left(\bigcup\limits_{\substack{k \in \mathbb{N}_0 \\ k < \b t_\ell}}
      S_{\ell, k} \right).
\end{align*}
For $\ell \in \{1, \dots, d\}$, there are $\b t_\ell$ sets of the form $S_{\ell, k}$ 
with $k < \b t_\ell$. Therefore, $\bnu(\b j)$ can be represented 
as a union of $\sum_{\ell=1}^{d} \b t_\ell = |\b t| = c$ sets of the form $S_{\ell, k}$.
\end{proof}
\begin{theorem}
Translating a local expansion in compressed representation around center $\b c_1$
to a local expansion in compressed representation around center $\b c_2$ can be
achieved with $\bigO{p^{d}}$ work.
\label{thm:l2l_cost}
\end{theorem}
\begin{proof}
Follows from Theorem~\ref{thm:mult_decompression} and
Lemmas~\ref{lem:l2l_fix_one_dim} and~\ref{lem:divide_multi_indices}.
\end{proof}
\subsection{Asymptotic Cost of Derivative Evaluation}
\label{sec:calculate_derivatives}

In Definition~\ref{def:wp}, we have introduced $w(p)$ as the amortized
operation count needed to compute one derivative when computing
the derivatives $[\DD{p}G(\b x)]_{\b j}$ required in the
compressed case.
For kernels that depend only on the (scalar) distance
between two points, \citet{tausch2003fast} gives an algorithm with $\bigO{p^{d+1}}$
cost for calculating all derivatives $\DD{p}G(\b x)$ when the derivatives of
the kernel
\begin{equation}
  \label{eq:tausch-dist-derivatives}
  \left(\left(\frac1{r} \frac{d}{dr}\right)^i G(r)\right)_{i \in \{0, \ldots, p\}}
\end{equation}
with respect to the distance are known. A straightforward modification of this
algorithm yields an algorithm to compute $[\DD{p}G(\b x)]_{\b j}$ with
$w(p) = \bigO{p}$ cost, assuming an algorithm to compute
\eqref{eq:tausch-dist-derivatives} is available.

In some special cases, the amortized cost for computing derivatives
can be lowered further. In the following sections, we describe a few
algorithms we found for reducing the cost to $w(p)=\bigO{1}$.
To the best of our knowledge, these formulas have not previously been
used in the context of Fast Multipole Methods.

\subsubsection{Laplace 2D}
Let $n \in \mathbb{N}_0$ and $r=\sqrt{x^2+y^2}$.
To facilitate derivative computations in our expansions,
we give a recurrence for the $x$ derivatives of $\log(r)$\lean{PDETranslations.Recurrences.logRadius}:
\begin{align*}
  r^2 \frac{\partial^{n}}{\partial x^{n}} \log(r) =
    - 2(n-1)x \frac{\partial^{n-1}}{\partial x^{n-1}} \log(r)
    - (n-1)(n-2) \frac{\partial^{n-2}}{\partial x^{n-2}} \log(r)
\end{align*}
\lean{PDETranslations.Recurrences.logRadius_x_recurrence}
where $n\ge3$.
A recurrence can still be found when a $y$ derivative is present:
\begin{equation}
  r^2 \frac{\partial^{n+1}}{\partial x^{n}\,\partial y} \log(r) =
    - 2(n-1)x \frac{\partial^{n}}{\partial x^{n-1}\,\partial y} \log(r)
    - (n-1)(n-2) \frac{\partial^{n-1}}{\partial x^{n-2}\,\partial y} \log(r)
    - 2 y \frac{\partial^{n}}{\partial x^{n}} \log(r),
  \label{eq:deriv_laplace_2d}
\end{equation}
\lean{PDETranslations.Recurrences.logRadius_xy_recurrence}
again for $n\ge 3$.
These formulae can be derived using induction or by using the
definitions of the $n$-th order Chebyshev
polynomials of the first and second kind as defined
in \dlmfref{18.5.1} and \dlmfref{18.5.2}.

For our expansion machinery, we can choose $\bnu$ such
that
\[
  \bnu(\b j) = \{(n, 0): n \in \{0, \ldots p\}\} \cup \{(n, 1): n \in \{0, \ldots, p - 1\}\},
\]
i.e., so that compressed expansions involve only the two `slices' of multi-indices
shown in Figure~\ref{fig:j_laplace}. Therefore, the above recurrences suffice to provide
an amortized constant-time algorithm with $w(p)=\bigO{1}$ for the
Laplace kernel in two dimensions.

\subsubsection{Laplace 3D}

Here, $r=\sqrt{x^2+y^2+z^2}$.
Using induction we can give a recurrence for the $x$ derivatives of $1/r$:
\begin{align*}
  r^2 \frac{\partial^{n}}{\partial x^{n}} \left(\frac{1}{r}\right) =
    - (2n-1)x \frac{\partial^{n-1}}{\partial x^{n-1}} \left(\frac{1}{r}\right)
    - (n-1)^2 \frac{\partial^{n-2}}{\partial x^{n-2}} \left(\frac{1}{r}\right)
\end{align*}
\lean{PDETranslations.Recurrences.invRadius_x_recurrence}
where $n\ge2$.
A similar recurrence holds when a $y$ derivative is present:
\begin{align*}
  r^2 \frac{\partial^{n+1}}{\partial x^{n}\,\partial y} \left(\frac{1}{r}\right) =
    - (2n-1)x \frac{\partial^{n}}{\partial x^{n-1}\,\partial y} \left(\frac{1}{r}\right)
    - (n-1)^2 \frac{\partial^{n-1}}{\partial x^{n-2}\,\partial y} \left(\frac{1}{r}\right)
    - 2 y \frac{\partial^{n}}{\partial x^{n}}
    \left(\frac{1}{r}\right),
    \quad (n\ge 2).
\end{align*}
This in turn can be generalized to the case of an arbitrary number
of derivatives along the $x$ and $y$ axes:
\begin{align*}
  r^2 \frac{\partial^{n+m}}{\partial x^{n}\,\partial y^m} \left(\frac{1}{r}\right) =
    & -(2n-1)x \frac{\partial^{n+m-1}}{\partial x^{n-1}\,\partial y^m} \left(\frac{1}{r}\right)
      - (n-1)^2 \frac{\partial^{n+m-2}}{\partial x^{n-2}\,\partial y^m} \left(\frac{1}{r}\right) \\
    & - 2m y \frac{\partial^{n+m-1}}{\partial x^{n}\,\partial y^{m-1}} \left(\frac{1}{r}\right)
      - m(m-1) \frac{\partial^{n+m-2}}{\partial x^{n}\,\partial y^{m-2}} \left(\frac{1}{r}\right)
\end{align*}
\lean{PDETranslations.Recurrences.invRadius_mixed_recurrence_2D}
for all $m\ge 0, n\ge 1$ where the last term is defined and nonzero for $m\ge 2$ and the preceding
term is defined and nonzero for $m\ge 1$.
An analogous generalization to three variables yields the recurrence formula
\begin{align}
  r^2 \frac{\partial^{n+m+l}}{\partial x^{n}\,\partial y^m\,\partial z^l} \left(\frac{1}{r}\right) =
    & -(2n-1)x \frac{\partial^{n+m+l-1}}{\partial x^{n-1}\,\partial y^m\,\partial z^l} \left(\frac{1}{r}\right)
      - (n-1)^2 \frac{\partial^{n+m+l-2}}{\partial x^{n-2}\,\partial y^m\,\partial z^l} \left(\frac{1}{r}\right) \nonumber \\
    & - 2m y \frac{\partial^{n+m+l-1}}{\partial x^{n}\,\partial y^{m-1}\,\partial z^l} \left(\frac{1}{r}\right)
      - m(m-1) \frac{\partial^{n+m+l-2}}{\partial x^{n}\,\partial y^{m-2}\,\partial z^l} \left(\frac{1}{r}\right) \nonumber \\
    & - 2l z \frac{\partial^{n+m+l-1}}{\partial x^{n}\,\partial y^{m}\,\partial z^{l-1}} \left(\frac{1}{r}\right)
      - l(l-1) \frac{\partial^{n+m+l-2}}{\partial x^{n}\,\partial y^{m}\,\partial z^{l-2}} \left(\frac{1}{r}\right).
  \label{eq:deriv_laplace_3d}
\end{align}
\lean{PDETranslations.Recurrences.invRadius_mixed_recurrence_3D}
As above, this leads to an algorithm for the computation of the
derivatives needed for our expansions with $w(p) = \bigO{1}$.

\subsubsection{Biharmonic 2D}
Using induction, we can prove that,
\begin{align*}
  r^2 \frac{\partial^{n}}{\partial x^{n}} \left(r^2 \log(r)\right) = &
    -2(n-2)x \frac{\partial^{n-1}}{\partial x^{n-1}} \left(r^2 \log(r)\right)
    -(n-1)(n-4) \frac{\partial^{n-2}}{\partial x^{n-2}} \left(r^2 \log(r)\right)
\end{align*}
\lean{PDETranslations.Recurrences.biharmonicRadius_x_recurrence}
for $n \ge 5$.
Differentiating by $y$, we get,
\begin{align*}
  r^2 \frac{\partial^{n+1}}{\partial x^{n}\,\partial y} \left(r^2 \log(r)\right) = &
      -2(n-2)x \frac{\partial^{n}}{\partial x^{n-1}\,\partial y} \left(r^2 \log(r)\right) \\
    &-(n-1)(n-4) \frac{\partial^{n-1}}{\partial x^{n-2}\,\partial y} \left(r^2 \log(r)\right)
      -2y\frac{\partial^{n}}{\partial x^{n}} \left(r^2 \log(r)\right)
\end{align*}
which in turn leads to
\begin{align}
    r^2 \frac{\partial^{n+m}}{\partial x^{n}\,\partial y^m} \left(r^2 \log(r)\right) = &
        -2(n-2)x \frac{\partial^{n+m-1}}{\partial x^{n-1}\,\partial y^m} \left(r^2 \log(r)\right)
        -(n-1)(n-4) \frac{\partial^{n+m-2}}{\partial x^{n-2}\,\partial y^m} \left(r^2 \log(r)\right) \nonumber \\
      & -2 m y\frac{\partial^{n+m-1}}{\partial x^{n}\,\partial y^{m-1}} \left(r^2 \log(r)\right)
        -m(m-1) \frac{\partial^{n+m-2}}{\partial x^{n}\,\partial y^{m-2}} \left(r^2 \log(r)\right)
  \label{eq:deriv_biharmonic_2d}
\end{align}
\lean{PDETranslations.Recurrences.biharmonicRadius_mixed_recurrence}
which is true for $n\ge5, m\ge0$ where the last term is defined and nonzero
for $m\ge2$ and the preceding term is defined and nonzero for $m\ge 1$.

As above, this leads to a $w(p)=\bigO{1}$ algorithm for computing the needed derivatives.

\subsection{Summary}
\begin{table}
\centering
\begin{tabular}{lcccc}\toprule
                                  & P2L / M2P      & P2M / L2P & M2M / L2L & M2L              \\ \midrule
Taylor Series                     & $p^{d} + p^{d-1} w(p) $  & $p^d$     & $p^{d+1}$ & $p^{2d}$         \\
Taylor Series with FFT            & $p^{d} + p^{d-1} w(p) $  & $p^d$     & $p^{d+1}$ & $p^{d}\log(p)$   \\
Compressed Taylor Series          & $p^{d-1} w(p)$ & $p^d$     & $p^d$     & $p^{2d-2}$       \\
Compressed Taylor Series with FFT & $p^{d-1} w(p)$ & $p^d$     & $p^{d}$   & $p^{d-1}\log(p)$ \\ \bottomrule
\end{tabular}
  \caption{
    Time complexities for expansions, translations and evaluations.
    In the table,
    P2L is the formation of a local expansion from a point source,
    P2M is the formation of a multipole expansion from a point source,
    M2M is the translation of a multipole expansion into another multipole expansion,
    M2L is the translation of a multipole expansion into a local expansion,
    L2L is the translation of a local expansion into another local expansion,
    M2P is the evaluation of a multipole expansion at a target point, and
    L2P is the evaluation of a local expansion at a target point.
    For uncompressed P2L/M2P, the derivatives are obtained from
    $[\DD{p}G]_{\b j}$ using the decompression operator
    (Theorem~\ref{thm:mult_decompression}).
    All M2L costs assume precomputed kernel derivatives (cf.~Section~\ref{sec:m2l}).
    The $p^{d-1}\log p$ cost bound additionally assumes Property~\ref{prop:nice}.
  } \label{tab:compressed}
\end{table}

Table~\ref{tab:compressed} summarizes time complexities for all the
formation, translation and evaluation operators for both
local and multipole expansions. In the table, $p$ refers to the order of the
expansions, $d$ is the number of dimensions and $w(p)$ is the amortized asymptotic cost
of calculating one derivative of the potential function
(cf.~Section~\ref{sec:calculate_derivatives}).

\section{Numerical Results}
\label{sec:results}

In order to support the claims presented, we conduct four numerical
experiments, described below.
Code to reproduce these experiments is available to the public,
along with associated instructions, at~\citep{papercode}.
Furthermore, an implementation of the approach in this paper
is available in open-source form under the MIT license
in our software package \textsc{Sumpy} at \citep{sumpy}.
By combining \textsc{Sumpy} with its sister package
\textsc{Boxtree} at \citep{boxtree} under the same license, the translation
operators in this paper can be used in the context of a
Fast Multipole Method.

\subsection{Accuracy of Multipole-to-Multipole Translation}
\label{sec:m2m-experiment}
\begin{figure}[ht!]
  \centering
    \begin{subfigure}[t]{.29\linewidth}
      \caption{Laplace 2D}
      \label{fig:m2m_accuracy_laplace_2d}
      \begin{tikzpicture}[scale=0.5]

\definecolor{crimson2143940}{RGB}{214,39,40}
\definecolor{darkgrey176}{RGB}{176,176,176}
\definecolor{darkorange25512714}{RGB}{255,127,14}
\definecolor{forestgreen4416044}{RGB}{44,160,44}
\definecolor{lightgrey204}{RGB}{204,204,204}
\definecolor{mediumpurple148103189}{RGB}{148,103,189}
\definecolor{sienna1408675}{RGB}{140,86,75}
\definecolor{steelblue31119180}{RGB}{31,119,180}

\begin{axis}[
legend cell align={left},
legend style={
  fill opacity=0.8,
  draw opacity=1,
  text opacity=1,
  at={(0.03,0.97)},
  anchor=north west,
  draw=lightgrey204
},
log basis x={10},
log basis y={10},
tick align=outside,
tick pos=left,
x grid style={darkgrey176},
xlabel={Geometric parameter \(\displaystyle R\)},
xmajorgrids,
xmin=0.000766195408102296, xmax=0.159320078414908,
xmode=log,
xtick style={color=black},
xtick={1e-05,0.0001,0.001,0.01,0.1,1,10},
xticklabels={
  \(\displaystyle {10^{-5}}\),
  \(\displaystyle {10^{-4}}\),
  \(\displaystyle {10^{-3}}\),
  \(\displaystyle {10^{-2}}\),
  \(\displaystyle {10^{-1}}\),
  \(\displaystyle {10^{0}}\),
  \(\displaystyle {10^{1}}\)
},
y grid style={darkgrey176},
ylabel={\(\displaystyle \epsilon_{\mathrm{rel}}\)},
ymajorgrids,
ymin=7.8248467794913e-17, ymax=5.93710400837055e-16,
ymode=log,
ytick style={color=black},
ytick={2e-16,5e-16},
yticklabels={
\(\displaystyle {2e-16}\),
\(\displaystyle {5e-16}\),
  \(\displaystyle {10^{-18}}\),
  \(\displaystyle {10^{-17}}\),
  \(\displaystyle {10^{-16}}\),
  \(\displaystyle {10^{-15}}\),
  \(\displaystyle {10^{-14}}\)
}
]
\addplot [semithick, steelblue31119180, mark=*, mark size=3, mark options={solid}]
table {%
0.0009765625 8.57986261654943e-17
0.001953125 8.76379199098918e-17
0.00390625 9.00551458434996e-17
0.0078125 8.8393811034262e-17
0.015625 9.24055063326114e-17
0.03125 9.57518881346906e-17
0.0625 1.0719316505808e-16
0.125 1.26295664101647e-16
};
\addlegendentry{\(\displaystyle p=2\)}
\addplot [semithick, darkorange25512714, mark=*, mark size=3, mark options={solid}]
table {%
0.0009765625 1.92204085424399e-16
0.001953125 1.89228760422877e-16
0.00390625 2.05085606123782e-16
0.0078125 1.98669959685394e-16
0.015625 1.97198342330093e-16
0.03125 2.06649522827625e-16
0.0625 2.18188918017302e-16
0.125 2.29651423987583e-16
};
\addlegendentry{\(\displaystyle p=4\)}
\addplot [semithick, forestgreen4416044, mark=*, mark size=3, mark options={solid}]
table {%
0.0009765625 2.5620959739454e-16
0.001953125 2.35647942489875e-16
0.00390625 2.74813413486996e-16
0.0078125 2.77672558246405e-16
0.015625 2.73167605857975e-16
0.03125 2.88238003053096e-16
0.0625 3.00955582541828e-16
0.125 3.1509943671023e-16
};
\addlegendentry{\(\displaystyle p=6\)}
\addplot [semithick, crimson2143940, mark=*, mark size=3, mark options={solid}]
table {%
0.0009765625 2.56035065754151e-16
0.001953125 2.35820208046267e-16
0.00390625 2.80717655993172e-16
0.0078125 3.21949907367903e-16
0.015625 3.49947814259316e-16
0.03125 3.56527133782916e-16
0.0625 3.68108729024944e-16
0.125 3.86201739642758e-16
};
\addlegendentry{\(\displaystyle p=8\)}
\addplot [semithick, mediumpurple148103189, mark=*, mark size=3, mark options={solid}]
table {%
0.0009765625 2.56265954427199e-16
0.001953125 2.36465811208875e-16
0.00390625 2.80768226839482e-16
0.0078125 3.21900088526531e-16
0.015625 3.72823031984961e-16
0.03125 4.25329013521898e-16
0.0625 4.33379088748244e-16
0.125 4.68001718875338e-16
};
\addlegendentry{\(\displaystyle p=10\)}
\addplot [semithick, sienna1408675, mark=*, mark size=3, mark options={solid}]
table {%
0.0009765625 2.5620959739454e-16
0.001953125 2.35602811587663e-16
0.00390625 2.80485918301982e-16
0.0078125 3.22121545701945e-16
0.015625 3.72031109874404e-16
0.03125 4.51924345392422e-16
0.0625 5.00826737332243e-16
0.125 5.41464721006066e-16
};
\addlegendentry{\(\displaystyle p=12\)}
\end{axis}
      \end{tikzpicture}
    \end{subfigure}\hspace{3mm}
    \begin{subfigure}[t]{.29\linewidth}
      \caption{Helmholtz 2D ($\kappa=1$)}
      \label{fig:m2m_accuracy_helmholtz_2d}
      \begin{tikzpicture}[scale=0.5]

\definecolor{crimson2143940}{RGB}{214,39,40}
\definecolor{darkgrey176}{RGB}{176,176,176}
\definecolor{darkorange25512714}{RGB}{255,127,14}
\definecolor{forestgreen4416044}{RGB}{44,160,44}
\definecolor{grey}{RGB}{128,128,128}
\definecolor{lightgrey204}{RGB}{204,204,204}
\definecolor{mediumpurple148103189}{RGB}{148,103,189}
\definecolor{sienna1408675}{RGB}{140,86,75}
\definecolor{steelblue31119180}{RGB}{31,119,180}

\begin{axis}[
legend cell align={left},
legend style={
  fill opacity=0.8,
  draw opacity=1,
  text opacity=1,
  at={(0.03,0.97)},
  anchor=north west,
  draw=lightgrey204
},
log basis x={10},
log basis y={10},
tick align=outside,
tick pos=left,
x grid style={darkgrey176},
xlabel={Geometric parameter \(\displaystyle R\)},
xmajorgrids,
xmin=0.000766195408102296, xmax=0.159320078414908,
xmode=log,
xtick style={color=black},
xtick={1e-05,0.0001,0.001,0.01,0.1,1,10},
xticklabels={
  \(\displaystyle {10^{-5}}\),
  \(\displaystyle {10^{-4}}\),
  \(\displaystyle {10^{-3}}\),
  \(\displaystyle {10^{-2}}\),
  \(\displaystyle {10^{-1}}\),
  \(\displaystyle {10^{0}}\),
  \(\displaystyle {10^{1}}\)
},
y grid style={darkgrey176},
ymajorgrids,
ymin=1e-16, ymax=0.00196116521531852,
ymode=log,
ytick style={color=black},
ytick={1e-18,1e-16,1e-14,1e-12,1e-10,1e-08,1e-06,0.0001,0.01,1},
yticklabels={
  \(\displaystyle {10^{-18}}\),
  \(\displaystyle {10^{-16}}\),
  \(\displaystyle {10^{-14}}\),
  \(\displaystyle {10^{-12}}\),
  \(\displaystyle {10^{-10}}\),
  \(\displaystyle {10^{-8}}\),
  \(\displaystyle {10^{-6}}\),
  \(\displaystyle {10^{-4}}\),
  \(\displaystyle {10^{-2}}\),
  \(\displaystyle {10^{0}}\)
}
]
\addplot [semithick, steelblue31119180, mark=*, mark size=3, mark options={solid}]
table {%
0.0009765625 2.52630890948885e-10
0.001953125 2.02070814132743e-09
0.00390625 1.6160370643103e-08
0.0078125 1.29198525199096e-07
0.015625 1.03224230339829e-06
0.03125 8.23651331405856e-06
0.0625 6.55494162438926e-05
0.125 0.000518704630468735
};
\addlegendentry{\(\displaystyle p=2\)}
\addplot [semithick, darkorange25512714, mark=*, mark size=3, mark options={solid}]
table {%
0.0009765625 3.14234069840418e-15
0.001953125 2.54355888155819e-15
0.00390625 7.21089959599111e-14
0.0078125 2.34773470024703e-12
0.015625 7.51444192774197e-11
0.03125 2.40410443747435e-09
0.0625 7.69003375890826e-08
0.125 2.45909997112355e-06
};
\addlegendentry{\(\displaystyle p=4\)}
\addplot [semithick, forestgreen4416044, mark=*, mark size=3, mark options={solid}]
table {%
0.0009765625 3.18829985321841e-15
0.001953125 3.4071106400925e-15
0.00390625 1.96725284822655e-15
0.0078125 1.47612732587265e-15
0.015625 3.14263360260281e-14
0.03125 4.01651805761239e-12
0.0625 5.32385129771771e-10
0.125 7.25043316047608e-08
};
\addlegendentry{\(\displaystyle p=6\)}
\addplot [semithick, crimson2143940, mark=*, mark size=3, mark options={solid}]
table {%
0.0009765625 3.16230863685638e-15
0.001953125 3.40432578726191e-15
0.00390625 1.96654588673725e-15
0.0078125 1.45558975607786e-15
0.015625 5.69941155932923e-15
0.03125 1.5196018836966e-14
0.0625 7.96089005884397e-12
0.125 4.44340150316667e-09
};
\addlegendentry{\(\displaystyle p=8\)}
\addplot [semithick, mediumpurple148103189, mark=*, mark size=3, mark options={solid}]
table {%
0.0009765625 3.14896580110272e-15
0.001953125 3.40457958657312e-15
0.00390625 1.96983132438661e-15
0.0078125 1.4576030433005e-15
0.015625 5.70675349976233e-15
0.03125 3.1373947822333e-15
0.0625 1.79684799720004e-13
0.125 4.05046663805671e-10
};
\addlegendentry{\(\displaystyle p=10\)}
\addplot [semithick, sienna1408675, mark=*, mark size=3, mark options={solid}]
table {%
0.0009765625 3.16207922113603e-15
0.001953125 3.4052851097646e-15
0.00390625 1.96760798495757e-15
0.0078125 1.45876813219913e-15
0.015625 5.70110302464062e-15
0.03125 3.13220045823586e-15
0.0625 5.33598590569826e-15
0.125 4.32776819578593e-11
};
\addlegendentry{\(\displaystyle p=12\)}
\addplot [semithick, grey, dashed]
table {%
0.0009765625 2.47337641939514e-10
0.00350765306122449 1.14614609577849e-08
0.00603874362244898 5.84830294805219e-08
0.00856983418367347 1.67150370149214e-07
0.011100924744898 3.63301809902926e-07
0.0136320153061224 6.72775675680721e-07
0.0161631058673469 1.12141029442166e-06
0.0186941964285714 1.73504399306481e-06
0.0212252869897959 2.53951509854924e-06
0.0237563775510204 3.56066193781401e-06
0.0262874681122449 4.82432283779817e-06
0.0288185586734694 6.35633612544081e-06
0.0313496492346939 8.18254012768097e-06
0.0338807397959184 1.03287731714577e-05
0.0364118303571429 1.28208735837101e-05
0.0389429209183673 1.56846796913773e-05
0.0414740114795918 1.89460298213982e-05
0.0440051020408163 2.2630762300712e-05
0.0465361926020408 2.67647154562577e-05
0.0490672831632653 3.13737276149743e-05
0.0515983737244898 3.6483637103801e-05
0.0541294642857143 4.21202822496768e-05
0.0566605548469388 4.83095013795408e-05
0.0591916454081633 5.5077132820332e-05
0.0617227359693878 6.24490148989895e-05
0.0642538265306122 7.04509859424524e-05
0.0667849170918367 7.91088842776596e-05
0.0693160076530612 8.84485482315504e-05
0.0718470982142857 9.84958161310637e-05
0.0743781887755102 0.000109276526303139
0.0769092793367347 0.000120816517074714
0.0794403698979592 0.00013314162677273
0.0819714604591837 0.000146277693724124
0.0845025510204082 0.000160250556255836
0.0870336415816327 0.000175086052694805
0.0895647321428571 0.00019081002136797
0.0920958227040816 0.00020744830060227
0.0946269132653061 0.000225026728724644
0.0971580038265306 0.000243571144062031
0.0996890943877551 0.000263107384941371
0.10222018494898 0.000283661289689602
0.104751275510204 0.000305258696633663
0.107282366071429 0.000327925444100494
0.109813456632653 0.000351687370417033
0.112344547193878 0.00037657031391022
0.114875637755102 0.000402600112906993
0.117406728316327 0.000429802605734292
0.119937818877551 0.000458203630719056
0.122468909438776 0.000487829026188224
0.125 0.000518704630468735
};
\addlegendentry{\(\displaystyle R^{3}\)}
\addplot [semithick, red, dashed]
table {%
0.0009765625 1.74797165389537e-38
0.00350765306122449 2.89500972394327e-31
0.00603874362244898 3.37862130055749e-28
0.00856983418367347 3.19944955954324e-26
0.011100924744898 9.24917256868135e-25
0.0136320153061224 1.33571765878969e-23
0.0161631058673469 1.22252134330704e-22
0.0186941964285714 8.10255978451145e-22
0.0212252869897959 4.2221168281462e-21
0.0237563775510204 1.82631911849675e-20
0.0262874681122449 6.81032264247649e-20
0.0288185586734694 2.24995148511573e-19
0.0313496492346939 6.72138448741278e-19
0.0338807397959184 1.84425650148978e-18
0.0364118303571429 4.70529070906441e-18
0.0389429209183673 1.12720944001815e-17
0.0414740114795918 2.55577364832108e-17
0.0440051020408163 5.52043426639298e-17
0.0465361926020408 1.14212818480862e-16
0.0490672831632653 2.27368332624811e-16
0.0515983737244898 4.37221994311675e-16
0.0541294642857143 8.14839921244147e-16
0.0566605548469388 1.47599375915808e-15
0.0591916454081633 2.60506873290963e-15
0.0617227359693878 4.48973817055065e-15
0.0642538265306122 7.57043397168998e-15
0.0667849170918367 1.25098731817673e-14
0.0693160076530612 2.02894303975104e-14
0.0718470982142857 3.23410296797022e-14
0.0743781887755102 5.0725530908678e-14
0.0769092793367347 7.83713750464663e-14
0.0794403698979592 1.19390610225577e-13
0.0819714604591837 1.79493435008227e-13
0.0845025510204082 2.66527165683051e-13
0.0870336415816327 3.91171171048346e-13
0.0895647321428571 5.67825942320424e-13
0.0920958227040816 8.15742211528394e-13
0.0946269132653061 1.16044504612939e-12
0.0971580038265306 1.63551829174779e-12
0.0996890943877551 2.28482678495133e-12
0.10222018494898 3.16526760565625e-12
0.104751275510204 4.35015749263796e-12
0.107282366071429 5.93338163189589e-12
0.109813456632653 8.03444959108116e-12
0.112344547193878 1.08046274782569e-11
0.114875637755102 1.44343417349683e-11
0.117406728316327 1.91620796386788e-11
0.119937818877551 2.52850449175644e-11
0.122468909438776 3.31718642247716e-11
0.125 4.32776819578593e-11
};
\addlegendentry{\(\displaystyle R^{13}\)}
\end{axis}
      \end{tikzpicture}
    \end{subfigure}
    \begin{subfigure}[t]{.29\linewidth}
      \caption{Biharmonic 2D}
      \label{fig:m2m_accuracy_biharmonic_2d}
      \begin{tikzpicture}[scale=0.5]

\definecolor{crimson2143940}{RGB}{214,39,40}
\definecolor{darkgrey176}{RGB}{176,176,176}
\definecolor{darkorange25512714}{RGB}{255,127,14}
\definecolor{forestgreen4416044}{RGB}{44,160,44}
\definecolor{lightgrey204}{RGB}{204,204,204}
\definecolor{mediumpurple148103189}{RGB}{148,103,189}
\definecolor{sienna1408675}{RGB}{140,86,75}
\definecolor{steelblue31119180}{RGB}{31,119,180}

\begin{axis}[
legend cell align={left},
legend style={
  fill opacity=0.8,
  draw opacity=1,
  text opacity=1,
  at={(0.03,0.97)},
  anchor=north west,
  draw=lightgrey204
},
log basis x={10},
log basis y={10},
tick align=outside,
tick pos=left,
x grid style={darkgrey176},
xlabel={Geometric parameter \(\displaystyle R\)},
xmajorgrids,
xmin=0.000766195408102296, xmax=0.159320078414908,
xmode=log,
xtick style={color=black},
xtick={1e-05,0.0001,0.001,0.01,0.1,1,10},
xticklabels={
  \(\displaystyle {10^{-5}}\),
  \(\displaystyle {10^{-4}}\),
  \(\displaystyle {10^{-3}}\),
  \(\displaystyle {10^{-2}}\),
  \(\displaystyle {10^{-1}}\),
  \(\displaystyle {10^{0}}\),
  \(\displaystyle {10^{1}}\)
},
y grid style={darkgrey176},
ymajorgrids,
ymin=3.74476362660787e-18, ymax=3.64170312792849e-16,
ymode=log,
ytick style={color=black},
ytick={1e-19,1e-18,1e-17,1e-16,1e-15,1e-14},
yticklabels={
  \(\displaystyle {10^{-19}}\),
  \(\displaystyle {10^{-18}}\),
  \(\displaystyle {10^{-17}}\),
  \(\displaystyle {10^{-16}}\),
  \(\displaystyle {10^{-15}}\),
  \(\displaystyle {10^{-14}}\)
}
]
\addplot [semithick, steelblue31119180, mark=*, mark size=3, mark options={solid}]
table {%
0.0009765625 4.61086702399001e-18
0.001953125 1.15002276889131e-17
0.00390625 1.70025069801774e-17
0.0078125 2.90781752352504e-17
0.015625 2.9690729087324e-17
0.03125 5.51457084955248e-17
0.0625 7.49896779703581e-17
0.125 1.01277951553603e-16
};
\addlegendentry{\(\displaystyle p=2\)}
\addplot [semithick, darkorange25512714, mark=*, mark size=3, mark options={solid}]
table {%
0.0009765625 5.27967040282205e-18
0.001953125 1.46521290506643e-17
0.00390625 1.9402189388415e-17
0.0078125 2.76343803968312e-17
0.015625 3.40728389459389e-17
0.03125 5.72496987542302e-17
0.0625 7.82264851409557e-17
0.125 1.11068135347524e-16
};
\addlegendentry{\(\displaystyle p=4\)}
\addplot [semithick, forestgreen4416044, mark=*, mark size=3, mark options={solid}]
table {%
0.0009765625 1.27756267316786e-16
0.001953125 1.17008366839626e-16
0.00390625 1.37586567164567e-16
0.0078125 1.41854690525656e-16
0.015625 1.53906528741027e-16
0.03125 1.50625356283618e-16
0.0625 1.75959330458839e-16
0.125 2.16217418046948e-16
};
\addlegendentry{\(\displaystyle p=6\)}
\addplot [semithick, crimson2143940, mark=*, mark size=3, mark options={solid}]
table {%
0.0009765625 1.28191963502437e-16
0.001953125 1.17190781704007e-16
0.00390625 1.38323889569896e-16
0.0078125 1.67831051520597e-16
0.015625 1.83291664283492e-16
0.03125 1.96047579237358e-16
0.0625 2.05929526902188e-16
0.125 2.42554996395081e-16
};
\addlegendentry{\(\displaystyle p=8\)}
\addplot [semithick, mediumpurple148103189, mark=*, mark size=3, mark options={solid}]
table {%
0.0009765625 1.28208508640395e-16
0.001953125 1.16839918924852e-16
0.00390625 1.38486223899699e-16
0.0078125 1.66700222647948e-16
0.015625 1.85384730837351e-16
0.03125 2.26344326815172e-16
0.0625 2.40524990061666e-16
0.125 2.70245141121097e-16
};
\addlegendentry{\(\displaystyle p=10\)}
\addplot [semithick, sienna1408675, mark=*, mark size=3, mark options={solid}]
table {%
0.0009765625 1.27756267316786e-16
0.001953125 1.17010376478475e-16
0.00390625 1.37848234543632e-16
0.0078125 1.67839535532164e-16
0.015625 1.84478723061703e-16
0.03125 2.25826791259075e-16
0.0625 2.64607243222042e-16
0.125 2.95764708490111e-16
};
\addlegendentry{\(\displaystyle p=12\)}
\end{axis}
      \end{tikzpicture}
    \end{subfigure}\vspace{5mm}
    \begin{subfigure}[t]{.29\linewidth}
      \caption{Laplace 3D}
      \label{fig:m2m_accuracy_laplace_3d}
      \begin{tikzpicture}[scale=0.5]

\definecolor{crimson2143940}{RGB}{214,39,40}
\definecolor{darkgrey176}{RGB}{176,176,176}
\definecolor{darkorange25512714}{RGB}{255,127,14}
\definecolor{forestgreen4416044}{RGB}{44,160,44}
\definecolor{lightgrey204}{RGB}{204,204,204}
\definecolor{mediumpurple148103189}{RGB}{148,103,189}
\definecolor{sienna1408675}{RGB}{140,86,75}
\definecolor{steelblue31119180}{RGB}{31,119,180}

\begin{axis}[
legend cell align={left},
legend style={
  fill opacity=0.8,
  draw opacity=1,
  text opacity=1,
  at={(0.03,0.97)},
  anchor=north west,
  draw=lightgrey204
},
log basis x={10},
log basis y={10},
tick align=outside,
tick pos=left,
x grid style={darkgrey176},
xlabel={Geometric parameter \(\displaystyle R\)},
xmajorgrids,
xmin=0.000766195408102296, xmax=0.159320078414908,
xmode=log,
xtick style={color=black},
xtick={1e-05,0.0001,0.001,0.01,0.1,1,10},
xticklabels={
  \(\displaystyle {10^{-5}}\),
  \(\displaystyle {10^{-4}}\),
  \(\displaystyle {10^{-3}}\),
  \(\displaystyle {10^{-2}}\),
  \(\displaystyle {10^{-1}}\),
  \(\displaystyle {10^{0}}\),
  \(\displaystyle {10^{1}}\)
},
y grid style={darkgrey176},
ylabel={\(\displaystyle \epsilon_{\mathrm{rel}}\)},
ymajorgrids,
ymin=1.06895667795015e-16, ymax=1.27478940766724e-15,
ymode=log,
ytick style={color=black},
ytick={1e-15,2e-16},
yticklabels={
\(\displaystyle {1e-15}\),
\(\displaystyle {2e-16}\),
  \(\displaystyle {10^{-17}}\),
  \(\displaystyle {10^{-16}}\),
  \(\displaystyle {10^{-15}}\),
  \(\displaystyle {10^{-14}}\),
  \(\displaystyle {10^{-13}}\)
}
]
\addplot [semithick, steelblue31119180, mark=*, mark size=3, mark options={solid}]
table {%
0.0009765625 1.20148149866584e-16
0.001953125 1.19644003933456e-16
0.00390625 1.19833060034111e-16
0.0078125 1.20951472571248e-16
0.015625 1.20963660651136e-16
0.03125 1.22648120389837e-16
0.0625 1.24802460593231e-16
0.125 1.27712785954924e-16
};
\addlegendentry{\(\displaystyle p=2\)}
\addplot [semithick, darkorange25512714, mark=*, mark size=3, mark options={solid}]
table {%
0.0009765625 3.32240652649364e-16
0.001953125 3.31925190746637e-16
0.00390625 3.31623274595282e-16
0.0078125 3.30719589128992e-16
0.015625 3.28203545344892e-16
0.03125 3.24355739230331e-16
0.0625 3.15463025078951e-16
0.125 2.9949393279903e-16
};
\addlegendentry{\(\displaystyle p=4\)}
\addplot [semithick, forestgreen4416044, mark=*, mark size=3, mark options={solid}]
table {%
0.0009765625 4.42576529687905e-16
0.001953125 4.80060721679787e-16
0.00390625 5.37140042405824e-16
0.0078125 5.31492337334431e-16
0.015625 5.26785352914528e-16
0.03125 5.19265633183001e-16
0.0625 5.03125319150544e-16
0.125 4.72286074613297e-16
};
\addlegendentry{\(\displaystyle p=6\)}
\addplot [semithick, crimson2143940, mark=*, mark size=3, mark options={solid}]
table {%
0.0009765625 4.42575601953759e-16
0.001953125 4.80093004840883e-16
0.00390625 5.55326348633839e-16
0.0078125 6.64186964430168e-16
0.015625 7.33532136425193e-16
0.03125 7.22291579106435e-16
0.0625 7.00748124569284e-16
0.125 6.55883638644864e-16
};
\addlegendentry{\(\displaystyle p=8\)}
\addplot [semithick, mediumpurple148103189, mark=*, mark size=3, mark options={solid}]
table {%
0.0009765625 4.42593719704063e-16
0.001953125 4.80119713373637e-16
0.00390625 5.55284849527053e-16
0.0078125 6.65729720181771e-16
0.015625 8.19710897525146e-16
0.03125 9.33216031413006e-16
0.0625 9.12043800936897e-16
0.125 8.52468948357299e-16
};
\addlegendentry{\(\displaystyle p=10\)}
\addplot [semithick, sienna1408675, mark=*, mark size=3, mark options={solid}]
table {%
0.0009765625 4.42576529687905e-16
0.001953125 4.80060721679787e-16
0.00390625 5.55458649561125e-16
0.0078125 6.6564892603173e-16
0.015625 8.21176275011953e-16
0.03125 1.04548538101808e-15
0.0625 1.13895774590084e-15
0.125 1.05761850143654e-15
};
\addlegendentry{\(\displaystyle p=12\)}
\end{axis}
      \end{tikzpicture}
    \end{subfigure}\hspace{3mm}
    \begin{subfigure}[t]{.29\linewidth}
      \caption{Helmholtz 3D ($\kappa=1$)}
      \label{fig:m2m_accuracy_helmholtz_3d}
      \begin{tikzpicture}[scale=0.5]

\definecolor{crimson2143940}{RGB}{214,39,40}
\definecolor{darkgrey176}{RGB}{176,176,176}
\definecolor{darkorange25512714}{RGB}{255,127,14}
\definecolor{forestgreen4416044}{RGB}{44,160,44}
\definecolor{grey}{RGB}{128,128,128}
\definecolor{lightgrey204}{RGB}{204,204,204}
\definecolor{mediumpurple148103189}{RGB}{148,103,189}
\definecolor{sienna1408675}{RGB}{140,86,75}
\definecolor{steelblue31119180}{RGB}{31,119,180}

\begin{axis}[
legend cell align={left},
legend style={
  fill opacity=0.8,
  draw opacity=1,
  text opacity=1,
  at={(0.03,0.97)},
  anchor=north west,
  draw=lightgrey204
},
log basis x={10},
log basis y={10},
tick align=outside,
tick pos=left,
x grid style={darkgrey176},
xlabel={Geometric parameter \(\displaystyle R\)},
xmajorgrids,
xmin=0.000766195408102296, xmax=0.159320078414908,
xmode=log,
xtick style={color=black},
xtick={1e-05,0.0001,0.001,0.01,0.1,1,10},
xticklabels={
  \(\displaystyle {10^{-5}}\),
  \(\displaystyle {10^{-4}}\),
  \(\displaystyle {10^{-3}}\),
  \(\displaystyle {10^{-2}}\),
  \(\displaystyle {10^{-1}}\),
  \(\displaystyle {10^{0}}\),
  \(\displaystyle {10^{1}}\)
},
y grid style={darkgrey176},
ymajorgrids,
ymin=1e-16, ymax=0.00252284041656394,
ymode=log,
ytick style={color=black},
ytick={1e-18,1e-16,1e-14,1e-12,1e-10,1e-08,1e-06,0.0001,0.01,1},
yticklabels={
  \(\displaystyle {10^{-18}}\),
  \(\displaystyle {10^{-16}}\),
  \(\displaystyle {10^{-14}}\),
  \(\displaystyle {10^{-12}}\),
  \(\displaystyle {10^{-10}}\),
  \(\displaystyle {10^{-8}}\),
  \(\displaystyle {10^{-6}}\),
  \(\displaystyle {10^{-4}}\),
  \(\displaystyle {10^{-2}}\),
  \(\displaystyle {10^{0}}\)
}
]
\addplot [semithick, steelblue31119180, mark=*, mark size=3, mark options={solid}]
table {%
0.0009765625 3.21294421475714e-10
0.001953125 2.56937840837345e-09
0.00390625 2.05399327892265e-08
0.0078125 1.64076180448694e-07
0.015625 1.30867070173216e-06
0.03125 1.04047912460387e-05
0.0625 8.2151098412257e-05
0.125 0.000637914616199765
};
\addlegendentry{\(\displaystyle p=2\)}
\addplot [semithick, darkorange25512714, mark=*, mark size=3, mark options={solid}]
table {%
0.0009765625 7.67770796059873e-15
0.001953125 8.31043396419914e-15
0.00390625 2.11129684915448e-13
0.0078125 6.78551399250205e-12
0.015625 2.17439992356482e-10
0.03125 6.97454456406824e-09
0.0625 2.23880998919188e-07
0.125 7.16107770914263e-06
};
\addlegendentry{\(\displaystyle p=4\)}
\addplot [semithick, forestgreen4416044, mark=*, mark size=3, mark options={solid}]
table {%
0.0009765625 7.70495540597146e-15
0.001953125 3.7551498422976e-15
0.00390625 3.56408379128076e-15
0.0078125 1.26872876344863e-15
0.015625 1.23789789357429e-13
0.03125 1.59289338387482e-11
0.0625 2.07254679717215e-09
0.125 2.71217430081608e-07
};
\addlegendentry{\(\displaystyle p=6\)}
\addplot [semithick, crimson2143940, mark=*, mark size=3, mark options={solid}]
table {%
0.0009765625 7.69956576317867e-15
0.001953125 3.75233325741724e-15
0.00390625 3.57406696552449e-15
0.0078125 7.46225982045674e-16
0.015625 2.95142272041718e-15
0.03125 5.66138642807363e-14
0.0625 2.95106597118842e-11
0.125 1.57912366116239e-08
};
\addlegendentry{\(\displaystyle p=8\)}
\addplot [semithick, mediumpurple148103189, mark=*, mark size=3, mark options={solid}]
table {%
0.0009765625 7.72316669315473e-15
0.001953125 3.7467742550405e-15
0.00390625 3.57430230094929e-15
0.0078125 7.41031764432325e-16
0.015625 2.95860555303196e-15
0.03125 1.80162882755588e-14
0.0625 6.50907022784949e-13
0.125 1.41541467913266e-09
};
\addlegendentry{\(\displaystyle p=10\)}
\addplot [semithick, sienna1408675, mark=*, mark size=3, mark options={solid}]
table {%
0.0009765625 7.72116267164705e-15
0.001953125 3.75134470502999e-15
0.00390625 3.56279959409323e-15
0.0078125 7.28145241558674e-16
0.015625 2.93306670622603e-15
0.03125 1.81034408202638e-14
0.0625 1.81631344615183e-14
0.125 1.59753283448976e-10
};
\addlegendentry{\(\displaystyle p=12\)}
\addplot [semithick, grey, dashed]
table {%
0.0009765625 3.04181392764933e-10
0.00350765306122449 1.40955623653617e-08
0.00603874362244898 7.19237444854762e-08
0.00856983418367347 2.05565283126601e-07
0.011100924744898 4.4679673366223e-07
0.0136320153061224 8.27394651465856e-07
0.0161631058673469 1.37913559191097e-06
0.0186941964285714 2.13379611037107e-06
0.0212252869897959 3.12315276221964e-06
0.0237563775510204 4.37898210283018e-06
0.0262874681122449 5.93306068757619e-06
0.0288185586734694 7.81716507183114e-06
0.0313496492346939 1.00630718109685e-05
0.0338807397959184 1.27025574603619e-05
0.0364118303571429 1.57673985753847e-05
0.0389429209183673 1.92893717114104e-05
0.0414740114795918 2.33002534238125e-05
0.0440051020408163 2.78318202679645e-05
0.0465361926020408 3.291584879924e-05
0.0490672831632653 3.85841155730123e-05
0.0515983737244898 4.4868397144655e-05
0.0541294642857143 5.18004700695416e-05
0.0566605548469388 5.94121109030456e-05
0.0591916454081633 6.77350962005405e-05
0.0617227359693878 7.68012025173997e-05
0.0642538265306122 8.66422064089968e-05
0.0667849170918367 9.72898844307052e-05
0.0693160076530612 0.000108776013137898
0.0718470982142857 0.00012113236908595
0.0743781887755102 0.000134390728830234
0.0769092793367347 0.000148582868926122
0.0794403698979592 0.00016374056592899
0.0819714604591837 0.00017989559639421
0.0845025510204082 0.000197079736877155
0.0870336415816327 0.0002153247639332
0.0895647321428571 0.000234662454117718
0.0920958227040816 0.000255124583986082
0.0946269132653061 0.000276742930093666
0.0971580038265306 0.000299549268995843
0.0996890943877551 0.000323575377247987
0.10222018494898 0.000348853031405471
0.104751275510204 0.000375414008023669
0.107282366071429 0.000403290083657954
0.109813456632653 0.000432513034863699
0.112344547193878 0.000463114638196279
0.114875637755102 0.000495126670211067
0.117406728316327 0.000528580907463436
0.119937818877551 0.00056350912650876
0.122468909438776 0.000599943103902411
0.125 0.000637914616199765
};
\addlegendentry{\(\displaystyle R^{3}\)}
\addplot [semithick, red, dashed]
table {%
0.0009765625 6.45238373343172e-38
0.00350765306122449 1.06865083362595e-30
0.00603874362244898 1.24716903003325e-27
0.00856983418367347 1.18103038158122e-25
0.011100924744898 3.4141978502264e-24
0.0136320153061224 4.93060792766607e-23
0.0161631058673469 4.5127601536035e-22
0.0186941964285714 2.99094237805467e-21
0.0212252869897959 1.55853316510415e-20
0.0237563775510204 6.7415920309593e-20
0.0262874681122449 2.51393178715514e-19
0.0288185586734694 8.30536944418912e-19
0.0313496492346939 2.48110155768685e-18
0.0338807397959184 6.80780527760243e-18
0.0364118303571429 1.73688979249614e-17
0.0389429209183673 4.1609300921645e-17
0.0414740114795918 9.4342675855243e-17
0.0440051020408163 2.03778820912649e-16
0.0465361926020408 4.21600047388077e-16
0.0490672831632653 8.39297209229025e-16
0.0515983737244898 1.61394155203166e-15
0.0541294642857143 3.00786333775481e-15
0.0566605548469388 5.44841679841591e-15
0.0591916454081633 9.61623323767233e-15
0.0617227359693878 1.6573217005247e-14
0.0642538265306122 2.79451585528259e-14
0.0667849170918367 4.61783816948331e-14
0.0693160076530612 7.48954883597489e-14
0.0718470982142857 1.19382218448855e-13
0.0743781887755102 1.87245937184079e-13
0.0769092793367347 2.8929655946627e-13
0.0794403698979592 4.40713114327261e-13
0.0819714604591837 6.6257397122196e-13
0.0845025510204082 9.83846359601162e-13
0.0870336415816327 1.44395161983037e-12
0.0895647321428571 2.09604707575433e-12
0.0920958227040816 3.01119401141873e-12
0.0946269132653061 4.28361450970926e-12
0.0971580038265306 6.03727846380465e-12
0.0996890943877551 8.43410655320132e-12
0.10222018494898 1.16841260927663e-11
0.104751275510204 1.60579751855882e-11
0.107282366071429 2.19022173732454e-11
0.109813456632653 2.96580048841412e-11
0.112344547193878 3.98837146078043e-11
0.114875637755102 5.32822781227315e-11
0.117406728316327 7.07340366096895e-11
0.119937818877551 9.33360745076185e-11
0.122468909438776 1.22449123619675e-10
0.125 1.59753283448976e-10
};
\addlegendentry{\(\displaystyle R^{13}\)}
\end{axis}
      \end{tikzpicture}
    \end{subfigure}
    \begin{subfigure}[t]{.29\linewidth}
      \caption{Biharmonic 3D}
      \label{fig:m2m_accuracy_biharmonic_3d}
      \begin{tikzpicture}[scale=0.5]

\definecolor{crimson2143940}{RGB}{214,39,40}
\definecolor{darkgrey176}{RGB}{176,176,176}
\definecolor{darkorange25512714}{RGB}{255,127,14}
\definecolor{forestgreen4416044}{RGB}{44,160,44}
\definecolor{lightgrey204}{RGB}{204,204,204}
\definecolor{mediumpurple148103189}{RGB}{148,103,189}
\definecolor{sienna1408675}{RGB}{140,86,75}
\definecolor{steelblue31119180}{RGB}{31,119,180}

\begin{axis}[
legend cell align={left},
legend style={
  fill opacity=0.8,
  draw opacity=1,
  text opacity=1,
  at={(0.03,0.97)},
  anchor=north west,
  draw=lightgrey204
},
log basis x={10},
log basis y={10},
tick align=outside,
tick pos=left,
x grid style={darkgrey176},
xlabel={Geometric parameter \(\displaystyle R\)},
xmajorgrids,
xmin=0.000766195408102296, xmax=0.159320078414908,
xmode=log,
xtick style={color=black},
xtick={1e-05,0.0001,0.001,0.01,0.1,1,10},
xticklabels={
  \(\displaystyle {10^{-5}}\),
  \(\displaystyle {10^{-4}}\),
  \(\displaystyle {10^{-3}}\),
  \(\displaystyle {10^{-2}}\),
  \(\displaystyle {10^{-1}}\),
  \(\displaystyle {10^{0}}\),
  \(\displaystyle {10^{1}}\)
},
y grid style={darkgrey176},
ymajorgrids,
ymin=4.03962566163905e-18, ymax=1.65414822401268e-15,
ymode=log,
ytick style={color=black},
ytick={1e-19,1e-18,1e-17,1e-16,1e-15,1e-14,1e-13},
yticklabels={
  \(\displaystyle {10^{-19}}\),
  \(\displaystyle {10^{-18}}\),
  \(\displaystyle {10^{-17}}\),
  \(\displaystyle {10^{-16}}\),
  \(\displaystyle {10^{-15}}\),
  \(\displaystyle {10^{-14}}\),
  \(\displaystyle {10^{-13}}\)
}
]
\addplot [semithick, steelblue31119180, mark=*, mark size=3, mark options={solid}]
table {%
0.0009765625 5.30981022749988e-18
0.001953125 8.09107544447069e-18
0.00390625 1.23816018845077e-17
0.0078125 1.78741151878381e-17
0.015625 2.47039674690868e-17
0.03125 3.52879356484682e-17
0.0625 4.98987223391498e-17
0.125 6.83802177274614e-17
};
\addlegendentry{\(\displaystyle p=2\)}
\addplot [semithick, darkorange25512714, mark=*, mark size=3, mark options={solid}]
table {%
0.0009765625 1.55046320784341e-16
0.001953125 1.62036110701717e-16
0.00390625 1.61651714133185e-16
0.0078125 1.63264663695501e-16
0.015625 1.65836598392736e-16
0.03125 1.70345905683628e-16
0.0625 1.80316203525456e-16
0.125 1.94810706480356e-16
};
\addlegendentry{\(\displaystyle p=4\)}
\addplot [semithick, forestgreen4416044, mark=*, mark size=3, mark options={solid}]
table {%
0.0009765625 1.9859710941596e-16
0.001953125 3.02967903456101e-16
0.00390625 3.96518071192259e-16
0.0078125 4.36776878055573e-16
0.015625 4.39387348827432e-16
0.03125 4.4189102311209e-16
0.0625 4.50556449358369e-16
0.125 4.6776954316474e-16
};
\addlegendentry{\(\displaystyle p=6\)}
\addplot [semithick, crimson2143940, mark=*, mark size=3, mark options={solid}]
table {%
0.0009765625 1.98559648318398e-16
0.001953125 3.03072392368867e-16
0.00390625 3.96646960203285e-16
0.0078125 5.07269011086054e-16
0.015625 6.5519703769135e-16
0.03125 6.91100958608867e-16
0.0625 7.01353229780669e-16
0.125 7.34527237705163e-16
};
\addlegendentry{\(\displaystyle p=8\)}
\addplot [semithick, mediumpurple148103189, mark=*, mark size=3, mark options={solid}]
table {%
0.0009765625 1.98556750480463e-16
0.001953125 3.03069811152267e-16
0.00390625 3.96606328738763e-16
0.0078125 5.07192371096403e-16
0.015625 6.66990324966217e-16
0.03125 8.87648384960342e-16
0.0625 9.57864577463267e-16
0.125 9.97375574606799e-16
};
\addlegendentry{\(\displaystyle p=10\)}
\addplot [semithick, sienna1408675, mark=*, mark size=3, mark options={solid}]
table {%
0.0009765625 1.9859710941596e-16
0.001953125 3.02967903456101e-16
0.00390625 3.96601327655438e-16
0.0078125 5.07297071226929e-16
0.015625 6.66837345822328e-16
0.03125 9.12501981087672e-16
0.0625 1.20305743967721e-15
0.125 1.25845168237256e-15
};
\addlegendentry{\(\displaystyle p=12\)}
\end{axis}
      \end{tikzpicture}
    \end{subfigure}
  \caption{Relative difference between potentials evaluated using
  compressed and uncompressed M2M translations.}
  \label{fig:m2m_accuracy}
\end{figure}
As shown in Section~\ref{sec:translations}, our compression approach can
differ from uncompressed Cartesian Taylor expansions \emph{only} when
applied to PDEs with terms of varying order, such as Helmholtz or
Yukawa. The details are in
Proposition~\ref{prop:m2m_single_coeff_equal_order} and
Theorem~\ref{thm:varying}, the latter giving a bound on this difference.

As numerical confirmation of these expectations, we conducted
numerical experiments to measure the relative difference between
potentials evaluated using compressed and uncompressed M2M translations.
In our numerical experiments, we consider an arrangement of sources, targets and the
expansion centers of the multipole expansions as illustrated in Figure~\ref{fig:experiment_points}.
As we change a single geometric parameter $R$, we measure the relative
difference
\begin{align}
  \epsilon_{\text{rel}} = \sqrt{\frac{\sum_{\b x}|\epsilon(\b x)|^2}{
    \sum_{\b x}\left|\sum_{n=1}^{N(p)} \b \rho_{n} \D x^{\bnu(n)} G(\b x - \b c_2)\right|^2}},
    \label{compressed-m2m-error}
\end{align}
where $\epsilon(\b x)$ is defined as in Theorem~\ref{thm:varying}.
$\epsilon_{\text{rel}}$ measures the relative difference between
the potentials evaluated using compressed and uncompressed expansions in the
setting of forming a multipole expansion, translating it to another
multipole expansion, and evaluating the translated expansion.

For the purposes of our experiment,
the sources are arranged in a uniform grid centered at $\b c_1=[R,R,R]^T$ with side
length $2R$ and $50$ grid points per dimension.
Each source is assigned strength from a uniform distribution on $(0, 1)$.
The multipole expansion is first formed about center $\b
c_1$, and then translated to $\b c_2=[0,0,0]^T$.
The target points (at which $\epsilon_{\text{rel}}$ is measured) are
once again located in a uniform grid centered at $[1, 1, 1]^T$
with side length $1$ having 50 grid points per dimension.
Figure~\ref{fig:experiment_points} shows the geometric setup
for a specific $R$, with $R$ varied over $2^{-10},\dots,2^{-3}$.
For these values, the target box is at least one source-box width
away from the source box, as in a typical FMM. The separation
condition of Theorem~\ref{thm:varying} is therefore satisfied.
An analogous two-dimensional experiment is constructed from the
three-dimensional one by dropping the last dimension. We perform this
calculation for different values of the Taylor series order and for
Laplace, Helmholtz and biharmonic kernels/PDEs in two and three
dimensions. All calculations were performed using double-precision
floating-point arithmetic. Figure~\ref{fig:m2m_accuracy} summarizes
the relative difference $\epsilon_{\text{rel}}$ for the three kernels.

For PDEs with only terms of equal order, our techniques introduce no additional
error relative to the conventional Cartesian expansions, 
cf. Proposition~\ref{prop:m2m_single_coeff_equal_order}
and the discussion following Lemma~\ref{lemma:l2l_difference}.
As numerical examples of this case, we use the Laplace and biharmonic
PDEs, which have only terms of equal order.
We observe a relative difference less than $10^{-14}$ in
Figures~\ref{fig:m2m_accuracy_laplace_2d},~\ref{fig:m2m_accuracy_biharmonic_2d},~\ref{fig:m2m_accuracy_laplace_3d}
and~\ref{fig:m2m_accuracy_biharmonic_3d}. While we observe some incipient
growth of $\epsilon_{\text{rel}}$ as $R$ increases in these figures near
the level of machine epsilon, a tolerance of $10^{-14}$ is
maintained in our experiments. We hypothesize that this incipient
growth in $\epsilon_{\text{rel}}$ is the result of amplification of
round-off error.

\begin{figure}
  \centering
  \begin{tikzpicture}[scale=1.3]
  \foreach \i in {0,0.2,...,1.0}{
    \foreach \j in {0,0.2,...,1.0} {
      \draw (\i,\j) node[fill=black!20,circle,scale=0.3] {};
    }
  }
  \draw[thick]
      (0.5-0.5, 0.5-0.5) node[anchor=north]{}
   -- (0.5-0.5, 0.5+0.5) node{}
   -- (0.5+0.5, 0.5+0.5) node{}
   -- (0.5+0.5, 0.5-0.5) node{}
   -- (0.5-0.5, 0.5-0.5) node{}
   -- cycle;
  \draw (0.5,0.5) node[anchor=center,fill=black,circle,scale=0.4,label=below:$\b c_1$] {};
  \draw (0,0) node[anchor=center,fill=black,circle,scale=0.4,label=below:$\b c_2$] {};
  \foreach \i in {0,0.3,0.6,0.9,1.2,1.5,1.8,2.1,2.4,2.7,3.0}{
    \foreach \j in {0,0.3,0.6,0.9,1.2,1.5,1.8,2.1,2.4,2.7,3.0} {
      \draw (1.5+\i,1.5+\j) node[fill=black!20,circle,scale=0.3] {};
    }
  }
  \draw[thick]
      (3-1.5, 3-1.5) node[anchor=north]{}
   -- (3-1.5, 3+1.5) node{}
   -- (3+1.5, 3+1.5) node{}
   -- (3+1.5, 3-1.5) node{}
   -- (3-1.5, 3-1.5) node{}
   -- cycle;
  \draw[dashed] (0.5, 0.5) circle (0.707);
  \draw (0, 0) circle (1.414);
  \draw[<->,shorten >=0.7mm,shorten <=0.4mm] (1.0,1.0)--(0.5,0.5) node[midway, left]{$R_1$};
  \draw[<->,shorten >=0.4mm,shorten <=0.7mm] (-0.0,-0.0)--(-1.0,-1.0) node[midway, left]{$R_2$};
  \draw[->,shorten >=0.5mm] (2.85, 1.0) node[below] {targets ($\b x$)} -- (2.85, 1.5);
  \draw[->,shorten >=0.5mm] (-0.5,0.5) node[left] {sources ($\b y$)} -- (0, 0.5);

  \draw[->] (-2,0) -- (4.75,0) node [right] {$x$};
  \draw[->] (0,-2) -- (0,4.75) node [above] {$y$};
  \draw (0.1,4.5) -- ++ (-0.2,0) node[left] {\footnotesize 1.5};
  \draw (4.5,0.1) -- ++ (0,-0.2) node[below] {\footnotesize 1.5};
  \draw (0.1,1.5) -- ++ (-0.2,0) node[left] {\footnotesize 0.5};
  \draw (1.5,0.1) -- ++ (0,-0.2) node[below] {\footnotesize 0.5};
  \end{tikzpicture}
  \caption{Geometric setting of a numerical experiment to measure the
  difference between compressed and uncompressed multipole expansions.
  Targets are in a box of size $1$ centered at $(1, 1)^T$ and
  sources are in a box of size $2R$ centered at $(R, R)^T$.
  The initial multipole expansion is centered at $\b c_1$ with radius $R_1 = \sqrt{2}R$.
  It is then translated to a multipole expansion centered at $\b c_2$ with radius
  $R_2 = 2 \sqrt{2}R$.}
  \label{fig:experiment_points}
\end{figure}

For PDEs with terms of varying order, we use the example of the Helmholtz equation
\[\Delta u+\kappa^2u=0\] with $\kappa=1$.
The classical upper bound on the truncation error for uncompressed expansions for the chain of translations performed is
$\bigO{R^{p+1}}$ as $R\to 0$, cf.~\eqref{eq:mpole_truncated}.
Under the additional assumption~\eqref{eq:assumption_v}, Theorem~\ref{thm:varying}
shows that the added error introduced by our compressed translation
method obeys a bound of the same order as this estimate
at a fixed expansion order $p$. The experimental data in
Figures~\ref{fig:m2m_accuracy_helmholtz_2d} and~\ref{fig:m2m_accuracy_helmholtz_3d}
show no disagreement with this claim.

Next, we compare the accuracy of compressed and uncompressed M2M
translation for the Helmholtz equation as the wave number $\kappa$
varies. The experimental setup is the same as before with $R=10^{-2}$.
The two relative $\ell^2$ errors are defined as
\[
\begin{aligned}
  \epsilon_{\mathrm{trunc}}^{\mathrm{uncomp}} &= \sqrt{\frac{
    \sum_{\b x}\left|\sum_{n=1}^{N(p)} \b \rho_{n} \D x^{\bnu(n)} G(\b x - \b c_2)
          - \sum_{\b y}G(\b x - \b y)w(\b y)\right|^2}
        {\sum_{\b x}\left|\sum_{\b y}G(\b x - \b y)w(\b y)\right|^2}},\\[1ex]
  \epsilon_{\mathrm{trunc}}^{\mathrm{comp}} &= \sqrt{\frac{
    \sum_{\b x}\left|\sum_{n=1}^{N(p)} \b \sigma_{n} \D x^{\bnu(n)} G(\b x - \b c_2)
          - \sum_{\b y}G(\b x - \b y)w(\b y)\right|^2}
        {\sum_{\b x}\left|\sum_{\b y}G(\b x - \b y)w(\b y)\right|^2}},
\end{aligned}
\]
where $\b\rho$ and $\b\sigma$ are the multipole coefficients obtained by
uncompressed and compressed M2M translation from $\b c_1$ to $\b c_2$, respectively.
Figure~\ref{fig:m2m_accuracy_compare} shows both errors for the
Helmholtz equation in two and three dimensions.
The two errors are comparable in magnitude. In the
higher-wavenumber portion of this sweep, both errors exhibit approximately
$\kappa^{p+1}$ growth. This is an empirical trend over the sampled range,
not an asymptotic claim as $\kappa\to\infty$.

\begin{figure}[ht!]
  \centering
    \begin{subfigure}[t]{.45\linewidth}
      \caption{Helmholtz 2D}
      \label{fig:m2m_accuracy_compare_helmholtz_2d}
      \begin{tikzpicture}[scale=0.8]

\definecolor{crimson2143940}{RGB}{214,39,40}
\definecolor{darkgrey176}{RGB}{176,176,176}
\definecolor{darkorange25512714}{RGB}{255,127,14}
\definecolor{forestgreen4416044}{RGB}{44,160,44}
\definecolor{mediumpurple148103189}{RGB}{148,103,189}
\definecolor{sienna1408675}{RGB}{140,86,75}
\definecolor{steelblue31119180}{RGB}{31,119,180}

\begin{axis}[
log basis x={10},
log basis y={10},
tick align=outside,
tick pos=left,
x grid style={darkgrey176},
xlabel={Wave number \(\displaystyle \kappa\)},
xmajorgrids,
xmin=0.822340159426889, xmax=60.8020895329329,
xmode=log,
xtick style={color=black},
xtick={0.01,0.1,1,10,100,1000},
xticklabels={
  \(\displaystyle {10^{-2}}\),
  \(\displaystyle {10^{-1}}\),
  \(\displaystyle {10^{0}}\),
  \(\displaystyle {10^{1}}\),
  \(\displaystyle {10^{2}}\),
  \(\displaystyle {10^{3}}\)
},
y grid style={darkgrey176},
ylabel={Error},
ymajorgrids,
ymin=2.95021818463545e-16, ymax=0.427425655169706,
ymode=log,
ytick style={color=black},
ytick={1e-18,1e-16,1e-14,1e-12,1e-10,1e-08,1e-06,0.0001,0.01,1,100},
yticklabels={
  \(\displaystyle {10^{-18}}\),
  \(\displaystyle {10^{-16}}\),
  \(\displaystyle {10^{-14}}\),
  \(\displaystyle {10^{-12}}\),
  \(\displaystyle {10^{-10}}\),
  \(\displaystyle {10^{-8}}\),
  \(\displaystyle {10^{-6}}\),
  \(\displaystyle {10^{-4}}\),
  \(\displaystyle {10^{-2}}\),
  \(\displaystyle {10^{0}}\),
  \(\displaystyle {10^{2}}\)
}
]
\addplot [semithick, steelblue31119180, mark=*, mark size=3, mark options={solid}]
table {%
1 4.27349833236685e-07
1.54445210494638 1.34473073565044e-06
2.3853323044733 5.49395726202654e-06
3.68403149864039 2.12704457201833e-05
5.68981020276391 7.99372452675384e-05
8.7876393444041 0.000296803611366027
13.5720880829745 0.00109631778024051
20.9614400082677 0.00403763357555936
32.3739401434763 0.0148223933717348
50 0.0540655670236004
};
\addplot [semithick, steelblue31119180, dashed, mark=x, mark size=3, mark options={solid}]
table {%
1 5.1302235209622e-07
1.54445210494638 2.1643277855756e-06
2.3853323044733 8.75247168022381e-06
3.68403149864039 3.34062834294464e-05
5.68981020276391 0.000124763929471242
8.7876393444041 0.000462296550251908
13.5720880829745 0.00170854569281591
20.9614400082677 0.00631508056903982
32.3739401434763 0.0234071367807142
50 0.0874427745392864
};
\addplot [semithick, darkorange25512714, mark=*, mark size=3, mark options={solid}]
table {%
1 8.92576190707788e-11
1.54445210494638 1.09222360261604e-10
2.3853323044733 2.1680903525258e-10
3.68403149864039 2.04402728299413e-09
5.68981020276391 2.00777840872503e-08
8.7876393444041 1.84620068582076e-07
13.5720880829745 1.65133029084997e-06
20.9614400082677 1.458502238255e-05
32.3739401434763 0.000127732970335718
50 0.00110577275218825
};
\addplot [semithick, darkorange25512714, dashed, mark=x, mark size=3, mark options={solid}]
table {%
1 8.50219026313127e-11
1.54445210494638 1.34180232775053e-10
2.3853323044733 8.94235463573077e-10
3.68403149864039 8.52376345850836e-09
5.68981020276391 7.82119772556769e-08
8.7876393444041 7.00687404624178e-07
13.5720880829745 6.21481492931985e-06
20.9614400082677 5.49725259276424e-05
32.3739401434763 0.000487487202765393
50 0.00436623494682955
};
\addplot [semithick, forestgreen4416044, mark=*, mark size=3, mark options={solid}]
table {%
1 4.65628439665195e-14
1.54445210494638 5.73604605512703e-14
2.3853323044733 7.19061020597157e-14
3.68403149864039 1.19312493387791e-13
5.68981020276391 2.0643713775985e-12
8.7876393444041 4.51421972354117e-11
13.5720880829745 9.53915213052887e-10
20.9614400082677 2.00567287310757e-08
32.3739401434763 4.22751959046529e-07
50 8.99705689341725e-06
};
\addplot [semithick, forestgreen4416044, dashed, mark=x, mark size=3, mark options={solid}]
table {%
1 4.53941653508643e-14
1.54445210494638 5.44370263159985e-14
2.3853323044733 8.54136645279963e-14
3.68403149864039 1.23089147323295e-12
5.68981020276391 2.71838734579295e-11
8.7876393444041 5.900673768149e-10
13.5720880829745 1.25931222130461e-08
20.9614400082677 2.66857254535736e-07
32.3739401434763 5.65986071718357e-06
50 0.000121286574013021
};
\addplot [semithick, crimson2143940, mark=*, mark size=3, mark options={solid}]
table {%
1 1.72021354506271e-15
1.54445210494638 1.62089050663008e-15
2.3853323044733 1.61042265978009e-15
3.68403149864039 1.44496853862544e-15
5.68981020276391 2.21729866259354e-15
8.7876393444041 4.22722751009592e-14
13.5720880829745 2.11015807734241e-12
20.9614400082677 1.05441892415612e-10
32.3739401434763 5.30992189082236e-09
50 2.71825225218978e-07
};
\addplot [semithick, crimson2143940, dashed, mark=x, mark size=3, mark options={solid}]
table {%
1 2.76337773468556e-15
1.54445210494638 2.79478499274244e-15
2.3853323044733 2.804818792409e-15
3.68403149864039 2.78712987466365e-15
5.68981020276391 5.85735389838912e-15
8.7876393444041 3.19616904589991e-13
13.5720880829745 1.6464733591483e-11
20.9614400082677 8.3646198452169e-10
32.3739401434763 4.24336737003574e-08
50 2.17394067038893e-06
};
\addplot [semithick, mediumpurple148103189, mark=*, mark size=3, mark options={solid}]
table {%
1 1.71908832588983e-15
1.54445210494638 1.62014717381795e-15
2.3853323044733 1.60316383376303e-15
3.68403149864039 1.44208477727894e-15
5.68981020276391 2.14594340038401e-15
8.7876393444041 2.00236854376269e-15
13.5720880829745 2.63986525875894e-15
20.9614400082677 2.12584866437587e-13
32.3739401434763 2.5572864928893e-11
50 3.10623887686657e-09
};
\addplot [semithick, mediumpurple148103189, dashed, mark=x, mark size=3, mark options={solid}]
table {%
1 2.76425550888658e-15
1.54445210494638 2.79544757597063e-15
2.3853323044733 2.79715357220972e-15
3.68403149864039 2.81755190045949e-15
5.68981020276391 2.88525512793713e-15
8.7876393444041 3.06412571950241e-15
13.5720880829745 1.62743126189892e-14
20.9614400082677 1.84393470225659e-12
32.3739401434763 2.23650517101763e-10
50 2.73831996885705e-08
};
\addplot [semithick, sienna1408675, mark=*, mark size=3, mark options={solid}]
table {%
1 1.71259406038412e-15
1.54445210494638 1.62149979284748e-15
2.3853323044733 1.60440127378031e-15
3.68403149864039 1.44266568968008e-15
5.68981020276391 2.14464155982491e-15
8.7876393444041 2.00769897297192e-15
13.5720880829745 2.09996783298953e-15
20.9614400082677 2.78395078664838e-15
32.3739401434763 5.35689265844073e-14
50 1.55023035389205e-11
};
\addplot [semithick, sienna1408675, dashed, mark=x, mark size=3, mark options={solid}]
table {%
1 2.76654312980835e-15
1.54445210494638 2.79022495221361e-15
2.3853323044733 2.79621959520655e-15
3.68403149864039 2.82115537817921e-15
5.68981020276391 2.87939070115485e-15
8.7876393444041 3.03240797815361e-15
13.5720880829745 3.31263311186742e-15
20.9614400082677 4.13993585954543e-15
32.3739401434763 8.75396936417368e-13
50 2.56410535823675e-10
};
\end{axis}
      \end{tikzpicture}
    \end{subfigure}
    \begin{subfigure}[t]{.45\linewidth}
      \caption{Helmholtz 3D}
      \label{fig:m2m_accuracy_compare_helmholtz_3d}
      \begin{tikzpicture}[scale=0.8]

\definecolor{crimson2143940}{RGB}{214,39,40}
\definecolor{darkgrey176}{RGB}{176,176,176}
\definecolor{darkorange25512714}{RGB}{255,127,14}
\definecolor{forestgreen4416044}{RGB}{44,160,44}
\definecolor{mediumpurple148103189}{RGB}{148,103,189}
\definecolor{sienna1408675}{RGB}{140,86,75}
\definecolor{steelblue31119180}{RGB}{31,119,180}

\begin{axis}[
log basis x={10},
log basis y={10},
tick align=outside,
tick pos=left,
x grid style={darkgrey176},
xlabel={Wave number \(\displaystyle \kappa\)},
xmajorgrids,
xmin=0.822340159426889, xmax=60.8020895329329,
xmode=log,
xtick style={color=black},
xtick={0.01,0.1,1,10,100,1000},
xticklabels={
  \(\displaystyle {10^{-2}}\),
  \(\displaystyle {10^{-1}}\),
  \(\displaystyle {10^{0}}\),
  \(\displaystyle {10^{1}}\),
  \(\displaystyle {10^{2}}\),
  \(\displaystyle {10^{3}}\)
},
y grid style={darkgrey176},
ymajorgrids,
ymin=2.0350615455135e-15, ymax=0.626092460459022,
ymode=log,
ytick style={color=black},
ytick={1e-17,1e-15,1e-13,1e-11,1e-09,1e-07,1e-05,0.001,0.1,10,1000},
yticklabels={
  \(\displaystyle {10^{-17}}\),
  \(\displaystyle {10^{-15}}\),
  \(\displaystyle {10^{-13}}\),
  \(\displaystyle {10^{-11}}\),
  \(\displaystyle {10^{-9}}\),
  \(\displaystyle {10^{-7}}\),
  \(\displaystyle {10^{-5}}\),
  \(\displaystyle {10^{-3}}\),
  \(\displaystyle {10^{-1}}\),
  \(\displaystyle {10^{1}}\),
  \(\displaystyle {10^{3}}\)
}
]
\addplot [semithick, steelblue31119180, mark=*, mark size=3, mark options={solid}]
table {%
1 1.43941888316255e-06
1.54445210494638 2.50915558588515e-06
2.3853323044733 9.63069435518459e-06
3.68403149864039 3.77777394346464e-05
5.68981020276391 0.00014296373014685
8.7876393444041 0.000532364311639951
13.5720880829745 0.00196867486524951
20.9614400082677 0.0072527445670749
32.3739401434763 0.026618678254132
50 0.0969985338627231
};
\addplot [semithick, steelblue31119180, dashed, mark=x, mark size=3, mark options={solid}]
table {%
1 1.28570057599681e-06
1.54445210494638 3.13379229048187e-06
2.3853323044733 1.32122984272007e-05
3.68403149864039 5.18391670201774e-05
5.68981020276391 0.000195709201937397
8.7876393444041 0.000728268517414281
13.5720880829745 0.00269571857994949
20.9614400082677 0.00996574144844758
32.3739401434763 0.036901028823745
50 0.13743275421405
};
\addplot [semithick, darkorange25512714, mark=*, mark size=3, mark options={solid}]
table {%
1 3.30422364638287e-10
1.54445210494638 3.26078680571486e-10
2.3853323044733 5.16497345751431e-10
3.68403149864039 5.72017084718266e-09
5.68981020276391 5.76635303367106e-08
8.7876393444041 5.33682180193109e-07
13.5720880829745 4.78488171574113e-06
20.9614400082677 4.23170431652027e-05
32.3739401434763 0.000371157544214468
50 0.00322264168218655
};
\addplot [semithick, darkorange25512714, dashed, mark=x, mark size=3, mark options={solid}]
table {%
1 3.10020913355545e-10
1.54445210494638 3.10602381965725e-10
2.3853323044733 1.55116486985997e-09
3.68403149864039 1.61795236550906e-08
5.68981020276391 1.53222411409167e-07
8.7876393444041 1.38898646864428e-06
13.5720880829745 1.23765274645912e-05
20.9614400082677 0.000109650515962153
32.3739401434763 0.0009723311132893
50 0.0086940349624897
};
\addplot [semithick, forestgreen4416044, mark=*, mark size=3, mark options={solid}]
table {%
1 1.32140493414557e-13
1.54445210494638 1.34470001136122e-13
2.3853323044733 1.33474262943265e-13
3.68403149864039 2.91608104359608e-13
5.68981020276391 7.01811491780355e-12
8.7876393444041 1.65205235881962e-10
13.5720880829745 3.61908849778984e-09
20.9614400082677 7.69036124695339e-08
32.3739401434763 1.60368076440213e-06
50 3.26715500417581e-05
};
\addplot [semithick, forestgreen4416044, dashed, mark=x, mark size=3, mark options={solid}]
table {%
1 1.27365756750169e-13
1.54445210494638 1.22354310754274e-13
2.3853323044733 1.65640756832775e-13
3.68403149864039 2.84358905438527e-12
5.68981020276391 6.66141029222545e-11
8.7876393444041 1.4774000877262e-09
13.5720880829745 3.17840921639569e-08
20.9614400082677 6.75529842356365e-07
32.3739401434763 1.43376714297815e-05
50 0.000306957527336588
};
\addplot [semithick, crimson2143940, mark=*, mark size=3, mark options={solid}]
table {%
1 1.04698536328952e-14
1.54445210494638 1.31411379819484e-14
2.3853323044733 1.70501686974936e-14
3.68403149864039 9.58932151783303e-15
5.68981020276391 1.86812661307645e-14
8.7876393444041 1.26010139116176e-13
13.5720880829745 5.84933669173743e-12
20.9614400082677 2.86319711622716e-10
32.3739401434763 1.43647972984366e-08
50 7.45021672963224e-07
};
\addplot [semithick, crimson2143940, dashed, mark=x, mark size=3, mark options={solid}]
table {%
1 9.78548170385142e-15
1.54445210494638 9.2736519746998e-15
2.3853323044733 9.55508914975648e-15
3.68403149864039 9.63556280320399e-15
5.68981020276391 2.04510454381415e-14
8.7876393444041 1.02248505344497e-12
13.5720880829745 5.32624515849462e-11
20.9614400082677 2.71919037784917e-09
32.3739401434763 1.38164444533898e-07
50 7.07617365170237e-06
};
\addplot [semithick, mediumpurple148103189, mark=*, mark size=3, mark options={solid}]
table {%
1 1.04763744020535e-14
1.54445210494638 1.31505181001678e-14
2.3853323044733 1.70557148893481e-14
3.68403149864039 9.58789028662763e-15
5.68981020276391 1.89737370613934e-14
8.7876393444041 2.29697528563975e-14
13.5720880829745 2.05328285858263e-14
20.9614400082677 1.33891069719555e-12
32.3739401434763 1.61214257905086e-10
50 1.96887611571401e-08
};
\addplot [semithick, mediumpurple148103189, dashed, mark=x, mark size=3, mark options={solid}]
table {%
1 9.78670722128295e-15
1.54445210494638 9.2723108565765e-15
2.3853323044733 9.55396041633857e-15
3.68403149864039 9.65800049022389e-15
5.68981020276391 9.65694934387095e-15
8.7876393444041 9.75623185022454e-15
13.5720880829745 6.47519500635603e-14
20.9614400082677 7.7920924192028e-12
32.3739401434763 9.48106613849096e-10
50 1.16124800875709e-07
};
\addplot [semithick, sienna1408675, mark=*, mark size=3, mark options={solid}]
table {%
1 1.04715252110722e-14
1.54445210494638 1.31473036266936e-14
2.3853323044733 1.70495440812163e-14
3.68403149864039 9.58825894830028e-15
5.68981020276391 1.89522570319286e-14
8.7876393444041 2.30152711948718e-14
13.5720880829745 1.4473555885765e-14
20.9614400082677 1.02752616523993e-14
32.3739401434763 7.95341054277102e-13
50 2.29736532458071e-10
};
\addplot [semithick, sienna1408675, dashed, mark=x, mark size=3, mark options={solid}]
table {%
1 9.7849922293428e-15
1.54445210494638 9.27098272535261e-15
2.3853323044733 9.5531864643512e-15
3.68403149864039 9.65717078325766e-15
5.68981020276391 9.65623821711033e-15
8.7876393444041 9.70624411199057e-15
13.5720880829745 9.75600008822635e-15
20.9614400082677 1.93108194561652e-14
32.3739401434763 4.89472405727035e-12
50 1.43359503042018e-09
};
\end{axis}
      \end{tikzpicture}
    \end{subfigure}
    \par\smallskip
    \begin{tikzpicture}[font=\footnotesize]
\draw[semithick, solid, mark=*, mark size=2.4pt, mark options={solid}, mark indices={2}] plot coordinates {(0.6,0.6) (1.05,0.6) (1.5,0.6)} node[right] {compressed, $\epsilon_{\mathrm{trunc}}^{\mathrm{comp}}$};
\draw[semithick, dashed, mark=x, mark size=2.4pt, mark options={solid}, mark indices={2}] plot coordinates {(5.0,0.6) (5.45,0.6) (5.9,0.6)} node[right] {uncompressed, $\epsilon_{\mathrm{trunc}}^{\mathrm{uncomp}}$};
\definecolor{order0}{HTML}{1f77b4}
\draw[line width=1.2pt, order0] (0.0,0) -- (0.45,0) node[right, black] {$p=2$};
\definecolor{order1}{HTML}{ff7f0e}
\draw[line width=1.2pt, order1] (1.65,0) -- (2.1,0) node[right, black] {$p=4$};
\definecolor{order2}{HTML}{2ca02c}
\draw[line width=1.2pt, order2] (3.3,0) -- (3.75,0) node[right, black] {$p=6$};
\definecolor{order3}{HTML}{d62728}
\draw[line width=1.2pt, order3] (4.949999999999999,0) -- (5.3999999999999995,0) node[right, black] {$p=8$};
\definecolor{order4}{HTML}{9467bd}
\draw[line width=1.2pt, order4] (6.6,0) -- (7.05,0) node[right, black] {$p=10$};
\definecolor{order5}{HTML}{8c564b}
\draw[line width=1.2pt, order5] (8.25,0) -- (8.7,0) node[right, black] {$p=12$};
\end{tikzpicture}
  \caption{Relative $\ell^2$ errors of uncompressed (dashed) and
    compressed (solid) Taylor M2M translation against direct evaluation
    for the Helmholtz problem with varying wavenumber $\kappa$.}
  \label{fig:m2m_accuracy_compare}
\end{figure}

\subsection{Accuracy of Local-to-Local Translation}
\label{sec:l2l-experiment}
We repeat the experiment of Section~\ref{sec:m2m-experiment} for
local-to-local translation, exchanging the roles of the sources and
targets in Figure~\ref{fig:experiment_points}. We first form a local expansion
about $\b c_2$ and then translate it to $\b c_1$.
The relative difference is defined as
\[
  \epsilon_{\text{rel}} = \sqrt{\frac{\sum_{\b x}\left|\sum_{n=1}^{N(p)}
    (\b\mu^{c}_{n}-\b\mu_{n})(\b x-\b c_1)^{\bnu(n)}\right|^2}
    {\sum_{\b x}\left|\sum_{n=1}^{N(p)} \b\mu_{n}(\b x-\b c_1)^{\bnu(n)}\right|^2}},
\]
where $\b\mu$ and $\b\mu^c$ are the local expansion coefficients obtained by
uncompressed and compressed L2L translation, respectively.
In Figure~\ref{fig:l2l_accuracy}, for the Laplace and biharmonic kernels,
which have only terms of equal order, the relative difference stays at the
level of machine precision.
For the Helmholtz kernel with $\kappa=1$, the relative difference is
$\bigO{R^{p+1}}$, as one may expect from Theorem~\ref{thm:l2l_error}.

\begin{figure}[ht!]
  \centering
    \begin{subfigure}[t]{.29\linewidth}
      \caption{Laplace 2D}
      \label{fig:l2l_accuracy_laplace_2d}
      \begin{tikzpicture}[scale=0.5]

\definecolor{crimson2143940}{RGB}{214,39,40}
\definecolor{darkgrey176}{RGB}{176,176,176}
\definecolor{darkorange25512714}{RGB}{255,127,14}
\definecolor{forestgreen4416044}{RGB}{44,160,44}
\definecolor{lightgrey204}{RGB}{204,204,204}
\definecolor{mediumpurple148103189}{RGB}{148,103,189}
\definecolor{sienna1408675}{RGB}{140,86,75}
\definecolor{steelblue31119180}{RGB}{31,119,180}

\begin{axis}[
legend cell align={left},
legend style={
  fill opacity=0.8,
  draw opacity=1,
  text opacity=1,
  at={(0.03,0.97)},
  anchor=north west,
  draw=lightgrey204
},
log basis x={10},
log basis y={10},
tick align=outside,
tick pos=left,
x grid style={darkgrey176},
xlabel={Geometric parameter \(\displaystyle R\)},
xmajorgrids,
xmin=0.000766195408102296, xmax=0.159320078414908,
xmode=log,
xtick style={color=black},
xtick={1e-05,0.0001,0.001,0.01,0.1,1,10},
xticklabels={
  \(\displaystyle {10^{-5}}\),
  \(\displaystyle {10^{-4}}\),
  \(\displaystyle {10^{-3}}\),
  \(\displaystyle {10^{-2}}\),
  \(\displaystyle {10^{-1}}\),
  \(\displaystyle {10^{0}}\),
  \(\displaystyle {10^{1}}\)
},
y grid style={darkgrey176},
ylabel={\(\displaystyle \epsilon_{\mathrm{rel}}\)},
ymajorgrids,
ymin=7.32919229554917e-19, ymax=6.82058390108146e-16,
ymode=log,
ytick style={color=black},
ytick={1e-20,1e-19,1e-18,1e-17,1e-16,1e-15,1e-14},
yticklabels={
  \(\displaystyle {10^{-20}}\),
  \(\displaystyle {10^{-19}}\),
  \(\displaystyle {10^{-18}}\),
  \(\displaystyle {10^{-17}}\),
  \(\displaystyle {10^{-16}}\),
  \(\displaystyle {10^{-15}}\),
  \(\displaystyle {10^{-14}}\)
}
]
\addplot [semithick, steelblue31119180, mark=*, mark size=3, mark options={solid}]
table {%
0.0009765625 1.6422738752882e-16
0.001953125 1.63180527333082e-16
0.00390625 1e-18
0.0078125 1e-18
0.015625 1e-18
0.03125 1e-18
0.0625 1e-18
0.125 1.70604472208397e-17
};
\addlegendentry{\(\displaystyle p=2\)}
\addplot [semithick, darkorange25512714, mark=*, mark size=3, mark options={solid}]
table {%
0.0009765625 1.66291923647129e-16
0.001953125 2.82211719719996e-16
0.00390625 1e-18
0.0078125 1.68535770617277e-16
0.015625 1.72581305993968e-16
0.03125 3.03020170002148e-16
0.0625 3.28452286769491e-16
0.125 2.12646347354905e-16
};
\addlegendentry{\(\displaystyle p=4\)}
\addplot [semithick, forestgreen4416044, mark=*, mark size=3, mark options={solid}]
table {%
0.0009765625 1.65623772279518e-16
0.001953125 1.65581138602969e-16
0.00390625 1e-18
0.0078125 1.67690444813104e-16
0.015625 1.72000199929911e-16
0.03125 1.66672588228511e-16
0.0625 1.81405399974297e-16
0.125 4.1961397306393e-16
};
\addlegendentry{\(\displaystyle p=6\)}
\addplot [semithick, crimson2143940, mark=*, mark size=3, mark options={solid}]
table {%
0.0009765625 1.65623772279518e-16
0.001953125 2.81206808328081e-16
0.00390625 1.66588461922052e-16
0.0078125 2.84266595158896e-16
0.015625 1.7302594781735e-16
0.03125 3.06712908461593e-17
0.0625 3.30541386433842e-16
0.125 4.23412817166199e-16
};
\addlegendentry{\(\displaystyle p=8\)}
\addplot [semithick, mediumpurple148103189, mark=*, mark size=3, mark options={solid}]
table {%
0.0009765625 1.65623772279518e-16
0.001953125 2.79370435997801e-16
0.00390625 4.74998636245837e-17
0.0078125 1.64967062796196e-16
0.015625 2.89582467957077e-16
0.03125 2.99513683805367e-16
0.0625 6.67891812760717e-17
0.125 4.34230846994864e-16
};
\addlegendentry{\(\displaystyle p=10\)}
\addplot [semithick, sienna1408675, mark=*, mark size=3, mark options={solid}]
table {%
0.0009765625 1.65623772279518e-16
0.001953125 1.65321891319179e-16
0.00390625 1.6622342181098e-16
0.0078125 2.8210245543822e-16
0.015625 1.68816633384797e-16
0.03125 1.68024175779392e-16
0.0625 4.99893709789529e-16
0.125 1.56083131810652e-16
};
\addlegendentry{\(\displaystyle p=12\)}
\end{axis}
      \end{tikzpicture}
    \end{subfigure}\hspace{3mm}
    \begin{subfigure}[t]{.29\linewidth}
      \caption{Helmholtz 2D ($\kappa=1$)}
      \label{fig:l2l_accuracy_helmholtz_2d}
      \begin{tikzpicture}[scale=0.5]

\definecolor{crimson2143940}{RGB}{214,39,40}
\definecolor{darkgrey176}{RGB}{176,176,176}
\definecolor{darkorange25512714}{RGB}{255,127,14}
\definecolor{forestgreen4416044}{RGB}{44,160,44}
\definecolor{grey}{RGB}{128,128,128}
\definecolor{lightgrey204}{RGB}{204,204,204}
\definecolor{mediumpurple148103189}{RGB}{148,103,189}
\definecolor{sienna1408675}{RGB}{140,86,75}
\definecolor{steelblue31119180}{RGB}{31,119,180}

\begin{axis}[
legend cell align={left},
legend style={
  fill opacity=0.8,
  draw opacity=1,
  text opacity=1,
  at={(0.03,0.97)},
  anchor=north west,
  draw=lightgrey204
},
log basis x={10},
log basis y={10},
tick align=outside,
tick pos=left,
x grid style={darkgrey176},
xlabel={Geometric parameter \(\displaystyle R\)},
xmajorgrids,
xmin=0.000766195408102296, xmax=0.159320078414908,
xmode=log,
xtick style={color=black},
xtick={1e-05,0.0001,0.001,0.01,0.1,1,10},
xticklabels={
  \(\displaystyle {10^{-5}}\),
  \(\displaystyle {10^{-4}}\),
  \(\displaystyle {10^{-3}}\),
  \(\displaystyle {10^{-2}}\),
  \(\displaystyle {10^{-1}}\),
  \(\displaystyle {10^{0}}\),
  \(\displaystyle {10^{1}}\)
},
y grid style={darkgrey176},
ymajorgrids,
ymin=1e-16, ymax=0.00303511717268405,
ymode=log,
ytick style={color=black},
ytick={1e-18,1e-16,1e-14,1e-12,1e-10,1e-08,1e-06,0.0001,0.01,1},
yticklabels={
  \(\displaystyle {10^{-18}}\),
  \(\displaystyle {10^{-16}}\),
  \(\displaystyle {10^{-14}}\),
  \(\displaystyle {10^{-12}}\),
  \(\displaystyle {10^{-10}}\),
  \(\displaystyle {10^{-8}}\),
  \(\displaystyle {10^{-6}}\),
  \(\displaystyle {10^{-4}}\),
  \(\displaystyle {10^{-2}}\),
  \(\displaystyle {10^{0}}\)
}
]
\addplot [semithick, steelblue31119180, mark=*, mark size=3, mark options={solid}]
table {%
0.0009765625 3.36345746831529e-10
0.001953125 2.69035441429055e-09
0.00390625 2.15162371521829e-08
0.0078125 1.72023919328638e-07
0.015625 1.37448127199359e-06
0.03125 1.09679460530137e-05
0.0625 8.72745117238716e-05
0.125 0.000689636300194168
};
\addlegendentry{\(\displaystyle p=2\)}
\addplot [semithick, darkorange25512714, mark=*, mark size=3, mark options={solid}]
table {%
0.0009765625 1.50977721918335e-16
0.001953125 3.81207172639413e-15
0.00390625 1.21463254023835e-13
0.0078125 3.8856094309288e-12
0.015625 1.24116889314024e-10
0.03125 3.95728352430379e-09
0.0625 1.2568460295276e-07
0.125 3.95745234499194e-06
};
\addlegendentry{\(\displaystyle p=4\)}
\addplot [semithick, forestgreen4416044, mark=*, mark size=3, mark options={solid}]
table {%
0.0009765625 9.50093674574056e-17
0.001953125 9.28553294570771e-17
0.00390625 9.2792556334321e-17
0.0078125 3.74573266120336e-16
0.015625 4.30943586729501e-14
0.03125 5.6050472644496e-12
0.0625 7.38904538106098e-10
0.125 9.95728834550446e-08
};
\addlegendentry{\(\displaystyle p=6\)}
\addplot [semithick, crimson2143940, mark=*, mark size=3, mark options={solid}]
table {%
0.0009765625 1.49125333372881e-16
0.001953125 1.28921256743776e-16
0.00390625 1.91955488222874e-16
0.0078125 1.52169585533611e-16
0.015625 1.94424195505822e-16
0.03125 1.61737663222002e-14
0.0625 8.61944255817461e-12
0.125 4.74687355430338e-09
};
\addlegendentry{\(\displaystyle p=8\)}
\addplot [semithick, mediumpurple148103189, mark=*, mark size=3, mark options={solid}]
table {%
0.0009765625 1.53764862655446e-16
0.001953125 1.39534768637598e-16
0.00390625 2.02320742510108e-16
0.0078125 1.36092545187414e-16
0.015625 1.51284949208579e-16
0.03125 2.21420553691058e-16
0.0625 1.56133061015093e-13
0.125 3.49414439141523e-10
};
\addlegendentry{\(\displaystyle p=10\)}
\addplot [semithick, sienna1408675, mark=*, mark size=3, mark options={solid}]
table {%
0.0009765625 1.62882264115418e-16
0.001953125 1.47868126476934e-16
0.00390625 2.08288497428232e-16
0.0078125 1.50425638909417e-16
0.015625 2.06088234946167e-16
0.03125 1.49573052857584e-16
0.0625 3.74562423765365e-15
0.125 3.38397898864036e-11
};
\addlegendentry{\(\displaystyle p=12\)}
\addplot [semithick, grey, dashed]
table {%
0.0009765625 3.28844213578304e-10
0.00350765306122449 1.52384209923168e-08
0.00603874362244898 7.77552728585572e-08
0.00856983418367347 2.22232376722031e-07
0.011100924744898 4.8302270949247e-07
0.0136320153061224 8.94479248079606e-07
0.0161631058673469 1.49095496939317e-06
0.0186941964285714 2.30680285034289e-06
0.0212252869897959 3.37637586783851e-06
0.0237563775510204 4.73402699878975e-06
0.0262874681122449 6.41410922010634e-06
0.0288185586734694 8.45097550869802e-06
0.0313496492346939 1.08789788414745e-05
0.0338807397959184 1.37324721953456e-05
0.0364118303571429 1.70458085472209e-05
0.0389429209183673 2.08533408740102e-05
0.0414740114795918 2.51894221526233e-05
0.0440051020408163 3.00884053599699e-05
0.0465361926020408 3.55846434729596e-05
0.0490672831632653 4.17124894685023e-05
0.0515983737244898 4.85062963235077e-05
0.0541294642857143 5.60004170148855e-05
0.0566605548469388 6.42292045195454e-05
0.0591916454081633 7.32270118143972e-05
0.0617227359693878 8.30281918763505e-05
0.0642538265306122 9.36670976823152e-05
0.0667849170918367 0.000105178082209201
0.0693160076530612 0.000117595498433918
0.0718470982142857 0.000130953699333375
0.0743781887755102 0.000145287037884482
0.0769092793367347 0.000160629867064149
0.0794403698979592 0.000177016539849286
0.0819714604591837 0.000194481409216803
0.0845025510204082 0.000213058828143608
0.0870336415816327 0.000232783149606613
0.0895647321428571 0.000253688726582726
0.0920958227040816 0.000275809912048858
0.0946269132653061 0.000299181058981917
0.0971580038265306 0.000323836520358815
0.0996890943877551 0.00034981064915646
0.10222018494898 0.000377137798351762
0.104751275510204 0.000405852320921632
0.107282366071429 0.000435988569842978
0.109813456632653 0.000467580898092711
0.112344547193878 0.00050066365864774
0.114875637755102 0.000535271204484975
0.117406728316327 0.000571437888581325
0.119937818877551 0.000609198063913701
0.122468909438776 0.000648586083459012
0.125 0.000689636300194168
};
\addlegendentry{\(\displaystyle R^{3}\)}
\addplot [semithick, red, dashed]
table {%
0.0009765625 1.36677822885259e-38
0.00350765306122449 2.26367301448188e-31
0.00603874362244898 2.64181974967886e-28
0.00856983418367347 2.50172134802677e-26
0.011100924744898 7.23213541455459e-25
0.0136320153061224 1.0444275866026e-23
0.0161631058673469 9.55916849461508e-23
0.0186941964285714 6.33557316948899e-22
0.0212252869897959 3.30136781538896e-21
0.0237563775510204 1.42803986811563e-20
0.0262874681122449 5.32514397384838e-20
0.0288185586734694 1.75928751417545e-19
0.0313496492346939 5.2556012362505e-19
0.0338807397959184 1.44206551006631e-18
0.0364118303571429 3.6791723065073e-18
0.0389429209183673 8.8139033521544e-18
0.0414740114795918 1.99841672066926e-17
0.0440051020408163 4.31655133096878e-17
0.0465361926020408 8.93055636272228e-17
0.0490672831632653 1.77784397286747e-16
0.0515983737244898 3.41873680656654e-16
0.0541294642857143 6.37141604598998e-16
0.0566605548469388 1.1541126147234e-15
0.0591916454081633 2.03696165259361e-15
0.0617227359693878 3.51062694356736e-15
0.0642538265306122 5.91949206522504e-15
0.0667849170918367 9.78175033470531e-15
0.0693160076530612 1.58647605533752e-14
0.0718470982142857 2.52881762506764e-14
0.0743781887755102 3.96634299752325e-14
0.0769092793367347 6.12803353761727e-14
0.0794403698979592 9.33541950878295e-14
0.0819714604591837 1.40349941398934e-13
0.0845025510204082 2.0840356686653e-13
0.0870336415816327 3.05865509404683e-13
0.0895647321428571 4.43995835980369e-13
0.0920958227040816 6.37847125603224e-13
0.0946269132653061 9.0737800083595e-13
0.0971580038265306 1.27884842358254e-12
0.0996890943877551 1.78655729308392e-12
0.10222018494898 2.47499371186159e-12
0.104751275510204 3.40148568185731e-12
0.107282366071429 4.63944413507898e-12
0.109813456632653 6.28231628209267e-12
0.112344547193878 8.44838048445151e-12
0.114875637755102 1.12865354465032e-11
0.117406728316327 1.49832597178109e-11
0.119937818877551 1.97709435572778e-11
0.122468909438776 2.59378244102727e-11
0.125 3.38397898864036e-11
};
\addlegendentry{\(\displaystyle R^{13}\)}
\end{axis}
      \end{tikzpicture}
    \end{subfigure}
    \begin{subfigure}[t]{.29\linewidth}
      \caption{Biharmonic 2D}
      \label{fig:l2l_accuracy_biharmonic_2d}
      \begin{tikzpicture}[scale=0.5]

\definecolor{crimson2143940}{RGB}{214,39,40}
\definecolor{darkgrey176}{RGB}{176,176,176}
\definecolor{darkorange25512714}{RGB}{255,127,14}
\definecolor{forestgreen4416044}{RGB}{44,160,44}
\definecolor{lightgrey204}{RGB}{204,204,204}
\definecolor{mediumpurple148103189}{RGB}{148,103,189}
\definecolor{sienna1408675}{RGB}{140,86,75}
\definecolor{steelblue31119180}{RGB}{31,119,180}

\begin{axis}[
legend cell align={left},
legend style={
  fill opacity=0.8,
  draw opacity=1,
  text opacity=1,
  at={(0.03,0.97)},
  anchor=north west,
  draw=lightgrey204
},
log basis x={10},
log basis y={10},
tick align=outside,
tick pos=left,
x grid style={darkgrey176},
xlabel={Geometric parameter \(\displaystyle R\)},
xmajorgrids,
xmin=0.000766195408102296, xmax=0.159320078414908,
xmode=log,
xtick style={color=black},
xtick={1e-05,0.0001,0.001,0.01,0.1,1,10},
xticklabels={
  \(\displaystyle {10^{-5}}\),
  \(\displaystyle {10^{-4}}\),
  \(\displaystyle {10^{-3}}\),
  \(\displaystyle {10^{-2}}\),
  \(\displaystyle {10^{-1}}\),
  \(\displaystyle {10^{0}}\),
  \(\displaystyle {10^{1}}\)
},
y grid style={darkgrey176},
ymajorgrids,
ymin=7.21784746572467e-19, ymax=9.40672184038981e-16,
ymode=log,
ytick style={color=black},
ytick={1e-20,1e-19,1e-18,1e-17,1e-16,1e-15,1e-14},
yticklabels={
  \(\displaystyle {10^{-20}}\),
  \(\displaystyle {10^{-19}}\),
  \(\displaystyle {10^{-18}}\),
  \(\displaystyle {10^{-17}}\),
  \(\displaystyle {10^{-16}}\),
  \(\displaystyle {10^{-15}}\),
  \(\displaystyle {10^{-14}}\)
}
]
\addplot [semithick, steelblue31119180, mark=*, mark size=3, mark options={solid}]
table {%
0.0009765625 1e-18
0.001953125 1e-18
0.00390625 1e-18
0.0078125 1e-18
0.015625 1e-18
0.03125 1.0390132318095e-17
0.0625 1.51584784318626e-17
0.125 1.72814091116336e-17
};
\addlegendentry{\(\displaystyle p=2\)}
\addplot [semithick, darkorange25512714, mark=*, mark size=3, mark options={solid}]
table {%
0.0009765625 3.99117648538082e-16
0.001953125 1e-18
0.00390625 2.08291292466353e-16
0.0078125 2.15850329170348e-16
0.015625 1e-18
0.03125 2.3762613541906e-16
0.0625 4.90238266186624e-17
0.125 1.95364494409635e-16
};
\addlegendentry{\(\displaystyle p=4\)}
\addplot [semithick, forestgreen4416044, mark=*, mark size=3, mark options={solid}]
table {%
0.0009765625 2.10265752854419e-16
0.001953125 2.07355427793771e-16
0.00390625 1e-18
0.0078125 2.15966122501717e-16
0.015625 2.22493949387205e-16
0.03125 2.3951179016056e-16
0.0625 7.07801373424756e-17
0.125 2.04277621439233e-16
};
\addlegendentry{\(\displaystyle p=6\)}
\addplot [semithick, crimson2143940, mark=*, mark size=3, mark options={solid}]
table {%
0.0009765625 2.10874012048406e-16
0.001953125 2.07355427793771e-16
0.00390625 1e-18
0.0078125 2.17764526299256e-16
0.015625 2.23874286626655e-16
0.03125 6.78962833964342e-16
0.0625 1.61529363959608e-16
0.125 2.02887370646223e-16
};
\addlegendentry{\(\displaystyle p=8\)}
\addplot [semithick, mediumpurple148103189, mark=*, mark size=3, mark options={solid}]
table {%
0.0009765625 2.16051922691197e-16
0.001953125 2.07355427793771e-16
0.00390625 1e-18
0.0078125 2.11592323340936e-16
0.015625 2.20551780659294e-16
0.03125 2.37146059351766e-16
0.0625 1.62687973383923e-16
0.125 2.26635282083231e-16
};
\addlegendentry{\(\displaystyle p=10\)}
\addplot [semithick, sienna1408675, mark=*, mark size=3, mark options={solid}]
table {%
0.0009765625 2.10265752854419e-16
0.001953125 2.07355427793771e-16
0.00390625 1e-18
0.0078125 2.1152141834369e-16
0.015625 2.17276257367028e-16
0.03125 2.38809555130283e-16
0.0625 1.6450380898264e-16
0.125 2.17292471236663e-16
};
\addlegendentry{\(\displaystyle p=12\)}
\end{axis}
      \end{tikzpicture}
    \end{subfigure}\vspace{5mm}
    \begin{subfigure}[t]{.29\linewidth}
      \caption{Laplace 3D}
      \label{fig:l2l_accuracy_laplace_3d}
      \begin{tikzpicture}[scale=0.5]

\definecolor{crimson2143940}{RGB}{214,39,40}
\definecolor{darkgrey176}{RGB}{176,176,176}
\definecolor{darkorange25512714}{RGB}{255,127,14}
\definecolor{forestgreen4416044}{RGB}{44,160,44}
\definecolor{lightgrey204}{RGB}{204,204,204}
\definecolor{mediumpurple148103189}{RGB}{148,103,189}
\definecolor{sienna1408675}{RGB}{140,86,75}
\definecolor{steelblue31119180}{RGB}{31,119,180}

\begin{axis}[
legend cell align={left},
legend style={
  fill opacity=0.8,
  draw opacity=1,
  text opacity=1,
  at={(0.03,0.97)},
  anchor=north west,
  draw=lightgrey204
},
log basis x={10},
log basis y={10},
tick align=outside,
tick pos=left,
x grid style={darkgrey176},
xlabel={Geometric parameter \(\displaystyle R\)},
xmajorgrids,
xmin=0.000766195408102296, xmax=0.159320078414908,
xmode=log,
xtick style={color=black},
xtick={1e-05,0.0001,0.001,0.01,0.1,1,10},
xticklabels={
  \(\displaystyle {10^{-5}}\),
  \(\displaystyle {10^{-4}}\),
  \(\displaystyle {10^{-3}}\),
  \(\displaystyle {10^{-2}}\),
  \(\displaystyle {10^{-1}}\),
  \(\displaystyle {10^{0}}\),
  \(\displaystyle {10^{1}}\)
},
y grid style={darkgrey176},
ylabel={\(\displaystyle \epsilon_{\mathrm{rel}}\)},
ymajorgrids,
ymin=7.26401206442623e-19, ymax=8.22796256237303e-16,
ymode=log,
ytick style={color=black},
ytick={1e-20,1e-19,1e-18,1e-17,1e-16,1e-15,1e-14},
yticklabels={
  \(\displaystyle {10^{-20}}\),
  \(\displaystyle {10^{-19}}\),
  \(\displaystyle {10^{-18}}\),
  \(\displaystyle {10^{-17}}\),
  \(\displaystyle {10^{-16}}\),
  \(\displaystyle {10^{-15}}\),
  \(\displaystyle {10^{-14}}\)
}
]
\addplot [semithick, steelblue31119180, mark=*, mark size=3, mark options={solid}]
table {%
0.0009765625 2.12536855565728e-16
0.001953125 1e-18
0.00390625 1e-18
0.0078125 1e-18
0.015625 1e-18
0.03125 2.17229129300375e-16
0.0625 2.08035193652568e-16
0.125 7.29785646718515e-17
};
\addlegendentry{\(\displaystyle p=2\)}
\addplot [semithick, darkorange25512714, mark=*, mark size=3, mark options={solid}]
table {%
0.0009765625 2.1333161805332e-16
0.001953125 2.12702097357899e-16
0.00390625 1e-18
0.0078125 1e-18
0.015625 1e-18
0.03125 1e-18
0.0625 2.03710115480754e-17
0.125 1.94989975689371e-16
};
\addlegendentry{\(\displaystyle p=4\)}
\addplot [semithick, forestgreen4416044, mark=*, mark size=3, mark options={solid}]
table {%
0.0009765625 1e-18
0.001953125 2.12961942744663e-16
0.00390625 2.12742991856103e-16
0.0078125 2.11744494659559e-16
0.015625 5.97680193187252e-16
0.03125 2.16247471522505e-16
0.0625 2.01390571896328e-16
0.125 8.33977237911211e-17
};
\addlegendentry{\(\displaystyle p=6\)}
\addplot [semithick, crimson2143940, mark=*, mark size=3, mark options={solid}]
table {%
0.0009765625 2.13370927534765e-16
0.001953125 2.12799049766025e-16
0.00390625 2.12693999126838e-16
0.0078125 4.07697656819174e-16
0.015625 2.09763427628274e-16
0.03125 2.09102273265812e-16
0.0625 5.77062265025461e-16
0.125 3.69131140532195e-16
};
\addlegendentry{\(\displaystyle p=8\)}
\addplot [semithick, mediumpurple148103189, mark=*, mark size=3, mark options={solid}]
table {%
0.0009765625 1e-18
0.001953125 2.12759712086321e-16
0.00390625 2.12598780976628e-16
0.0078125 2.11546289921504e-16
0.015625 1e-18
0.03125 1e-18
0.0625 3.95210713757199e-17
0.125 3.56697665304381e-16
};
\addlegendentry{\(\displaystyle p=10\)}
\addplot [semithick, sienna1408675, mark=*, mark size=3, mark options={solid}]
table {%
0.0009765625 1e-18
0.001953125 2.12741445691042e-16
0.00390625 2.12598780976628e-16
0.0078125 2.11395418236156e-16
0.015625 2.09838262795008e-16
0.03125 2.06303452052719e-16
0.0625 3.80866751928601e-16
0.125 2.02041492900336e-16
};
\addlegendentry{\(\displaystyle p=12\)}
\end{axis}
      \end{tikzpicture}
    \end{subfigure}\hspace{3mm}
    \begin{subfigure}[t]{.29\linewidth}
      \caption{Helmholtz 3D ($\kappa=1$)}
      \label{fig:l2l_accuracy_helmholtz_3d}
      \begin{tikzpicture}[scale=0.5]

\definecolor{crimson2143940}{RGB}{214,39,40}
\definecolor{darkgrey176}{RGB}{176,176,176}
\definecolor{darkorange25512714}{RGB}{255,127,14}
\definecolor{forestgreen4416044}{RGB}{44,160,44}
\definecolor{grey}{RGB}{128,128,128}
\definecolor{lightgrey204}{RGB}{204,204,204}
\definecolor{mediumpurple148103189}{RGB}{148,103,189}
\definecolor{sienna1408675}{RGB}{140,86,75}
\definecolor{steelblue31119180}{RGB}{31,119,180}

\begin{axis}[
legend cell align={left},
legend style={
  fill opacity=0.8,
  draw opacity=1,
  text opacity=1,
  at={(0.03,0.97)},
  anchor=north west,
  draw=lightgrey204
},
log basis x={10},
log basis y={10},
tick align=outside,
tick pos=left,
x grid style={darkgrey176},
xlabel={Geometric parameter \(\displaystyle R\)},
xmajorgrids,
xmin=0.000766195408102296, xmax=0.159320078414908,
xmode=log,
xtick style={color=black},
xtick={1e-05,0.0001,0.001,0.01,0.1,1,10},
xticklabels={
  \(\displaystyle {10^{-5}}\),
  \(\displaystyle {10^{-4}}\),
  \(\displaystyle {10^{-3}}\),
  \(\displaystyle {10^{-2}}\),
  \(\displaystyle {10^{-1}}\),
  \(\displaystyle {10^{0}}\),
  \(\displaystyle {10^{1}}\)
},
y grid style={darkgrey176},
ymajorgrids,
ymin=1e-16, ymax=0.00366976513894677,
ymode=log,
ytick style={color=black},
ytick={1e-18,1e-16,1e-14,1e-12,1e-10,1e-08,1e-06,0.0001,0.01,1},
yticklabels={
  \(\displaystyle {10^{-18}}\),
  \(\displaystyle {10^{-16}}\),
  \(\displaystyle {10^{-14}}\),
  \(\displaystyle {10^{-12}}\),
  \(\displaystyle {10^{-10}}\),
  \(\displaystyle {10^{-8}}\),
  \(\displaystyle {10^{-6}}\),
  \(\displaystyle {10^{-4}}\),
  \(\displaystyle {10^{-2}}\),
  \(\displaystyle {10^{0}}\)
}
]
\addplot [semithick, steelblue31119180, mark=*, mark size=3, mark options={solid}]
table {%
0.0009765625 4.34621023748431e-10
0.001953125 3.47571315904967e-09
0.00390625 2.77855459232879e-08
0.0078125 2.21959256650126e-07
0.015625 1.77038790650522e-06
0.03125 1.40756808290869e-05
0.0625 0.000111107793387137
0.125 0.000861411123238582
};
\addlegendentry{\(\displaystyle p=2\)}
\addplot [semithick, darkorange25512714, mark=*, mark size=3, mark options={solid}]
table {%
0.0009765625 5.96974011952414e-16
0.001953125 9.64282856395281e-15
0.00390625 3.0692759952729e-13
0.0078125 9.82026305913312e-12
0.015625 3.14317601325858e-10
0.03125 1.00588178167104e-08
0.0625 3.21432404765878e-07
0.125 1.01911687798156e-05
};
\addlegendentry{\(\displaystyle p=4\)}
\addplot [semithick, forestgreen4416044, mark=*, mark size=3, mark options={solid}]
table {%
0.0009765625 3.08270194272821e-16
0.001953125 3.06104545882556e-16
0.00390625 2.22154220683144e-16
0.0078125 1.34296396823964e-15
0.015625 1.6773885587487e-13
0.03125 2.16547870669447e-11
0.0625 2.81199606934036e-09
0.125 3.66573349954405e-07
};
\addlegendentry{\(\displaystyle p=6\)}
\addplot [semithick, crimson2143940, mark=*, mark size=3, mark options={solid}]
table {%
0.0009765625 2.3611844493286e-16
0.001953125 2.37714961783953e-16
0.00390625 3.14744009429779e-16
0.0078125 3.1465947374726e-16
0.015625 4.97073473512777e-16
0.03125 6.6415977555388e-14
0.0625 3.48110528122639e-11
0.125 1.84335003334729e-08
};
\addlegendentry{\(\displaystyle p=8\)}
\addplot [semithick, mediumpurple148103189, mark=*, mark size=3, mark options={solid}]
table {%
0.0009765625 2.45923340922629e-16
0.001953125 2.47996719174802e-16
0.00390625 3.23588589929346e-16
0.0078125 3.25379889098848e-16
0.015625 2.50619232540161e-16
0.03125 3.71044476187853e-16
0.0625 6.02252705812652e-13
0.125 1.29413434409888e-09
};
\addlegendentry{\(\displaystyle p=10\)}
\addplot [semithick, sienna1408675, mark=*, mark size=3, mark options={solid}]
table {%
0.0009765625 3.29833585882923e-16
0.001953125 2.51783379568925e-16
0.00390625 3.30928480303248e-16
0.0078125 3.26783610318511e-16
0.015625 3.27339535008666e-16
0.03125 3.15544008752056e-16
0.0625 1.36316879125844e-14
0.125 1.18420236311077e-10
};
\addlegendentry{\(\displaystyle p=12\)}
\addplot [semithick, grey, dashed]
table {%
0.0009765625 4.10752832049647e-10
0.00350765306122449 1.9034011608261e-08
0.00603874362244898 9.71225802817428e-08
0.00856983418367347 2.77586085880639e-07
0.011100924744898 6.03334155433094e-07
0.0136320153061224 1.11727641596725e-06
0.0161631058673469 1.86232249451126e-06
0.0186941964285714 2.88138201809325e-06
0.0212252869897959 4.21736461374138e-06
0.0237563775510204 5.91317990848379e-06
0.0262874681122449 8.01173752934862e-06
0.0288185586734694 1.0555947103364e-05
0.0313496492346939 1.35887182575581e-05
0.0338807397959184 1.71529606189591e-05
0.0364118303571429 2.1291583814595e-05
0.0389429209183673 2.60474974714942e-05
0.0414740114795918 3.14636112166845e-05
0.0440051020408163 3.75828346771944e-05
0.0465361926020408 4.44480774800518e-05
0.0490672831632653 5.21022492522849e-05
0.0515983737244898 6.05882596209219e-05
0.0541294642857143 6.99490182129909e-05
0.0566605548469388 8.022743465552e-05
0.0591916454081633 9.14664185755374e-05
0.0617227359693878 0.000103708879600071
0.0642538265306122 0.00011699772735615
0.0667849170918367 0.000131375871470801
0.0693160076530612 0.000146886221571053
0.0718470982142857 0.000163571687283934
0.0743781887755102 0.000181475178236472
0.0769092793367347 0.000200639604055696
0.0794403698979592 0.000221107874368633
0.0819714604591837 0.000242922898802311
0.0845025510204082 0.00026612758698376
0.0870336415816327 0.000290764848540006
0.0895647321428571 0.000316877593098078
0.0920958227040816 0.000344508730285004
0.0946269132653061 0.000373701169727813
0.0971580038265306 0.000404497821053532
0.0996890943877551 0.00043694159388919
0.10222018494898 0.000471075397861815
0.104751275510204 0.000506942142598434
0.107282366071429 0.000544584737726076
0.109813456632653 0.00058404609287177
0.112344547193878 0.000625369117662543
0.114875637755102 0.000668596721725423
0.117406728316327 0.000713771814687439
0.119937818877551 0.000760937306175619
0.122468909438776 0.00081013610581699
0.125 0.000861411123238582
};
\addlegendentry{\(\displaystyle R^{3}\)}
\addplot [semithick, red, dashed]
table {%
0.0009765625 4.78295525441754e-38
0.00350765306122449 7.92158267547809e-31
0.00603874362244898 9.24488361477502e-28
0.00856983418367347 8.75461798705874e-26
0.011100924744898 2.53084072832666e-24
0.0136320153061224 3.65490926599937e-23
0.0161631058673469 3.34517145605726e-22
0.0186941964285714 2.21709435671896e-21
0.0212252869897959 1.15529309774238e-20
0.0237563775510204 4.99733654409713e-20
0.0262874681122449 1.86350095520837e-19
0.0288185586734694 6.1565170430178e-19
0.0313496492346939 1.83916490747372e-18
0.0338807397959184 5.04641840423246e-18
0.0364118303571429 1.28750342548912e-17
0.0389429209183673 3.08437056284585e-17
0.0414740114795918 6.99333480213914e-17
0.0440051020408163 1.5105502438938e-16
0.0465361926020408 3.12519255708465e-16
0.0490672831632653 6.22145421404581e-16
0.0515983737244898 1.19636564493461e-15
0.0541294642857143 2.22963734803056e-15
0.0566605548469388 4.03874518795557e-15
0.0591916454081633 7.12822038985697e-15
0.0617227359693878 1.2285220258543e-14
0.0642538265306122 2.0714893666853e-14
0.0667849170918367 3.42306258419523e-14
0.0693160076530612 5.5517741098747e-14
0.0718470982142857 8.84943972032309e-14
0.0743781887755102 1.38799701958613e-13
0.0769092793367347 2.14446715562621e-13
0.0794403698979592 3.26687189253884e-13
0.0819714604591837 4.91145875388117e-13
0.0845025510204082 7.29295297614175e-13
0.0870336415816327 1.07035729313626e-12
0.0895647321428571 1.55373576474406e-12
0.0920958227040816 2.23210627482719e-12
0.0946269132653061 3.17531277951696e-12
0.0971580038265306 4.47525037936305e-12
0.0996890943877551 6.25194593525775e-12
0.10222018494898 8.66108628957059e-12
0.104751275510204 1.19032872132617e-11
0.107282366071429 1.62354456890064e-11
0.109813456632653 2.19845744079296e-11
0.112344547193878 2.9564581126924e-11
0.114875637755102 3.94965275846845e-11
0.117406728316327 5.2432983846815e-11
0.119937818877551 6.9187185145216e-11
0.122468909438776 9.07678004612504e-11
0.125 1.18420236311077e-10
};
\addlegendentry{\(\displaystyle R^{13}\)}
\end{axis}
      \end{tikzpicture}
    \end{subfigure}
    \begin{subfigure}[t]{.29\linewidth}
      \caption{Biharmonic 3D}
      \label{fig:l2l_accuracy_biharmonic_3d}
      \begin{tikzpicture}[scale=0.5]

\definecolor{crimson2143940}{RGB}{214,39,40}
\definecolor{darkgrey176}{RGB}{176,176,176}
\definecolor{darkorange25512714}{RGB}{255,127,14}
\definecolor{forestgreen4416044}{RGB}{44,160,44}
\definecolor{lightgrey204}{RGB}{204,204,204}
\definecolor{mediumpurple148103189}{RGB}{148,103,189}
\definecolor{sienna1408675}{RGB}{140,86,75}
\definecolor{steelblue31119180}{RGB}{31,119,180}

\begin{axis}[
legend cell align={left},
legend style={
  fill opacity=0.8,
  draw opacity=1,
  text opacity=1,
  at={(0.03,0.97)},
  anchor=north west,
  draw=lightgrey204
},
log basis x={10},
log basis y={10},
tick align=outside,
tick pos=left,
x grid style={darkgrey176},
xlabel={Geometric parameter \(\displaystyle R\)},
xmajorgrids,
xmin=0.000766195408102296, xmax=0.159320078414908,
xmode=log,
xtick style={color=black},
xtick={1e-05,0.0001,0.001,0.01,0.1,1,10},
xticklabels={
  \(\displaystyle {10^{-5}}\),
  \(\displaystyle {10^{-4}}\),
  \(\displaystyle {10^{-3}}\),
  \(\displaystyle {10^{-2}}\),
  \(\displaystyle {10^{-1}}\),
  \(\displaystyle {10^{0}}\),
  \(\displaystyle {10^{1}}\)
},
y grid style={darkgrey176},
ymajorgrids,
ymin=7.22771922167532e-19, ymax=9.14056975325473e-16,
ymode=log,
ytick style={color=black},
ytick={1e-20,1e-19,1e-18,1e-17,1e-16,1e-15,1e-14},
yticklabels={
  \(\displaystyle {10^{-20}}\),
  \(\displaystyle {10^{-19}}\),
  \(\displaystyle {10^{-18}}\),
  \(\displaystyle {10^{-17}}\),
  \(\displaystyle {10^{-16}}\),
  \(\displaystyle {10^{-15}}\),
  \(\displaystyle {10^{-14}}\)
}
]
\addplot [semithick, steelblue31119180, mark=*, mark size=3, mark options={solid}]
table {%
0.0009765625 1e-18
0.001953125 1e-18
0.00390625 1e-18
0.0078125 1e-18
0.015625 1e-18
0.03125 1e-18
0.0625 1e-18
0.125 1e-18
};
\addlegendentry{\(\displaystyle p=2\)}
\addplot [semithick, darkorange25512714, mark=*, mark size=3, mark options={solid}]
table {%
0.0009765625 1e-18
0.001953125 1.64012941492174e-16
0.00390625 1.64314184089502e-16
0.0078125 1e-18
0.015625 2.80568109175924e-16
0.03125 1.68735411722991e-16
0.0625 1.66010969482914e-16
0.125 1.62556414889725e-16
};
\addlegendentry{\(\displaystyle p=4\)}
\addplot [semithick, forestgreen4416044, mark=*, mark size=3, mark options={solid}]
table {%
0.0009765625 1e-18
0.001953125 1.64835009920372e-16
0.00390625 6.60498177022093e-16
0.0078125 5.37089600906422e-16
0.015625 1.66482430257048e-16
0.03125 1e-18
0.0625 1.70235880923664e-16
0.125 1.63737214855915e-16
};
\addlegendentry{\(\displaystyle p=6\)}
\addplot [semithick, crimson2143940, mark=*, mark size=3, mark options={solid}]
table {%
0.0009765625 1e-18
0.001953125 1.64265452383645e-16
0.00390625 5.34648383245218e-16
0.0078125 2.79085306259393e-16
0.015625 4.02252710999512e-16
0.03125 4.1110852442069e-17
0.0625 1.70831509444343e-16
0.125 1.67364740267575e-16
};
\addlegendentry{\(\displaystyle p=8\)}
\addplot [semithick, mediumpurple148103189, mark=*, mark size=3, mark options={solid}]
table {%
0.0009765625 1e-18
0.001953125 1.64441464918561e-16
0.00390625 5.34595978802719e-16
0.0078125 3.9744259034378e-16
0.015625 4.02989665255659e-16
0.03125 1.70223657694353e-16
0.0625 2.93904850800449e-16
0.125 1.59619103349967e-16
};
\addlegendentry{\(\displaystyle p=10\)}
\addplot [semithick, sienna1408675, mark=*, mark size=3, mark options={solid}]
table {%
0.0009765625 1e-18
0.001953125 1.64866822337174e-16
0.00390625 6.60654717026632e-16
0.0078125 2.79140464244554e-16
0.015625 4.0241582824107e-16
0.03125 1.6891659916103e-16
0.0625 2.87100875385344e-16
0.125 3.01838514551925e-16
};
\addlegendentry{\(\displaystyle p=12\)}
\end{axis}
      \end{tikzpicture}
    \end{subfigure}
  \caption{Relative difference between potentials evaluated using
  compressed and uncompressed L2L translations.}
  \label{fig:l2l_accuracy}
\end{figure}

\subsection{Operation Count}
\begin{figure}[ht!]
  \centering
  \begin{subfigure}[t]{.45\linewidth}
    \caption{Laplace P2M 2D}
    \label{fig:op_count_laplace_p2m_2d}
    \begin{tikzpicture}[scale=0.8]

\definecolor{darkgrey176}{RGB}{176,176,176}
\definecolor{darkorange25512714}{RGB}{255,127,14}
\definecolor{lightgrey204}{RGB}{204,204,204}
\definecolor{steelblue31119180}{RGB}{31,119,180}

\begin{axis}[
legend cell align={left},
legend style={
  fill opacity=0.8,
  draw opacity=1,
  text opacity=1,
  at={(0.03,0.97)},
  anchor=north west,
  draw=lightgrey204
},
log basis x={10},
log basis y={10},
tick align=outside,
tick pos=left,
x grid style={darkgrey176},
xlabel={Order \(\displaystyle p\)},
xmajorgrids,
xmin=1.72620478295848, xmax=44.0272213067016,
xmode=log,
xtick style={color=black},
y grid style={darkgrey176},
ylabel={FLOP Count},
ymajorgrids,
ymin=7.41850395774789, ymax=4825.85484520876,
ymode=log,
ytick style={color=black},
xtick={2,4,6,8,10,20,30,40},
xticklabels={\(\displaystyle { 2 }\),
\(\displaystyle { 4 }\),
\(\displaystyle { 6 }\),
\(\displaystyle { 8 }\),
\(\displaystyle { 10 }\),
\(\displaystyle { 20 }\),
\(\displaystyle { 30 }\),
\(\displaystyle { 40 }\)},
ytick={0.1,1,10,100,1000,10000,100000},
yticklabels={
  \(\displaystyle {10^{-1}}\),
  \(\displaystyle {10^{0}}\),
  \(\displaystyle {10^{1}}\),
  \(\displaystyle {10^{2}}\),
  \(\displaystyle {10^{3}}\),
  \(\displaystyle {10^{4}}\),
  \(\displaystyle {10^{5}}\)
}
]
\addplot [semithick, steelblue31119180, mark=*, mark size=3, mark options={solid}]
table {%
2 10
4 31
6 62
8 105
10 160
12 227
14 306
16 397
18 500
20 615
22 742
24 881
26 1032
28 1195
30 1370
32 1557
34 1756
36 1967
38 2190
};
\addlegendentry{P2M Full}
\addplot [semithick, darkorange25512714, mark=*, mark size=3, mark options={solid}]
table {%
2 11
4 42
6 91
8 160
10 249
12 358
14 487
16 636
18 805
20 994
22 1203
24 1432
26 1681
28 1950
30 2239
32 2548
34 2877
36 3226
38 3595
};
\addlegendentry{P2M Compressed}
\addplot [semithick, darkorange25512714, dashed]
table {%
2 9.9584487534626
4 39.8337950138504
6 89.6260387811634
8 159.335180055402
10 248.961218836565
12 358.504155124654
14 487.963988919668
16 637.340720221607
18 806.634349030471
20 995.84487534626
22 1204.97229916898
24 1434.01662049861
26 1682.97783933518
28 1951.85595567867
30 2240.65096952909
32 2549.36288088643
34 2877.99168975069
36 3226.53739612188
38 3595
};
\addlegendentry{\(\displaystyle p^2\)}
\end{axis}
    \end{tikzpicture}
  \end{subfigure}
  \begin{subfigure}[t]{.45\linewidth}
    \caption{Laplace P2L 2D}
    \label{fig:op_count_laplace_p2l_2d}
    \begin{tikzpicture}[scale=0.8]

\definecolor{darkgrey176}{RGB}{176,176,176}
\definecolor{darkorange25512714}{RGB}{255,127,14}
\definecolor{lightgrey204}{RGB}{204,204,204}
\definecolor{steelblue31119180}{RGB}{31,119,180}

\begin{axis}[
legend cell align={left},
legend style={
  fill opacity=0.8,
  draw opacity=1,
  text opacity=1,
  at={(0.03,0.97)},
  anchor=north west,
  draw=lightgrey204
},
log basis x={10},
log basis y={10},
tick align=outside,
tick pos=left,
x grid style={darkgrey176},
xlabel={Order \(\displaystyle p\)},
xmajorgrids,
xmin=1.72620478295848, xmax=44.0272213067016,
xmode=log,
xtick style={color=black},
y grid style={darkgrey176},
ylabel={FLOP Count},
ymajorgrids,
ymin=15.1114616085084, ymax=9830.24618399549,
ymode=log,
ytick style={color=black},
xtick={2,4,6,8,10,20,30,40},
xticklabels={\(\displaystyle { 2 }\),
\(\displaystyle { 4 }\),
\(\displaystyle { 6 }\),
\(\displaystyle { 8 }\),
\(\displaystyle { 10 }\),
\(\displaystyle { 20 }\),
\(\displaystyle { 30 }\),
\(\displaystyle { 40 }\)},
ytick={1,10,100,1000,10000,100000},
yticklabels={
  \(\displaystyle {10^{0}}\),
  \(\displaystyle {10^{1}}\),
  \(\displaystyle {10^{2}}\),
  \(\displaystyle {10^{3}}\),
  \(\displaystyle {10^{4}}\),
  \(\displaystyle {10^{5}}\)
}
]
\addplot [semithick, steelblue31119180, mark=*, mark size=3, mark options={solid}]
table {%
2 28
4 86
6 193
8 338
10 525
12 750
14 1015
16 1322
18 1667
20 2052
22 2479
24 2944
26 3449
28 3996
30 4581
32 5206
34 5871
36 6578
38 7323
};
\addlegendentry{P2L Full}
\addplot [semithick, darkorange25512714, mark=*, mark size=3, mark options={solid}]
table {%
2 23
4 50
6 78
8 106
10 134
12 162
14 190
16 218
18 246
20 274
22 302
24 330
26 358
28 386
30 414
32 442
34 470
36 498
38 526
};
\addlegendentry{P2L Compressed}
\addplot [semithick, darkorange25512714, dashed]
table {%
2 27.6842105263158
4 55.3684210526316
6 83.0526315789474
8 110.736842105263
10 138.421052631579
12 166.105263157895
14 193.789473684211
16 221.473684210526
18 249.157894736842
20 276.842105263158
22 304.526315789474
24 332.210526315789
26 359.894736842105
28 387.578947368421
30 415.263157894737
32 442.947368421053
34 470.631578947368
36 498.315789473684
38 526
};
\addlegendentry{\(\displaystyle p\)}
\addplot [semithick, steelblue31119180, dashed]
table {%
2 20.2853185595568
4 81.1412742382271
6 182.567867036011
8 324.565096952909
10 507.13296398892
12 730.271468144044
14 993.980609418283
16 1298.26038781163
18 1643.1108033241
20 2028.53185595568
22 2454.52354570637
24 2921.08587257618
26 3428.2188365651
28 3975.92243767313
30 4564.19667590028
32 5193.04155124654
34 5862.45706371191
36 6572.4432132964
38 7323
};
\addlegendentry{\(\displaystyle p^2\)}
\end{axis}
    \end{tikzpicture}
  \end{subfigure}\vspace{5mm}
  \begin{subfigure}[t]{.45\linewidth}
    \caption{Laplace M2M 2D}
    \label{fig:op_count_laplace_m2m_2d}
    \begin{tikzpicture}[scale=0.8]

\definecolor{darkgrey176}{RGB}{176,176,176}
\definecolor{darkorange25512714}{RGB}{255,127,14}
\definecolor{lightgrey204}{RGB}{204,204,204}
\definecolor{steelblue31119180}{RGB}{31,119,180}

\begin{axis}[
legend cell align={left},
legend style={
  fill opacity=0.8,
  draw opacity=1,
  text opacity=1,
  at={(0.03,0.97)},
  anchor=north west,
  draw=lightgrey204
},
log basis x={10},
log basis y={10},
tick align=outside,
tick pos=left,
x grid style={darkgrey176},
xlabel={Order \(\displaystyle p\)},
xmajorgrids,
xmin=1.72620478295848, xmax=44.0272213067016,
xmode=log,
xtick style={color=black},
y grid style={darkgrey176},
ylabel={FLOP Count},
ymajorgrids,
ymin=3.80913549355457, ymax=63199.4507897931,
ymode=log,
ytick style={color=black},
xtick={2,4,6,8,10,20,30,40},
xticklabels={\(\displaystyle { 2 }\),
\(\displaystyle { 4 }\),
\(\displaystyle { 6 }\),
\(\displaystyle { 8 }\),
\(\displaystyle { 10 }\),
\(\displaystyle { 20 }\),
\(\displaystyle { 30 }\),
\(\displaystyle { 40 }\)},
ytick={0.1,1,10,100,1000,10000,100000,1000000},
yticklabels={
  \(\displaystyle {10^{-1}}\),
  \(\displaystyle {10^{0}}\),
  \(\displaystyle {10^{1}}\),
  \(\displaystyle {10^{2}}\),
  \(\displaystyle {10^{3}}\),
  \(\displaystyle {10^{4}}\),
  \(\displaystyle {10^{5}}\),
  \(\displaystyle {10^{6}}\)
}
]
\addplot [semithick, steelblue31119180, mark=*, mark size=3, mark options={solid}]
table {%
2 30
4 124
6 299
8 590
10 1029
12 1648
14 2479
16 3554
18 4905
20 6564
22 8563
24 10934
26 13709
28 16920
30 20599
32 24778
34 29489
36 34764
38 40635
};
\addlegendentry{M2M Full}
\addplot [semithick, darkorange25512714, mark=*, mark size=3, mark options={solid}]
table {%
2 31
4 111
6 222
8 371
10 556
12 777
14 1034
16 1327
18 1656
20 2021
22 2422
24 2859
26 3332
28 3841
30 4386
32 4967
34 5584
36 6237
38 6926
};
\addlegendentry{M2M Compressed}
\addplot [semithick, darkorange25512714, dashed]
table {%
2 19.185595567867
4 76.7423822714681
6 172.670360110803
8 306.969529085873
10 479.639889196676
12 690.681440443213
14 940.094182825485
16 1227.87811634349
18 1554.03324099723
20 1918.5595567867
22 2321.45706371191
24 2762.72576177285
26 3242.36565096953
28 3760.37673130194
30 4316.75900277008
32 4911.51246537396
34 5544.63711911357
36 6216.13296398892
38 6926
};
\addlegendentry{\(\displaystyle p^2\)}
\addplot [semithick, steelblue31119180, dashed]
table {%
2 5.92433299314769
4 47.3946639451815
6 159.956990814988
8 379.157311561452
10 740.541624143461
12 1279.6559265199
14 2032.04621664966
16 3033.25849249162
18 4318.83875200467
20 5924.33299314769
22 7885.28721387957
24 10237.2474121592
26 13015.7595859455
28 16256.3697331973
30 19994.6238518734
32 24266.0679399329
34 29106.2479953346
36 34550.7100160373
38 40635
};
\addlegendentry{\(\displaystyle p^3\)}
\end{axis}
    \end{tikzpicture}
  \end{subfigure}
  \begin{subfigure}[t]{.45\linewidth}
    \caption{Laplace M2L 2D}
    \label{fig:op_count_laplace_m2l_2d}
    \begin{tikzpicture}[scale=0.8]

\definecolor{darkgrey176}{RGB}{176,176,176}
\definecolor{darkorange25512714}{RGB}{255,127,14}
\definecolor{lightgrey204}{RGB}{204,204,204}
\definecolor{steelblue31119180}{RGB}{31,119,180}

\begin{axis}[
legend cell align={left},
legend style={
  fill opacity=0.8,
  draw opacity=1,
  text opacity=1,
  at={(0.03,0.97)},
  anchor=north west,
  draw=lightgrey204
},
log basis x={10},
log basis y={10},
tick align=outside,
tick pos=left,
x grid style={darkgrey176},
xlabel={Order \(\displaystyle p\)},
xmajorgrids,
xmin=1.72620478295848, xmax=44.0272213067016,
xmode=log,
xtick style={color=black},
y grid style={darkgrey176},
ylabel={FLOP Count},
ymajorgrids,
ymin=2.72166266634766, ymax=93793.5496591982,
ymode=log,
ytick style={color=black},
xtick={2,4,6,8,10,20,30,40},
xticklabels={\(\displaystyle { 2 }\),
\(\displaystyle { 4 }\),
\(\displaystyle { 6 }\),
\(\displaystyle { 8 }\),
\(\displaystyle { 10 }\),
\(\displaystyle { 20 }\),
\(\displaystyle { 30 }\),
\(\displaystyle { 40 }\)},
ytick={0.1,1,10,100,1000,10000,100000,1000000},
yticklabels={
  \(\displaystyle {10^{-1}}\),
  \(\displaystyle {10^{0}}\),
  \(\displaystyle {10^{1}}\),
  \(\displaystyle {10^{2}}\),
  \(\displaystyle {10^{3}}\),
  \(\displaystyle {10^{4}}\),
  \(\displaystyle {10^{5}}\),
  \(\displaystyle {10^{6}}\)
}
]
\addplot [semithick, steelblue31119180, mark=*, mark size=3, mark options={solid}]
table {%
2 24.5185185185185
4 87
6 353.481481481481
8 772.851851851852
10 719.333333333333
12 1016.92592592593
14 3733.18518518519
16 2411.66666666667
18 7694.14814814815
20 10441.5185185185
22 3576
24 5304.85185185185
26 22433.8518518519
28 11024.3333333333
30 34122.8148148148
32 11827.6666666667
34 18968.6666666667
36 58335.1481481481
38 16590.6666666667
};
\addlegendentry{M2L Full}
\addplot [semithick, darkorange25512714, mark=*, mark size=3, mark options={solid}]
table {%
2 13.8888888888889
4 27
6 59.2222222222222
8 92.5555555555556
10 81.6666666666667
12 97.2222222222222
14 235.222222222222
16 157.666666666667
18 365.888888888889
20 441.888888888889
22 185
24 234.111111111111
26 712.555555555556
28 373.666666666667
30 928.555555555556
32 368.333333333333
34 513.666666666667
36 1305.88888888889
38 436.333333333333
};
\addlegendentry{M2L Compressed}
\addplot [semithick, darkorange25512714, dashed]
table {%
2 4.37599647134467
4 17.5039858853787
6 33.9353603451422
8 52.5119576561361
10 72.6837281064391
12 94.1266995183525
14 116.626801928566
16 140.03188708303
18 164.228193828751
20 189.127420926325
22 214.659027300956
24 240.765356692841
26 267.398399064599
28 294.517554455957
30 322.088038974859
32 350.079717707574
34 378.466230503947
36 407.224324141707
38 436.333333333333
};
\addlegendentry{\(\displaystyle p\log(p)\)}
\addplot [semithick, steelblue31119180, dashed]
table {%
2 8.75727137516654
4 70.0581710013323
6 203.734963020999
8 420.349026007994
10 727.275645392965
12 1130.20162158999
14 1633.75997342769
16 2241.86147204263
18 2957.89035298949
20 3784.82971908852
22 4725.346796472
24 5781.85356438394
26 6956.55142684902
28 8251.4650832434
30 9668.46882464919
32 11209.3073602132
34 12875.6125937406
36 14668.9173375119
38 16590.6666666667
};
\addlegendentry{\(\displaystyle p^2\log(p)\)}
\end{axis}
    \end{tikzpicture}
  \end{subfigure}
  \caption{Operation counts for multipole expansion (P2M), local expansion (P2L),
    multipole-to-multipole translation (M2M) and multipole-to-local translation (M2L)
    of Laplace kernel for two dimensions.}
  \label{fig:op_count}
\end{figure}

To confirm the time complexities in Table~\ref{tab:compressed}, floating point operation
counts were collected from the source code for the case of a single source and a single target.
For multipole-to-local translation, the FFT used in the accelerated evaluation of \eqref{eq:convolution_m2l} is
by far the dominant cost. To put this cost into perspective compared
to the other translation operators, while accounting for the fact that
the FFT cost is typically amortized across many different
translations, we divide the FFT cost by the maximal size of List 2
(27 in 2D and 189 in 3D, cf.~\citep{carrier1988fast}).
This assumes that each box performs only one forward FFT (of its multipole
expansion) and one inverse FFT (of its local expansion), shared by all
multipole-to-local translations involving that box.
The cost of the FFT of the derivatives (cf. \eqref{eq:convolution_m2l}) and the
cost of the evaluation of derivatives for the multipole to local
translation were not included, as these can be precomputed once per
tree and reused.

In our implementation, we have used~\eqref{eq:deriv_laplace_2d},
 \eqref{eq:deriv_laplace_3d} and~\eqref{eq:deriv_biharmonic_2d}
for the Laplace equation and the biharmonic equation to obtain $w(p) = \bigO{1}$.
For the Helmholtz equation, we have used the method of \citet{tausch2003fast} to obtain $w(p) = \bigO{p}$.
For PDEs with radially symmetric Green's functions, we have implemented the procedure of
\citet{tausch2003fast} together with symbolic differentiation for the derivatives with
respect to the (scalar) distance using the SymPy computer algebra system
\citep{meurer2017sympy}.
For other PDEs, we use symbolic differentiation to obtain the derivatives.

Figure~\ref{fig:op_count} and Figure~\ref{fig:op_count2} show the floating point operation
counts for the uncompressed representation and compressed representation as well as the expected
time complexities for the two-dimensional and three-dimensional Laplace kernel.
We observe that local expansion (``P2L'') has $\bigO{p^{d-1}}$ operations
for the Laplace equation, as the derivative calculation has amortized constant cost.
Multipole expansion (``P2M'') and multipole-to-multipole translation (``M2M'')
have $\bigO{p^d}$ operations as expected. Multipole-to-local translation (``M2L'')
has cost $\bigO{p^{d-1} \log(p)}$ as expected. Costs of multipole-to-local translation
across different values of $p$ fluctuate due to the use of the Cooley-Tukey algorithm
for the forward and inverse FFT,
which involves the divisibility of various vector sizes, cf.
Section~\ref{sec:m2l}.
We have not shown operation count graphs for multipole evaluation,
local evaluation and local-to-local translation as they are very
similar to those for local expansion, multipole expansion and
multipole-to-multipole translation, respectively.

\begin{figure}[ht!]
  \centering
  \begin{subfigure}[t]{.45\linewidth}
    \caption{Laplace P2M 3D}
    \label{fig:op_count_laplace_p2m_3d}
    \begin{tikzpicture}[scale=0.8]

\definecolor{darkgrey176}{RGB}{176,176,176}
\definecolor{darkorange25512714}{RGB}{255,127,14}
\definecolor{lightgrey204}{RGB}{204,204,204}
\definecolor{steelblue31119180}{RGB}{31,119,180}

\begin{axis}[
legend cell align={left},
legend style={
  fill opacity=0.8,
  draw opacity=1,
  text opacity=1,
  at={(0.03,0.97)},
  anchor=north west,
  draw=lightgrey204
},
log basis x={10},
log basis y={10},
tick align=outside,
tick pos=left,
x grid style={darkgrey176},
xlabel={Order \(\displaystyle p\)},
xmajorgrids,
xmin=1.72620478295848, xmax=44.0272213067016,
xmode=log,
xtick style={color=black},
y grid style={darkgrey176},
ylabel={FLOP Count},
ymajorgrids,
ymin=4.44600682475205, ymax=73766.1314509437,
ymode=log,
ytick style={color=black},
xtick={2,4,6,8,10,20,30,40},
xticklabels={\(\displaystyle { 2 }\),
\(\displaystyle { 4 }\),
\(\displaystyle { 6 }\),
\(\displaystyle { 8 }\),
\(\displaystyle { 10 }\),
\(\displaystyle { 20 }\),
\(\displaystyle { 30 }\),
\(\displaystyle { 40 }\)},
ytick={0.1,1,10,100,1000,10000,100000,1000000},
yticklabels={
  \(\displaystyle {10^{-1}}\),
  \(\displaystyle {10^{0}}\),
  \(\displaystyle {10^{1}}\),
  \(\displaystyle {10^{2}}\),
  \(\displaystyle {10^{3}}\),
  \(\displaystyle {10^{4}}\),
  \(\displaystyle {10^{5}}\),
  \(\displaystyle {10^{6}}\)
}
]
\addplot [semithick, steelblue31119180, mark=*, mark size=3, mark options={solid}]
table {%
2 16
4 60
6 139
8 276
10 487
12 788
14 1196
16 1715
18 2366
20 3171
22 4144
24 5297
26 6644
28 8207
30 10004
32 12049
34 14360
36 16953
38 19840
};
\addlegendentry{P2M Full}
\addplot [semithick, darkorange25512714, mark=*, mark size=3, mark options={solid}]
table {%
2 18
4 89
6 247
8 535
10 993
12 1657
14 2573
16 3773
18 5299
20 7193
22 9499
24 12251
26 15493
28 19259
30 23595
32 28539
34 34133
36 40417
38 47429
};
\addlegendentry{P2M Compressed}
\addplot [semithick, darkorange25512714, dashed]
table {%
2 6.91485639306021
4 55.3188511444817
6 186.701122612626
8 442.550809155854
10 864.357049132527
12 1493.60898090101
14 2371.79574281965
16 3540.40647324683
18 5040.9303105409
20 6914.85639306021
22 9203.67385916314
24 11948.871847208
26 15191.9394955533
28 18974.3659425572
30 23337.6403265782
32 28323.2517859746
34 33972.6894591048
36 40327.4424843272
38 47429
};
\addlegendentry{\(\displaystyle p^3\)}
\end{axis}
    \end{tikzpicture}
  \end{subfigure}
  \begin{subfigure}[t]{.45\linewidth}
    \caption{Laplace P2L 3D}
    \label{fig:op_count_laplace_p2l_3d}
    \begin{tikzpicture}[scale=0.8]

\definecolor{darkgrey176}{RGB}{176,176,176}
\definecolor{darkorange25512714}{RGB}{255,127,14}
\definecolor{lightgrey204}{RGB}{204,204,204}
\definecolor{steelblue31119180}{RGB}{31,119,180}

\begin{axis}[
legend cell align={left},
legend style={
  fill opacity=0.8,
  draw opacity=1,
  text opacity=1,
  at={(0.03,0.97)},
  anchor=north west,
  draw=lightgrey204
},
log basis x={10},
log basis y={10},
tick align=outside,
tick pos=left,
x grid style={darkgrey176},
xlabel={Order \(\displaystyle p\)},
xmajorgrids,
xmin=1.72620478295848, xmax=44.0272213067016,
xmode=log,
xtick style={color=black},
y grid style={darkgrey176},
ylabel={FLOP Count},
ymajorgrids,
ymin=13.6089298276799, ymax=225793.199638484,
ymode=log,
ytick style={color=black},
xtick={2,4,6,8,10,20,30,40},
xticklabels={\(\displaystyle { 2 }\),
\(\displaystyle { 4 }\),
\(\displaystyle { 6 }\),
\(\displaystyle { 8 }\),
\(\displaystyle { 10 }\),
\(\displaystyle { 20 }\),
\(\displaystyle { 30 }\),
\(\displaystyle { 40 }\)},
ytick={1,10,100,1000,10000,100000,1000000,10000000},
yticklabels={
  \(\displaystyle {10^{0}}\),
  \(\displaystyle {10^{1}}\),
  \(\displaystyle {10^{2}}\),
  \(\displaystyle {10^{3}}\),
  \(\displaystyle {10^{4}}\),
  \(\displaystyle {10^{5}}\),
  \(\displaystyle {10^{6}}\),
  \(\displaystyle {10^{7}}\)
}
]
\addplot [semithick, steelblue31119180, mark=*, mark size=3, mark options={solid}]
table {%
2 50
4 254
6 739
8 1636
10 3055
12 5123
14 7955
16 11670
18 16389
20 22229
22 29314
24 37765
26 47701
28 59238
30 72501
32 87606
34 104673
36 123822
38 145177
};
\addlegendentry{P2L Full}
\addplot [semithick, darkorange25512714, mark=*, mark size=3, mark options={solid}]
table {%
2 46
4 179
6 404
8 723
10 1134
12 1638
14 2234
16 2921
18 3700
20 4571
22 5534
24 6589
26 7737
28 8976
30 10307
32 11730
34 13245
36 14852
38 16551
};
\addlegendentry{P2L Compressed}
\addplot [semithick, darkorange25512714, dashed]
table {%
2 45.8476454293629
4 183.390581717452
6 412.628808864266
8 733.562326869806
10 1146.19113573407
12 1650.51523545706
14 2246.53462603878
16 2934.24930747922
18 3713.65927977839
20 4584.76454293629
22 5547.56509695291
24 6602.06094182825
26 7748.25207756233
28 8986.13850415512
30 10315.7202216066
32 11736.9972299169
34 13249.9695290859
36 14854.6371191136
38 16551
};
\addlegendentry{\(\displaystyle p^2\)}
\addplot [semithick, steelblue31119180, dashed]
table {%
2 21.1659133984546
4 169.327307187637
6 571.479661758274
8 1354.61845750109
10 2645.73917480682
12 4571.83729406619
14 7259.90829566992
16 10836.9476600087
18 15429.9508674734
20 21165.9133984546
22 28171.8307333431
24 36574.6983525295
26 46501.5117364047
28 58079.2663653594
30 71434.9577197842
32 86695.58128007
34 103988.132526607
36 123439.606939787
38 145177
};
\addlegendentry{\(\displaystyle p^3\)}
\end{axis}
    \end{tikzpicture}
  \end{subfigure}\vspace{5mm}
  \begin{subfigure}[t]{.45\linewidth}
    \caption{Laplace M2M 3D}
    \label{fig:op_count_laplace_m2m_3d}
    \begin{tikzpicture}[scale=0.8]

\definecolor{darkgrey176}{RGB}{176,176,176}
\definecolor{darkorange25512714}{RGB}{255,127,14}
\definecolor{lightgrey204}{RGB}{204,204,204}
\definecolor{steelblue31119180}{RGB}{31,119,180}

\begin{axis}[
legend cell align={left},
legend style={
  fill opacity=0.8,
  draw opacity=1,
  text opacity=1,
  at={(0.03,0.97)},
  anchor=north west,
  draw=lightgrey204
},
log basis x={10},
log basis y={10},
tick align=outside,
tick pos=left,
x grid style={darkgrey176},
xlabel={Order \(\displaystyle p\)},
xmajorgrids,
xmin=1.72620478295848, xmax=44.0272213067016,
xmode=log,
xtick style={color=black},
y grid style={darkgrey176},
ylabel={FLOP Count},
ymajorgrids,
ymin=2.63501915947453, ymax=1115063.55521759,
ymode=log,
ytick style={color=black},
xtick={2,4,6,8,10,20,30,40},
xticklabels={\(\displaystyle { 2 }\),
\(\displaystyle { 4 }\),
\(\displaystyle { 6 }\),
\(\displaystyle { 8 }\),
\(\displaystyle { 10 }\),
\(\displaystyle { 20 }\),
\(\displaystyle { 30 }\),
\(\displaystyle { 40 }\)},
ytick={0.1,1,10,100,1000,10000,100000,1000000,10000000,100000000},
yticklabels={
  \(\displaystyle {10^{-1}}\),
  \(\displaystyle {10^{0}}\),
  \(\displaystyle {10^{1}}\),
  \(\displaystyle {10^{2}}\),
  \(\displaystyle {10^{3}}\),
  \(\displaystyle {10^{4}}\),
  \(\displaystyle {10^{5}}\),
  \(\displaystyle {10^{6}}\),
  \(\displaystyle {10^{7}}\),
  \(\displaystyle {10^{8}}\)
}
]
\addplot [semithick, steelblue31119180, mark=*, mark size=3, mark options={solid}]
table {%
2 50
4 287
6 910
8 2243
10 4702
12 8799
14 15142
16 24435
18 37478
20 55167
22 78494
24 108547
26 146510
28 193663
30 251382
32 321139
34 404502
36 503135
38 618798
};
\addlegendentry{M2M Full}
\addplot [semithick, darkorange25512714, mark=*, mark size=3, mark options={solid}]
table {%
2 52
4 271
6 743
8 1583
10 2903
12 4815
14 7431
16 10863
18 15223
20 20623
22 27175
24 34991
26 44183
28 54863
30 67143
32 81135
34 96951
36 114703
38 134503
};
\addlegendentry{M2M Compressed}
\addplot [semithick, darkorange25512714, dashed]
table {%
2 19.6097098702435
4 156.877678961948
6 529.462166496574
8 1255.02143169558
10 2451.21373378043
12 4235.69733197259
14 6726.13048549351
16 10040.1714535647
18 14295.4784954075
20 19609.7098702435
22 26100.5238372941
24 33885.5786557807
26 43082.5325849249
28 53809.0438839481
30 66182.7708120717
32 80321.3716285173
34 96342.5045925062
36 114363.82796326
38 134503
};
\addlegendentry{\(\displaystyle p^3\)}
\addplot [semithick, steelblue31119180, dashed]
table {%
2 4.7482600655305
4 75.972161048488
6 384.60906530797
8 1215.55457677581
10 2967.66254095656
12 6153.74504492752
14 11400.5724173387
16 19448.8732284129
18 31153.3342899456
20 47482.600655305
22 69519.275619432
24 98459.9207188404
26 135615.055731617
28 182409.15867742
30 240380.665817481
32 311181.971654607
34 396579.428933173
36 498453.34863913
38 618798
};
\addlegendentry{\(\displaystyle p^4\)}
\end{axis}
    \end{tikzpicture}
  \end{subfigure}
  \begin{subfigure}[t]{.45\linewidth}
    \caption{Laplace M2L 3D}
    \label{fig:op_count_laplace_m2l_3d}
    \begin{tikzpicture}[scale=0.8]

\definecolor{darkgrey176}{RGB}{176,176,176}
\definecolor{darkorange25512714}{RGB}{255,127,14}
\definecolor{lightgrey204}{RGB}{204,204,204}
\definecolor{steelblue31119180}{RGB}{31,119,180}

\begin{axis}[
legend cell align={left},
legend style={
  fill opacity=0.8,
  draw opacity=1,
  text opacity=1,
  at={(0.03,0.97)},
  anchor=north west,
  draw=lightgrey204
},
log basis x={10},
log basis y={10},
tick align=outside,
tick pos=left,
x grid style={darkgrey176},
xlabel={Order \(\displaystyle p\)},
xmajorgrids,
xmin=1.72620478295848, xmax=44.0272213067016,
xmode=log,
xtick style={color=black},
y grid style={darkgrey176},
ylabel={FLOP Count},
ymajorgrids,
ymin=2.54421223513109, ymax=1674280.78451293,
ymode=log,
ytick style={color=black},
xtick={2,4,6,8,10,20,30,40},
xticklabels={\(\displaystyle { 2 }\),
\(\displaystyle { 4 }\),
\(\displaystyle { 6 }\),
\(\displaystyle { 8 }\),
\(\displaystyle { 10 }\),
\(\displaystyle { 20 }\),
\(\displaystyle { 30 }\),
\(\displaystyle { 40 }\)},
ytick={0.1,1,10,100,1000,10000,100000,1000000,10000000,100000000},
yticklabels={
  \(\displaystyle {10^{-1}}\),
  \(\displaystyle {10^{0}}\),
  \(\displaystyle {10^{1}}\),
  \(\displaystyle {10^{2}}\),
  \(\displaystyle {10^{3}}\),
  \(\displaystyle {10^{4}}\),
  \(\displaystyle {10^{5}}\),
  \(\displaystyle {10^{6}}\),
  \(\displaystyle {10^{7}}\),
  \(\displaystyle {10^{8}}\)
}
]
\addplot [semithick, steelblue31119180, mark=*, mark size=3, mark options={solid}]
table {%
2 29.8412698412698
4 173.857142857143
6 990.698412698413
8 2816.46031746032
10 3226
12 5415.31746031746
14 23133.3650793651
16 16941
18 60827.1746031746
20 91477.6984126984
22 34121.4285714286
24 55213.4444444444
26 254145.46031746
28 133836.142857143
30 445006.507936508
32 163473.571428571
34 278920.285714286
36 910670.46031746
38 271536
};
\addlegendentry{M2L Full}
\addplot [semithick, darkorange25512714, mark=*, mark size=3, mark options={solid}]
table {%
2 19.3174603174603
4 63.5714285714286
6 204.587301587302
8 420.460317460317
10 443
12 625.349206349206
14 1853.60317460317
16 1360.71428571429
18 3707.44444444444
20 4977.03174603175
22 2136.14285714286
24 2987.88888888889
26 10449.0317460317
28 5688.71428571429
30 15726.873015873
32 6319.95238095238
34 9537.85714285714
36 26576.0158730159
38 8861.66666666667
};
\addlegendentry{M2L Compressed}
\addplot [semithick, darkorange25512714, dashed]
table {%
2 4.6775709135418
4 37.4205673083344
6 108.822108653726
8 224.523403850007
10 388.463855838061
12 603.680987502409
14 872.649459406396
16 1197.4581538667
18 1579.91471177018
20 2021.61251470643
22 2523.97622326224
24 3088.29416155966
26 3715.74217798724
28 4407.40173667978
30 5164.2739633388
32 5987.29076933351
34 6877.32381267767
36 7835.19182306827
38 8861.66666666667
};
\addlegendentry{\(\displaystyle p^2\log(p)\)}
\addplot [semithick, steelblue31119180, dashed]
table {%
2 7.54360232172567
4 120.697637147611
6 526.498086294375
8 1448.37164577133
10 3132.41306114975
12 5841.40279184774
14 9851.36180037591
16 15449.2975548942
18 22931.6105673582
20 32602.9068109237
22 44775.0778655777
24 59766.5671467239
26 77901.7707920688
28 99510.5391738225
30 124927.755602016
32 154492.975548942
34 188550.114209163
36 227447.17327917
38 271536
};
\addlegendentry{\(\displaystyle p^3\log(p)\)}
\end{axis}
    \end{tikzpicture}
  \end{subfigure}
  \caption{Operation counts for multipole expansion (P2M), local expansion (P2L),
    multipole-to-multipole translation (M2M) and multipole-to-local translation (M2L)
    of Laplace kernel for three dimensions.}
  \label{fig:op_count2}
\end{figure}

\subsection{Expansions and Translation Operators for the Heat Kernel}
\label{sec:heat-kernel-experiment}

In the previous experiments, we have considered the Laplace,
Helmholtz, and biharmonic equations, all of which are elliptic.
For all of these PDEs, `manually' implemented FMMs are available.
In this experiment, we consider the heat equation, a parabolic
equation:
\[
  \frac{\partial u}{\partial t} = \alpha \Delta_{\b x} u,
\]
where $u$ is the temperature as a function of space $\b x$ and time $t>0$,
and $\alpha > 0$ is the diffusivity constant. It is a straightforward
observation that the heat equation, just like the elliptic operators
considered thus far, satisfies Property~\ref{prop:nice}.

The free-space Green's function for the heat equation is
\[
  G_{\alpha}(\b x - \b y, t) = \frac{1}{\sqrt{(4 \pi \alpha t)^{d}}}
    e^{\frac{- \norm{\b x - \b y}^2}{4 \alpha t}},
    \quad \b x, \b y \in \mathbb{R}^d, t > 0.
\]
The heat kernel is analytic for $t>0$, but does not satisfy the global
estimate~\eqref{eq:assumption_g}. A local expansion error estimate
is derived in Appendix~\ref{appendix:heat-kernel-local-expansion}.
While fast algorithms for this kernel have also been
considered in the literature (e.g.~\citep{greengard1991fast,tausch2007fast,messner2014fast}),
to the best of our knowledge, none of these approaches take
advantage of the PDE satisfied by the potential, while our
automatically generated translation operators do so
as soon as the PDE is provided.

Given the positivity constraint on $t$, standard
fast multipole geometry processing (interaction lists, etc.)
is not applicable in unmodified form.
Methods with geometry processing adapted to the heat kernel
exist; a prominent example is the Fast Gauss Transform of
\citet{greengard1991fast}, which treats the time difference as a
constant and only considers non-time variables in its expansions. Space-time FMMs
for the heat equation have also been considered, e.g.,
\citet{tausch2007fast}, using truncated Chebyshev
expansions for time and space separately. The cost of M2M/L2L/M2L
translations in \citet{tausch2007fast} is $\bigO{q^2\,p^{d+1}}$, where
$q$ and $p$ are the temporal and spatial expansion orders. In
contrast, our approach uses a single Taylor expansion of total order
$p$ in all $d+1$ space-time variables, with compressed translation costs
$\bigO{p^{d+1}}$ for M2M/L2L and $\bigO{p^d \log p}$ for M2L.

For the purposes of this contribution, we consider the convergence
properties of our expansions in an isolated setting,
leaving their incorporation in a full scheme including
geometry processing for future work.

\begin{figure}[ht!]
\centering
\begin{subfigure}[t]{0.48\linewidth}
\centering
  \begin{tikzpicture}[scale=0.9]
  \foreach \i in {-1.0,-0.8,...,1.0}{
    \foreach \j in {0,0.1,...,1.0} {
      \draw (\i,\j) node[fill=black!25,circle,scale=0.2] {};
    }
  }
  \draw[thick, black!55]
      (-1.0, 0) -- (-1.0, 1.0) -- (1.0, 1.0) -- (1.0, 0) -- cycle;
  \draw (0.0,0.5) node[anchor=center,fill=black!55,circle,scale=0.4,label=right:{$\b c_1$}] {};
  \node[anchor=east] at (-1.6, 0.5) {source box};
  \draw (-1.0, 0) -- ++(0, -0.15) node[below] {\footnotesize $-1$};
  \draw (1.0, 0) -- ++(0, -0.15) node[below] {\footnotesize $1$};
  \draw[thick, black!55]
      (-1.0, 2.0) -- (-1.0, 3.0) -- (1.0, 3.0) -- (1.0, 2.0) -- cycle;
  \foreach \i in {-1.0,-0.5,...,1.0}{
    \foreach \j in {2.0,2.25,...,3.0} {
      \draw (\i,\j) node[fill=black!25,circle,scale=0.2] {};
    }
  }
  \draw (0.0,2.5) node[anchor=center,fill=black!55,circle,scale=0.4,label=right:{$\b c_3$}] {};
  \node[anchor=east] at (-1.6, 2.5) {target box};
  \draw[<->] (0.0, 1.9) -- (1.0, 1.9) node[midway, below] {$\boldsymbol{h_x}$};
  \draw[<->] (1.15, 2.0) -- (1.15, 2.5) node[midway, right] {$\boldsymbol{h_t}$};
  \draw[<->] (1.45, 1.0) -- (1.45, 2.0) node[midway, right] {$\boldsymbol{t}'_{\mathbf{sep}}$};
  \draw (-1.0,2.0) node[anchor=center,fill=black!55,circle,scale=0.35,label=below left:{$\b c_2(h)$}] {};
  \draw[-{Latex[length=2mm,width=1.5mm]},shorten >=1.5mm,shorten <=1.5mm] (0.0,2.5) -- (-1.0,2.0);
  \draw[dashed] (0.0, 2.5) ellipse [x radius=1.414, y radius=0.707];
  \draw[->] (-2.0,0) -- (2.0,0) node [right] {$x$};
  \draw[->] (0.0,-0.3) -- (0.0,3.5) node [above] {$t$};
  \end{tikzpicture}
  \caption{Geometric setting for L2L and M2L experiments with
    $\b c_1 = [0, 0.5]^T$, $\b c_2(h) = \b c_3 - [h_x, h_t]^T$ and
    $\b c_3 = [0, 2.5]^T$.}
  \label{fig:heat_l2l_m2l}
\end{subfigure}\hspace{4mm}
\begin{subfigure}[t]{0.48\linewidth}
\centering
  \begin{tikzpicture}[scale=0.9]
  \foreach \i in {0.0,0.2,...,2.0}{
    \foreach \j in {0,0.1,...,1.0} {
      \draw (\i,\j) node[fill=black!25,circle,scale=0.2] {};
    }
  }
  \draw[thick, black!55]
      (0.0, 0) -- (0.0, 1.0) -- (2.0, 1.0) -- (2.0, 0) -- cycle;
  \draw (1.0,0.5) node[anchor=center,fill=black!55,circle,scale=0.4,label=right:{$\b c_4$}] {};
  \node[anchor=east] at (-0.6, 0.5) {source box};
  \draw (0.0, 0) -- ++(0, -0.15) node[below] {\footnotesize $0$};
  \draw (2.0, 0) -- ++(0, -0.15) node[below] {\footnotesize $2$};
  \draw[<->] (1.0, 1.1) -- (2.0, 1.1) node[midway, above] {$\boldsymbol{h_x}$};
  \draw[<->] (2.15, 0.5) -- (2.15, 1.0) node[midway, right] {$\boldsymbol{h_t}$};
  \draw (0.0, 1.0) node[anchor=center,fill=black!55,circle,scale=0.35,label=above left:{$\b c_5(h)$}] {};
  \draw[-{Latex[length=2mm,width=1.5mm]},shorten >=1.5mm,shorten <=1.5mm] (1.0,0.5) -- (0.0,1.0);
  \draw[dashed] (1.0, 0.5) ellipse [x radius=1.414, y radius=0.707];
  \draw[thick, black!55]
      (-1.0, 2.0) -- (-1.0, 3.0) -- (1.0, 3.0) -- (1.0, 2.0) -- cycle;
  \foreach \i in {-1.0,-0.5,...,1.0}{
    \foreach \j in {2.0,2.25,...,3.0} {
      \draw (\i,\j) node[fill=black!25,circle,scale=0.2] {};
    }
  }
  \draw (0.0,2.5) node[anchor=center,fill=black!55,circle,scale=0.4,label=right:{$\b c_3$}] {};
  \node[anchor=east] at (-1.6, 2.5) {target box};
  \draw[<->] (-1.45, 1.0) -- (-1.45, 2.0) node[midway, left] {$\boldsymbol{t}'_{\mathbf{sep}}$};
  \draw[->] (-2.0,0) -- (3.0,0) node [right] {$x$};
  \draw[->] (0.0,-0.3) -- (0.0,3.5) node [above] {$t$};
  \end{tikzpicture}
  \caption{Geometric setting for M2M experiments with
    $\b c_3 = [0, 2.5]^T$, $\b c_4 = [1, 0.5]^T$ and
    $\b c_5(h) = \b c_4 + [-h_x, h_t]^T$.}
  \label{fig:heat_m2m}
\end{subfigure}
\caption{Geometric settings for the heat-kernel translation
  experiments.}
\label{fig:heat_experiment_geometry}
\end{figure}

In the numerical experiments below, we consider the one-dimensional heat
kernel with diffusivity $\alpha=1$.
For the local-to-local translation operator, as illustrated in
Figure~\ref{fig:heat_l2l_m2l}, the sources are placed on a uniform
$11 \times 11$ grid in a source box centered at $\b c_1 = [0, 0.5]^T$
with spatial half-size $1$ and temporal half-size $0.5$. Source
densities are drawn from a uniform distribution on $(0, 1)$. Targets
are placed on a uniform $5 \times 5$ grid inside the
target box centered at $\b c_3 = [0, 2.5]^T$ with half-size
$[h_x, h_t]^T$.

Let $t'_{\text{sep}}$ be a lower bound on the minimum temporal separation between the
source box and the target box. In this experiment, we set
$t'_{\text{sep}} = 1$, which is the temporal side length of the source
box. We then define the natural diffusive length scale $L'$ as
\[
L' = \sqrt{4 \alpha t'_{\text{sep}}}.
\]
In Appendix~\ref{appendix:heat-kernel-local-expansion}, we show that
if $h:=h_x / L' \sim h_t / t'_{\text{sep}}$, the local expansion
truncation error behaves as $\bigO{h^{p+1}}$.
To test the $h$-convergence of the local-to-local translation
operator, we sweep $h$ over $\{2^{-1}, 2^{-2}, \ldots, 2^{-12}\}$
for each expansion order $p \in \{2, 3, \ldots, 9\}$.
For each $h$, the spatial and temporal half-sizes $h_x$ and $h_t$ are
defined by
\begin{align}
  h = \frac{h_x}{L'} = \frac{h_t}{t'_{\text{sep}}}, \label{adaptive_h}
\end{align}
which balances the spatial and temporal contributions to the truncation error.
For each value of $h$, a local expansion is formed at the fixed center
$\b c_3$ and translated to a local expansion $\psi_{l2l}(\b x)$ at
\begin{align*}
  \b c_2(h) = \b c_3 - [h_x, h_t]^T.
\end{align*}
We compare $\psi_{l2l}$ against a local
expansion $\psi_l^{\text{dir}}(\b x)$ formed directly at $\b c_2(h)$ from the
same source distribution. The error measure is defined as
\[
  \epsilon_{\text{rel}} =
    \sqrt{\frac{\sum_{\b x}|\psi_{l2l}(\b x) - \psi_l^{\text{dir}}(\b x)|^2}
    {\sum_{\b x}|\psi_l^{\text{dir}}(\b x)|^2}}.
\]
Figure~\ref{fig:translation_error_heat_l2l} shows the experimental
results, matching the predicted $\bigO{h^{p+1}}$ rate.

We next test the multipole-to-local (M2L) translation operator
reusing the source distribution from the L2L experiment.
Again, we sweep $h$ over $\{2^{-1}, 2^{-2}, \ldots, 2^{-12}\}$, shrinking
the target box according to \eqref{adaptive_h}. For each value of $h$,
a multipole expansion potential $\psi_{m}(\b x)$ is formed at the fixed center
$\b c_1$ and translated to a local expansion potential $\psi_{m2l}(\b x)$ at
$\b c_3$. We define the relative $\ell^2$
error as
\[
  \epsilon_{\text{rel}} =
    \sqrt{\frac{\sum_{\b x}|\psi_{m2l}(\b x) - \psi_{m}(\b x)|^2}
    {\sum_{\b x}|\psi_{m}(\b x)|^2}}.
\]
Results from this experiment are shown in
Figure~\ref{fig:translation_error_heat_m2l}, along with reference
asymptotes, confirming the expected $\bigO{h^{p+1}}$ convergence.

Last, we test the multipole-to-multipole (M2M) translation operator
with geometry set up as shown in Figure~\ref{fig:heat_m2m}. The target box is fixed
at $\b c_3 = [0, 2.5]^T$ with half-size $[1, 0.5]^T$, and contains
a uniform $5 \times 5$ grid of target points. The source box is centered at
$\b c_4 = [1, 0.5]^T$ with half-size $[h_x, h_t]^T$ defined as in
\eqref{adaptive_h}. It contains an $11 \times 11$ uniform grid of source
points with unit density. As in the previous two experiments, we sweep $h$ over
$\{2^{-1}, 2^{-2}, \ldots, 2^{-12}\}$, shrinking the source box according to
\eqref{adaptive_h}. For each $h$, a multipole expansion is formed at $\b c_4$
and translated to a multipole expansion $\psi_{m2m}(\b x)$ at
$\b c_5(h) = \b c_4 + [-h_x, h_t]^T$. We compare $\psi_{m2m}(\b x)$ against a
multipole expansion
$\psi_m^{\text{dir}}(\b x)$ formed directly at $\b c_5(h)$ from the same
source distribution.
We evaluate both expansions at the target grid and use the relative $\ell^2$ error
\[
  \epsilon_{\text{rel}} =
    \sqrt{\frac{\sum_{\b x}|\psi_{m2m}(\b x) - \psi_m^{\text{dir}}(\b x)|^2}
    {\sum_{\b x}|\psi_m^{\text{dir}}(\b x)|^2}}.
\]
We expect the error to behave as $\bigO{h^{p+1}}$. The experimental results
in Figure~\ref{fig:translation_error_heat_m2m} confirm this expected behavior.

\begin{figure}[ht!]
  \centering
    \begin{subfigure}[t]{.29\linewidth}
      \caption{local-to-local}
      \begin{tikzpicture}[scale=0.5]

\definecolor{crimson2143940}{RGB}{214,39,40}
\definecolor{darkgray176}{RGB}{176,176,176}
\definecolor{darkorange25512714}{RGB}{255,127,14}
\definecolor{forestgreen4416044}{RGB}{44,160,44}
\definecolor{gray127}{RGB}{127,127,127}
\definecolor{lightgray204}{RGB}{204,204,204}
\definecolor{mediumpurple148103189}{RGB}{148,103,189}
\definecolor{orchid227119194}{RGB}{227,119,194}
\definecolor{sienna1408675}{RGB}{140,86,75}
\definecolor{steelblue31119180}{RGB}{31,119,180}

\begin{axis}[
legend cell align={left},
legend style={
  fill opacity=0.8,
  draw opacity=1,
  text opacity=1,
  at={(0.03,0.97)},
  anchor=north west,
  draw=lightgray204
},
log basis x={10},
log basis y={10},
tick align=outside,
tick pos=left,
x grid style={darkgray176},
xlabel={$h$},
xmajorgrids,
xmin=0.000166752961029589, xmax=0.732042847972813,
xmode=log,
xtick style={color=black},
xtick={1e-05,0.0001,0.001,0.01,0.1,1,10},
xticklabels={
  \(\displaystyle {10^{-5}}\),
  \(\displaystyle {10^{-4}}\),
  \(\displaystyle {10^{-3}}\),
  \(\displaystyle {10^{-2}}\),
  \(\displaystyle {10^{-1}}\),
  \(\displaystyle {10^{0}}\),
  \(\displaystyle {10^{1}}\)
},
y grid style={darkgray176},
ylabel={$\varepsilon_{\mathrm{rel}}$},
ymajorgrids,
ymin=1e-17, ymax=6.66538479219175,
ymode=log,
ytick style={color=black},
ytick={1e-19,1e-17,1e-15,1e-13,1e-11,1e-09,1e-07,1e-05,0.001,0.1,10,1000},
yticklabels={
  \(\displaystyle {10^{-19}}\),
  \(\displaystyle {10^{-17}}\),
  \(\displaystyle {10^{-15}}\),
  \(\displaystyle {10^{-13}}\),
  \(\displaystyle {10^{-11}}\),
  \(\displaystyle {10^{-9}}\),
  \(\displaystyle {10^{-7}}\),
  \(\displaystyle {10^{-5}}\),
  \(\displaystyle {10^{-3}}\),
  \(\displaystyle {10^{-1}}\),
  \(\displaystyle {10^{1}}\),
  \(\displaystyle {10^{3}}\)
}
]
\addplot [semithick, steelblue31119180, mark=*, mark size=3, mark options={solid}]
table {%
0.5 0.110118842095988
0.25 0.00887429761824307
0.125 0.000910784102292185
0.0625 0.000103693390235402
0.03125 1.23844978879731e-05
0.015625 1.51364644901845e-06
0.0078125 1.87104929868405e-07
0.00390625 2.32583545975699e-08
0.001953125 2.89923219326441e-09
0.0009765625 3.61901749564407e-10
0.00048828125 4.5206343645341e-11
0.000244140625 5.6488027466083e-12
};
\addlegendentry{$p=2$}
\addplot [semithick, darkorange25512714, mark=*, mark size=3, mark options={solid}]
table {%
0.5 0.0860755905881934
0.25 0.00273460123032124
0.125 0.000122779196468612
0.0625 6.54432802624383e-06
0.03125 3.79334295431133e-07
0.015625 2.28685701416484e-08
0.0078125 1.4044203939233e-09
0.00390625 8.70205867787932e-11
0.001953125 5.41530345854603e-12
0.0009765625 3.37439671894638e-13
0.00048828125 2.08511105968109e-14
0.000244140625 1.3093209335917e-15
};
\addlegendentry{$p=3$}
\addplot [semithick, forestgreen4416044, mark=*, mark size=3, mark options={solid}]
table {%
0.5 0.0621971116112187
0.25 0.000908772023981423
0.125 1.72002893868749e-05
0.0625 4.06993819925623e-07
0.03125 1.11462068510602e-08
0.015625 3.28328325777452e-10
0.0078125 9.98824602442519e-12
0.00390625 3.08402009047213e-13
0.001953125 9.55654605014004e-15
0.0009765625 5.73042009249115e-16
0.00048828125 1.50118888217028e-16
0.000244140625 0
};
\addlegendentry{$p=4$}
\addplot [semithick, crimson2143940, mark=*, mark size=3, mark options={solid}]
table {%
0.5 0.03603021899416
0.25 0.000261474774979181
0.125 2.19007581995852e-06
0.0625 2.18495993838769e-08
0.03125 2.61404510843911e-10
0.015625 3.57416136276812e-12
0.0078125 5.21147845185002e-14
0.00390625 7.1320771444345e-16
0.001953125 1.47413916317597e-16
0.0009765625 4.87267434691089e-16
0.00048828125 1.72566776120983e-16
0.000244140625 0
};
\addlegendentry{$p=5$}
\addplot [semithick, mediumpurple148103189, mark=*, mark size=3, mark options={solid}]
table {%
0.5 0.0146516732889761
0.25 6.40647840594758e-05
0.125 2.72000044920169e-07
0.0625 1.22441275438608e-09
0.03125 6.1820318561763e-12
0.015625 3.62588707429706e-14
0.0078125 8.79473484389078e-16
0.00390625 4.90561077531483e-16
0.001953125 0
0.0009765625 4.87267434691089e-16
0.00048828125 1.72566776120983e-16
0.000244140625 0
};
\addlegendentry{$p=6$}
\addplot [semithick, sienna1408675, mark=*, mark size=3, mark options={solid}]
table {%
0.5 0.00744433039376009
0.25 1.23386670326764e-05
0.125 3.21911789596875e-08
0.0625 7.47532056167263e-11
0.03125 1.73855840343191e-13
0.015625 4.97909793646706e-16
0.0078125 7.51138574149683e-16
0.00390625 4.98696973606347e-16
0.001953125 0
0.0009765625 4.87267434691089e-16
0.00048828125 1.72566776120983e-16
0.000244140625 0
};
\addlegendentry{$p=7$}
\addplot [semithick, orchid227119194, mark=*, mark size=3, mark options={solid}]
table {%
0.5 0.00946142168433802
0.25 1.95520729412045e-06
0.125 3.46638791540613e-09
0.0625 4.75592483014946e-12
0.03125 5.50302284403676e-15
0.015625 1.86041373978528e-16
0.0078125 9.56203264645117e-16
0.00390625 4.93015943958645e-16
0.001953125 0
0.0009765625 4.87267434691089e-16
0.00048828125 1.72566776120983e-16
0.000244140625 0
};
\addlegendentry{$p=8$}
\addplot [semithick, gray127, mark=*, mark size=3, mark options={solid}]
table {%
0.5 0.00873091567378697
0.25 1.03160503465729e-06
0.125 3.12816349103594e-10
0.0625 2.99613455012112e-13
0.03125 3.08228882911963e-16
0.015625 1.70226261326634e-16
0.0078125 9.49869324778124e-16
0.00390625 4.93015943958645e-16
0.001953125 0
0.0009765625 4.87267434691089e-16
0.00048828125 1.72566776120983e-16
0.000244140625 0
};
\addlegendentry{$p=9$}
\addplot [semithick, crimson2143940, dash pattern=on 5.55pt off 2.4pt]
table {%
0.125 2.19007581995852e-06
0.0625 3.42199346868519e-08
0.03125 5.3468647948206e-10
0.015625 8.35447624190719e-12
0.0078125 1.305386912798e-13
0.00390625 2.03966705124687e-15
0.001953125 3.18697976757324e-17
0.0009765625 4.97965588683319e-19
0.00048828125 7.78071232317685e-21
0.000244140625 1.21573630049638e-22
};
\addlegendentry{$h^{6}$}
\addplot [semithick, gray127, dash pattern=on 5.55pt off 2.4pt]
table {%
0.125 3.12816349103594e-10
0.0625 3.05484715921479e-13
0.03125 2.98324917892069e-16
0.015625 2.91332927628974e-19
0.0078125 2.8450481213767e-22
0.00390625 2.77836730603193e-25
0.001953125 2.71324932229681e-28
0.0009765625 2.64965754130548e-31
0.00048828125 2.58755619268113e-34
0.000244140625 2.52691034441516e-37
};
\addlegendentry{$h^{10}$}
\end{axis}
      \end{tikzpicture}
      \label{fig:translation_error_heat_l2l}
    \end{subfigure}\hspace{3mm}
    \begin{subfigure}[t]{.29\linewidth}
      \caption{multipole-to-local}
      \begin{tikzpicture}[scale=0.5]

\definecolor{crimson2143940}{RGB}{214,39,40}
\definecolor{darkgray176}{RGB}{176,176,176}
\definecolor{darkorange25512714}{RGB}{255,127,14}
\definecolor{forestgreen4416044}{RGB}{44,160,44}
\definecolor{gray127}{RGB}{127,127,127}
\definecolor{lightgray204}{RGB}{204,204,204}
\definecolor{mediumpurple148103189}{RGB}{148,103,189}
\definecolor{orchid227119194}{RGB}{227,119,194}
\definecolor{sienna1408675}{RGB}{140,86,75}
\definecolor{steelblue31119180}{RGB}{31,119,180}

\begin{axis}[
legend cell align={left},
legend style={
  fill opacity=0.8,
  draw opacity=1,
  text opacity=1,
  at={(0.03,0.97)},
  anchor=north west,
  draw=lightgray204
},
log basis x={10},
log basis y={10},
tick align=outside,
tick pos=left,
x grid style={darkgray176},
xlabel={$h$},
xmajorgrids,
xmin=0.000166752961029589, xmax=0.732042847972813,
xmode=log,
xtick style={color=black},
xtick={1e-05,0.0001,0.001,0.01,0.1,1,10},
xticklabels={
  \(\displaystyle {10^{-5}}\),
  \(\displaystyle {10^{-4}}\),
  \(\displaystyle {10^{-3}}\),
  \(\displaystyle {10^{-2}}\),
  \(\displaystyle {10^{-1}}\),
  \(\displaystyle {10^{0}}\),
  \(\displaystyle {10^{1}}\)
},
y grid style={darkgray176},
ylabel={$\varepsilon_{\mathrm{rel}}$},
ymajorgrids,
ymin=1e-17, ymax=6.89999573379587,
ymode=log,
ytick style={color=black},
ytick={1e-19,1e-17,1e-15,1e-13,1e-11,1e-09,1e-07,1e-05,0.001,0.1,10,1000},
yticklabels={
  \(\displaystyle {10^{-19}}\),
  \(\displaystyle {10^{-17}}\),
  \(\displaystyle {10^{-15}}\),
  \(\displaystyle {10^{-13}}\),
  \(\displaystyle {10^{-11}}\),
  \(\displaystyle {10^{-9}}\),
  \(\displaystyle {10^{-7}}\),
  \(\displaystyle {10^{-5}}\),
  \(\displaystyle {10^{-3}}\),
  \(\displaystyle {10^{-1}}\),
  \(\displaystyle {10^{1}}\),
  \(\displaystyle {10^{3}}\)
}
]
\addplot [semithick, steelblue31119180, mark=*, mark size=3, mark options={solid}]
table {%
0.5 0.116377504050717
0.25 0.0136793426823117
0.125 0.00164131651308544
0.0625 0.000200392122994799
0.03125 2.47390902178826e-05
0.015625 3.07271258434612e-06
0.0078125 3.82850935170887e-07
0.00390625 4.7778731348757e-08
0.001953125 5.96748164063579e-09
0.0009765625 7.45631200377748e-10
0.00048828125 9.31849704955567e-11
0.000244140625 1.16469782125834e-11
};
\addlegendentry{$p=2$}
\addplot [semithick, darkorange25512714, mark=*, mark size=3, mark options={solid}]
table {%
0.5 0.0848126365306455
0.25 0.00353948192689041
0.125 0.000187123482095447
0.0625 1.08161628593746e-05
0.03125 6.50856590569886e-07
0.015625 3.99252319662318e-08
0.0078125 2.47227429754556e-09
0.00390625 1.53804157464167e-10
0.001953125 9.59057597857595e-12
0.0009765625 5.98691549265074e-13
0.00048828125 3.73673257958562e-14
0.000244140625 2.30132436193344e-15
};
\addlegendentry{$p=3$}
\addplot [semithick, forestgreen4416044, mark=*, mark size=3, mark options={solid}]
table {%
0.5 0.0554153090497912
0.25 0.000952259793060462
0.125 2.2676572551279e-05
0.0625 6.25305526300039e-07
0.03125 1.84099986871026e-08
0.015625 5.58836901009938e-10
0.0078125 1.72150321745995e-11
0.00390625 5.34180054469772e-13
0.001953125 1.66437967882912e-14
0.0009765625 5.97859917115876e-16
0.00048828125 2.23410582864662e-16
0.000244140625 3.5324313630047e-16
};
\addlegendentry{$p=4$}
\addplot [semithick, crimson2143940, mark=*, mark size=3, mark options={solid}]
table {%
0.5 0.0318221068816571
0.25 0.00028711205014552
0.125 3.36759948804574e-06
0.0625 4.55954745313259e-08
0.03125 6.63371905717611e-10
0.015625 1.0002909264153e-11
0.0078125 1.53568996571423e-13
0.00390625 2.37352812893244e-15
0.001953125 2.42402405628591e-16
0.0009765625 2.75067448855135e-16
0.00048828125 3.68823486894946e-16
0.000244140625 2.63102169233022e-16
};
\addlegendentry{$p=5$}
\addplot [semithick, mediumpurple148103189, mark=*, mark size=3, mark options={solid}]
table {%
0.5 0.0131089844791403
0.25 6.45240737372399e-05
0.125 3.8698068307291e-07
0.0625 2.63437223074673e-09
0.03125 1.91973861573034e-11
0.015625 1.44870204373896e-13
0.0078125 1.10388398860559e-15
0.00390625 2.42392045179948e-16
0.001953125 2.2157510287593e-16
0.0009765625 2.00604569556366e-16
0.00048828125 2.98894038545475e-16
0.000244140625 2.19751260551812e-16
};
\addlegendentry{$p=6$}
\addplot [semithick, sienna1408675, mark=*, mark size=3, mark options={solid}]
table {%
0.5 0.00709018730238275
0.25 1.22117368302478e-05
0.125 3.76630252799592e-08
0.0625 1.30233343778162e-10
0.03125 4.77685676528104e-13
0.015625 1.79725752410801e-15
0.0078125 5.65271031616367e-16
0.00390625 2.25179156319036e-16
0.001953125 3.25944748037028e-16
0.0009765625 3.81677085028127e-16
0.00048828125 1.94493760135739e-16
0.000244140625 2.5374747140747e-16
};
\addlegendentry{$p=7$}
\addplot [semithick, orchid227119194, mark=*, mark size=3, mark options={solid}]
table {%
0.5 0.00653891459197719
0.25 2.84270844466462e-06
0.125 3.30053357896147e-09
0.0625 5.43298022762751e-12
0.03125 9.92970750382158e-15
0.015625 5.09100065822839e-16
0.0078125 3.66622571186438e-16
0.00390625 3.19712237783076e-16
0.001953125 2.9618956773661e-16
0.0009765625 3.79562065857191e-16
0.00048828125 2.72113534990071e-16
0.000244140625 2.63090667156533e-16
};
\addlegendentry{$p=8$}
\addplot [semithick, gray127, mark=*, mark size=3, mark options={solid}]
table {%
0.5 0.00468725567081945
0.25 1.13589165420783e-06
0.125 5.00003422252562e-10
0.0625 3.32829549896332e-13
0.03125 4.55766262097934e-16
0.015625 3.20981710580206e-16
0.0078125 4.0121396969444e-16
0.00390625 3.13355556515239e-16
0.001953125 3.38065432580182e-16
0.0009765625 2.90704008824562e-16
0.00048828125 4.01209597467606e-16
0.000244140625 4.14043626061669e-16
};
\addlegendentry{$p=9$}
\addplot [semithick, crimson2143940, dash pattern=on 5.55pt off 2.4pt]
table {%
0.125 3.36759948804574e-06
0.0625 5.26187420007147e-08
0.03125 8.22167843761166e-10
0.015625 1.28463725587682e-11
0.0078125 2.00724571230754e-13
0.00390625 3.13632142548052e-15
0.001953125 4.90050222731332e-17
0.0009765625 7.65703473017706e-19
0.00048828125 1.19641167659017e-20
0.000244140625 1.86939324467213e-22
};
\addlegendentry{$h^{6}$}
\addplot [semithick, gray127, dash pattern=on 5.55pt off 2.4pt]
table {%
0.125 5.00003422252562e-10
0.0625 4.88284592043517e-13
0.03125 4.76840421917497e-16
0.015625 4.65664474528806e-19
0.0078125 4.54750463407037e-22
0.00390625 4.44092249420935e-25
0.001953125 4.33683837325131e-28
0.0009765625 4.23519372387824e-31
0.00048828125 4.13593137097484e-34
0.000244140625 4.03899547946762e-37
};
\addlegendentry{$h^{10}$}
\end{axis}
      \end{tikzpicture}
      \label{fig:translation_error_heat_m2l}
    \end{subfigure}\hspace{3mm}
    \begin{subfigure}[t]{.29\linewidth}
      \caption{multipole-to-multipole}
      \begin{tikzpicture}[scale=0.5]

\definecolor{crimson2143940}{RGB}{214,39,40}
\definecolor{darkgray176}{RGB}{176,176,176}
\definecolor{darkorange25512714}{RGB}{255,127,14}
\definecolor{forestgreen4416044}{RGB}{44,160,44}
\definecolor{gray127}{RGB}{127,127,127}
\definecolor{lightgray204}{RGB}{204,204,204}
\definecolor{mediumpurple148103189}{RGB}{148,103,189}
\definecolor{orchid227119194}{RGB}{227,119,194}
\definecolor{sienna1408675}{RGB}{140,86,75}
\definecolor{steelblue31119180}{RGB}{31,119,180}

\begin{axis}[
legend cell align={left},
legend style={
  fill opacity=0.8,
  draw opacity=1,
  text opacity=1,
  at={(0.03,0.97)},
  anchor=north west,
  draw=lightgray204
},
log basis x={10},
log basis y={10},
tick align=outside,
tick pos=left,
x grid style={darkgray176},
xlabel={$h$},
xmajorgrids,
xmin=0.000166752961029589, xmax=0.732042847972813,
xmode=log,
xtick style={color=black},
xtick={1e-05,0.0001,0.001,0.01,0.1,1,10},
xticklabels={
  \(\displaystyle {10^{-5}}\),
  \(\displaystyle {10^{-4}}\),
  \(\displaystyle {10^{-3}}\),
  \(\displaystyle {10^{-2}}\),
  \(\displaystyle {10^{-1}}\),
  \(\displaystyle {10^{0}}\),
  \(\displaystyle {10^{1}}\)
},
y grid style={darkgray176},
ylabel={$\varepsilon_{\mathrm{rel}}$},
ymajorgrids,
ymin=1e-17, ymax=5.41575839527731,
ymode=log,
ytick style={color=black},
ytick={1e-19,1e-17,1e-15,1e-13,1e-11,1e-09,1e-07,1e-05,0.001,0.1,10,1000},
yticklabels={
  \(\displaystyle {10^{-19}}\),
  \(\displaystyle {10^{-17}}\),
  \(\displaystyle {10^{-15}}\),
  \(\displaystyle {10^{-13}}\),
  \(\displaystyle {10^{-11}}\),
  \(\displaystyle {10^{-9}}\),
  \(\displaystyle {10^{-7}}\),
  \(\displaystyle {10^{-5}}\),
  \(\displaystyle {10^{-3}}\),
  \(\displaystyle {10^{-1}}\),
  \(\displaystyle {10^{1}}\),
  \(\displaystyle {10^{3}}\)
}
]
\addplot [semithick, steelblue31119180, mark=*, mark size=3, mark options={solid}]
table {%
0.5 0.0956447586215066
0.25 0.0069235615075706
0.125 0.000736013466612802
0.0625 8.58055859516404e-05
0.03125 1.03705092752514e-05
0.015625 1.27484846585839e-06
0.0078125 1.58034143824425e-07
0.00390625 1.96722174895533e-08
0.001953125 2.45391617813194e-09
0.0009765625 3.06420632408433e-10
0.00048828125 3.82826545551282e-11
0.000244140625 4.78408160646958e-12
};
\addlegendentry{$p=2$}
\addplot [semithick, darkorange25512714, mark=*, mark size=3, mark options={solid}]
table {%
0.5 0.075862887808196
0.25 0.00224731302340244
0.125 9.8516816219135e-05
0.0625 5.24796426010671e-06
0.03125 3.04596643468998e-07
0.015625 1.837488443682e-08
0.0078125 1.12873435330954e-09
0.00390625 6.99455333860517e-11
0.001953125 4.35308107453169e-12
0.0009765625 2.71449174849607e-13
0.00048828125 1.6951453128423e-14
0.000244140625 1.05139038713083e-15
};
\addlegendentry{$p=3$}
\addplot [semithick, forestgreen4416044, mark=*, mark size=3, mark options={solid}]
table {%
0.5 0.05950236727004
0.25 0.000693382881369201
0.125 1.33363684245052e-05
0.0625 3.29496120926471e-07
0.03125 9.18026688150581e-09
0.015625 2.71090892615706e-10
0.0078125 8.23675155503439e-12
0.00390625 2.53835205103037e-13
0.001953125 7.89017673543998e-15
0.0009765625 2.67545074534527e-16
0.00048828125 7.81698021607709e-17
0.000244140625 0
};
\addlegendentry{$p=4$}
\addplot [semithick, crimson2143940, mark=*, mark size=3, mark options={solid}]
table {%
0.5 0.0445063639129041
0.25 0.000199487750579689
0.125 1.68637980231235e-06
0.0625 1.93171459776494e-08
0.03125 2.58434200153511e-10
0.015625 3.73709003050485e-12
0.0078125 5.61974207200968e-14
0.00390625 8.85707194395218e-16
0.001953125 1.15693374010398e-16
0.0009765625 0
0.00048828125 0
0.000244140625 0
};
\addlegendentry{$p=5$}
\addplot [semithick, mediumpurple148103189, mark=*, mark size=3, mark options={solid}]
table {%
0.5 0.0430921428221767
0.25 6.94739639696932e-05
0.125 2.63950503114809e-07
0.0625 1.43831434237042e-09
0.03125 9.38584530928107e-12
0.015625 6.69997281715499e-14
0.0078125 5.8918689754443e-16
0.00390625 1.70580832040918e-17
0.001953125 0
0.0009765625 0
0.00048828125 0
0.000244140625 0
};
\addlegendentry{$p=6$}
\addplot [semithick, sienna1408675, mark=*, mark size=3, mark options={solid}]
table {%
0.5 0.0438300449335369
0.25 2.50579129936479e-05
0.125 4.41881091928879e-08
0.0625 1.16789194985944e-10
0.03125 3.75515437114266e-13
0.015625 1.28178095892759e-15
0.0078125 6.15041661207179e-17
0.00390625 0
0.001953125 0
0.0009765625 0
0.00048828125 0
0.000244140625 0
};
\addlegendentry{$p=7$}
\addplot [semithick, orchid227119194, mark=*, mark size=3, mark options={solid}]
table {%
0.5 0.0451251664207925
0.25 8.61695743093017e-06
0.125 6.7097054823566e-09
0.0625 8.50795246421936e-12
0.03125 1.34361594455263e-14
0.015625 1.33231923018335e-16
0.0078125 3.41163730158313e-17
0.00390625 0
0.001953125 0
0.0009765625 0
0.00048828125 0
0.000244140625 0
};
\addlegendentry{$p=8$}
\addplot [semithick, gray127, mark=*, mark size=3, mark options={solid}]
table {%
0.5 0.0473795831869089
0.25 3.16473816017636e-06
0.125 1.03133240185734e-09
0.0625 6.09405340655168e-13
0.03125 5.18586786018253e-16
0.015625 5.39440288422335e-17
0.0078125 0
0.00390625 0
0.001953125 0
0.0009765625 0
0.00048828125 0
0.000244140625 0
};
\addlegendentry{$p=9$}
\addplot [semithick, crimson2143940, dash pattern=on 5.55pt off 2.4pt]
table {%
0.125 1.68637980231235e-06
0.0625 2.63496844111305e-08
0.03125 4.11713818923913e-10
0.015625 6.43302842068615e-12
0.0078125 1.00516069073221e-13
0.00390625 1.57056357926908e-15
0.001953125 2.45400559260794e-17
0.0009765625 3.8343837384499e-19
0.00048828125 5.99122459132797e-21
0.000244140625 9.36128842394995e-23
};
\addlegendentry{$h^{6}$}
\addplot [semithick, gray127, dash pattern=on 5.55pt off 2.4pt]
table {%
0.125 1.03133240185734e-09
0.0625 1.00716054868881e-12
0.03125 9.83555223328915e-16
0.015625 9.60503147782143e-19
0.0078125 9.37991355255999e-22
0.00390625 9.16007182867187e-25
0.001953125 8.94538264518737e-28
0.0009765625 8.73572523944079e-31
0.00048828125 8.5309816791414e-34
0.000244140625 8.33103679603652e-37
};
\addlegendentry{$h^{10}$}
\end{axis}
      \end{tikzpicture}
      \label{fig:translation_error_heat_m2m}
    \end{subfigure}
  \caption{Error behavior of the three translation operators for the
  one-dimensional heat kernel with diffusivity $\alpha = 1$.}
  \label{fig:translation_error_heat}
\end{figure}

The experiments in this section support the notion that our
translation operators can be used even in situations that
do not coincide with the well-studied classical examples
of FMM translation operators in the literature.

\subsection{Application: A Boundary Integral Equation Solver for the Biharmonic Equation}\label{sec:biharmonic_ie}

To test our expansion algorithms in the context of an application, we consider a boundary integral equation solver
involving the biharmonic equation in two dimensions. We solve the
following interior boundary value problem:
\begin{align*}
  \Delta^2 u &= 0 \text{  on  } D, \\
  u &= g_1 \text{  on  } \partial D, \\
  \frac{\partial u}{\partial \b\eta} &= g_2 \text{  on  } \partial D,
\end{align*}
where $\frac{\partial u}{\partial \b\eta}$ is the derivative in the direction of the
outward unit normal $\b\eta$, and
$D$ is the interior of the ellipse defined by the curve $\gamma(t)$ where
\[
  \gamma(t) = \left[\cos(2\pi t), \frac{\sin(2\pi t)}{3}\right]^T, \quad t \in [0, 1).
\]
We place $10$ point `charges' $\{(2\cos(\frac{2\pi i}{10}), 2\sin(\frac{2\pi
i}{10}))\}_{i=0}^9$ outside $\bar{D}$ with strengths $\sigma_i$, where
$\sigma_0,\dots,\sigma_8$ are drawn from a standard normal distribution and
$\sigma_9=-\sum_{i=0}^{8}\sigma_i$.
For $\b x \in \bar{D}$, let
\begin{align*}
  \b r_1^{(i)} &= \b x_1 - 2 \cos\frac{2 \pi i}{10}, \\
  \b r_2^{(i)} &= \b x_2 - 2 \sin\frac{2 \pi i}{10}
\end{align*}
denote the components of the displacement from charge $i$ to $\b x$.
For $\b x \in \partial D$, let $(\b \eta_1(\b x), \b \eta_2(\b x))$ denote
the outward unit normal at $\b x$. The boundary data $g_1$ and $g_2$ are defined by
\begin{equation}
  g_1(\b x) = \sum_{i=0}^{9}
    \left((\b r_1^{(i)})^2 + (\b r_2^{(i)})^2\right) \log{\left((\b r_1^{(i)})^2 + (\b r_2^{(i)})^2\right)} \sigma_i
    \qquad \b x \in \partial D
  \label{eq:g_1}
\end{equation} and
\[
  g_2(\b x) = \sum_{i=0}^{9}
    2\left(\log\left((\b r_1^{(i)})^2 + (\b r_2^{(i)})^2\right) + 1\right)
    \left(\b r_1^{(i)} \b \eta_1(\b x) + \b r_2^{(i)} \b \eta_2(\b x) \right) \sigma_i \qquad \b x \in \partial D.
\]
The reference values $u_{\text{ref}}(\b x)$ for $\b x \in D$ are obtained by
evaluating the right-hand side of \eqref{eq:g_1} at interior points.

The system of integral equations used for this boundary value problem is the one described in
\citep{farkas1990biharmonic}. For layer potential evaluation, the method of quadrature by
expansion \citep{klockner2013quadrature} was used with QBX order 5.
The boundary was discretized into $100$ elements of equal length in the parameter $t$.
According to the criteria of~\citet{wala2019fast}, no refinement was required for
accurate QBX evaluation, and thus no further refinement was performed. The
generalized minimal residual method (GMRES) with a tolerance of $10^{-9}$ was used to
solve the system of integral equations. The error measure is defined in the
following:
\[
  \epsilon_{\text{rel}} = \frac{\max_{\b x \in T}\left|u_{\text{approx}}(\b x) - u_{\text{ref}}(\b x)\right|}
    {\max_{\b x \in T}|u_{\text{ref}}(\b x)|},
\]
where $T=\{0.1(\cos\frac{\pi i}{10},\sin\frac{\pi i}{10})\}_{i=0}^{19}$,
and $u_{\text{approx}}(\b x)$ is the potential at $\b x \in T$ computed via
the solution representation using the numerically-solved density function.
Errors are of very similar magnitude between the FMM based on Taylor
series expansions, the FMM using compressed Taylor series, and the
FMM using compressed Taylor series and FFT-based multipole-to-local
translations, cf. Table~\ref{tab:biharmonic_ie_error}. For high-order
expansions, the FFT-based approach differs slightly from the other two,
in line with the discussion in Section~\ref{sec:numerical_stability}.
\begin{table}
  \centering
    \begin{tabular}{lcccc} \toprule
    FMM order & Taylor Series & Compressed Taylor & Compressed Taylor with FFT \\ \midrule
    6 & 7.39569e-04 & 7.39569e-04 & 7.39569e-04 \\
    8 & 1.88837e-04 & 1.88837e-04 & 1.88837e-04 \\
    10 & 3.76156e-05 & 3.76156e-05 & 3.76156e-05 \\
    12 & 2.52269e-06 & 2.52269e-06 & 2.52269e-06 \\
    14 & 8.51901e-08 & 8.51900e-08 & 8.52020e-08 \\
    16 & 1.18990e-08 & 1.18992e-08 & 1.18917e-08 \\
    18 & 7.95453e-09 & 7.95441e-09 & 8.01112e-09 \\  \bottomrule
  \end{tabular}

  \caption{Infinity-norm relative errors for the biharmonic boundary value problem}
  \label{tab:biharmonic_ie_error}
\end{table}

\section{Conclusion}

In this paper, we have described a new family of algorithms for the automatic synthesis of
low-complexity translation operators for the FMM, based on the Taylor series of
the potential function for potentials satisfying a PDE.
We have shown that a time complexity of $\bigO{p^d}$ can be achieved
for M2M and L2L translations for a $d$-dimensional expansion of order
$p$, and $\bigO{p^{d-1}\log p}$ for M2L translation
when the PDE satisfies Property~\ref{prop:nice}.
Relative to uncompressed translations, our compression scheme
introduces no additional error for PDEs with only terms of equal order;
otherwise, the difference obeys a bound of the same order as
Taylor truncation at fixed $p$.

The FMM using compressed Taylor series achieves similar time complexities
compared to PDE-specialized fast algorithms. This includes FMMs using
spherical harmonic expansions, such as the kernel- and PDE-specific methods
of \citet{greengard1988efficient} and \citet{greengard2002new}.
In contrast to those, our approach provides a single method that is applicable
to a broad class of kernels and PDEs in two and three dimensions.

While we feel that the present contribution addresses the case of
non-oscillatory scalar kernels in a satisfactory manner, extensions
to the oscillatory setting as well as the case of systems of PDEs such
as Stokes, elasticity, or Maxwell's are of immediate interest.
In addition, a high-performance implementation of the approach could
benefit a broad range of applications.

\section*{Acknowledgments}

The authors' research was supported by the National Science Foundation
under grants DMS-1654756, SHF-1911019, and DMS-2410943, and by the
US Department of Energy under award number DE-NA0002374,
as well as the Siebel School of
Computing and Data Science at the University of Illinois at Urbana-Champaign.
The work was further supported by Anthropic PBC under the `Claude for Scientists'
program.
Any opinions, findings, and conclusions, or recommendations expressed
in this article are those of the authors and do not necessarily
reflect the views of the National Science Foundation; NSF has not
approved or endorsed its content. The authors would also like to thank
Matt Wala for helpful discussions.

Generative AI tools were used to assist in editing and reviewing this manuscript.
In addition, generative AI was used to aid in carrying out numerical experiments.
Every change made (to the manuscript and/or the experimental code) by an AI assistant
was reviewed by the authors before acceptance.
Models/services used include Qwen3.8-27B, Claude, as well as OpenAI GPT Luna, Terra and Sol,
at various versions.

In addition, the Lean formalization of the claims in the paper was carried out
assisted by OpenAI GPT Sol 6 and 6.1.

\clearpage
\appendix
\section{Heat Kernel Local Expansion}
\label{appendix:heat-kernel-local-expansion}
In this appendix, we derive the local expansion error estimate for the
heat kernel in one spatial dimension.
The heat kernel with diffusivity $\alpha > 0$ in one spatial dimension is
\[
    G_{\alpha}(x,t)
    = \frac{1}{\sqrt{4\pi \alpha t}}
      \exp\left(-\frac{x^2}{4\alpha t}\right),
    \qquad x \in \mathbb{R}, \quad t > 0.
\]
For a source $(y,t_y)$ and a target $(x,t)$ with $t>t_y$,
\[
    G_{\alpha}(x-y,t-t_y)
    = \frac{1}{\sqrt{4\pi\alpha(t-t_y)}}
      \exp\left(-\frac{(x-y)^2}{4\alpha(t-t_y)}\right).
\]
\subsection{Derivative Bound}
Let $(x_c,t_c)$ be a local expansion center with $t_c>t_y$. The
$p$-th-order Taylor expansion of $G_{\alpha}(x-y,t-t_y)$ about
$(x_c,t_c)$ is
\[
\sum_{\b q\in\mathbb{N}_0^2,\, |\b q|\le p}
    \frac{\partial_x^{q_1}\partial_t^{q_2}
      G_{\alpha}(x-y,t-t_y)|_{(x,t)=(x_c,t_c)}}{q_1!q_2!}
    (x-x_c)^{q_1}(t-t_c)^{q_2}.
\]
The Taylor remainder has the form
\[
    R_p(x,t)
    = \sum_{\b q\in\mathbb{N}_0^2,\, q_1+q_2=p+1}
      \frac{\partial_x^{q_1}\partial_t^{q_2}
        G_{\alpha}(x-y,t-t_y)|_{(x,t)=(\zeta_x,\zeta_t)}}{q_1!q_2!}
      (x-x_c)^{q_1}(t-t_c)^{q_2},
\]
where $\zeta_x$ lies between $x$ and $x_c$, and $\zeta_t$ lies between
$t$ and $t_c$. Using $\partial_t=\alpha\partial_x^2$ gives
\[
    R_p(x,t)
    = \sum_{\b q\in\mathbb{N}_0^2,\, q_1+q_2=p+1}
      \alpha^{q_2}
      \frac{\partial_x^{q_1+2q_2}
        G_{\alpha}(x-y,t-t_y)|_{(x,t)=(\zeta_x,\zeta_t)}}{q_1!q_2!}
      (x-x_c)^{q_1}(t-t_c)^{q_2}.
\]
Define
\[
    r:=x-y, \qquad s:=t-t_y, \qquad c:=\frac{1}{4\alpha s}.
\]
Then, the heat kernel can be written as
\[
    G_{\alpha}(r,s)
    = \frac{1}{\sqrt{4\pi\alpha s}}\exp(-cr^2).
\]
\begin{lemma}\label{lem:heat-gaussian-derivative}
For $k\in\mathbb{N}_0$, let $H_k$ be the $k$-th-order Hermite
polynomial~\cite[\S 18.5.13]{NIST:DLMF},
\[
    H_k(u)=k!\sum_{m=0}^{\lfloor k/2\rfloor}
       \frac{(-1)^m (2u)^{k-2m}}{m!\,(k-2m)!}.
\]
Then,
\[
    \frac{d^k}{dr^k} \exp(-cr^2)
    = (-1)^k c^{k/2} H_k(\sqrt{c}r) \exp(-cr^2).
\]
\end{lemma}

\begin{proof}
Let $u=\sqrt{c}r$,
\[
    \frac{d^k}{dr^k} \exp(-cr^2)
    = c^{k/2}\left.\frac{d^k}{du^k}\exp(-u^2)\right|_{u=\sqrt{c}\,r}.
\]
Then, the result follows from the Hermite polynomial
identity~\cite[\S 18.5]{NIST:DLMF}
\[
    \frac{d^k}{du^k}\exp(-u^2)=(-1)^k H_k(u)\exp(-u^2).
\]
\end{proof}
Combining Lemma~\ref{lem:heat-gaussian-derivative} with the bound on
Hermite polynomials \citep{greengard1991fast}:
\[
    |H_k(u)| \le K\,2^{k/2}\sqrt{k!}\,e^{u^2/2},
    \qquad K < 1.09,
\]
at $u = r/\sqrt{4\alpha s}$ gives
\begin{align}
    |\partial_r^k G_{\alpha}(r,s)|
    &= (4\alpha s)^{-k/2}\,|H_k(r/\sqrt{4\alpha s})|\,G_{\alpha}(r,s)
        \nonumber \\
    &\le K\,2^{k/2}\,\frac{\sqrt{k!}}{(4\alpha s)^{k/2}}\,
        \frac{1}{\sqrt{4\pi\alpha s}}\exp\!\left(-\frac{r^2}{8\alpha s}\right)
        \nonumber \\
    &= K\sqrt{2}\,\frac{2^{k/2}\sqrt{k!}}{(4\alpha s)^{k/2}}\,G_{2\alpha}(r,s).
    \label{eq:heat-bound-2}
\end{align}
\subsection{Local Expansion Error}
Recall the remainder of the $p$-th order Taylor expansion of $G_{\alpha}(x-y,t-t_y)$ about $(x_c,t_c)$:
\[
    R_p(x,t)
    = \sum_{\b q\in\mathbb{N}_0^2,\, q_1+q_2=p+1}
      \alpha^{q_2}
      \frac{\partial_x^{q_1+2q_2}
        G_{\alpha}(x-y,t-t_y)|_{(x,t)=(\zeta_x,\zeta_t)}}{q_1!q_2!}
      (x-x_c)^{q_1}(t-t_c)^{q_2}.
\] 
Denote $ t_{\text{sep}}:=\min(t-t_y,\, t_c-t_y)$ as the time separation.
Using \eqref{eq:heat-bound-2},
\begin{align*}
   |R_p(x,t)|
   &\lesssim_p
   |G_{2\alpha}(\zeta_x-y,\zeta_t-t_y)|
   \sum_{\b q\in\mathbb{N}_0^2,\, q_1+q_2=p+1}
   \frac{\alpha^{q_2}}{q_1!q_2!}
   \bigl(2 \alpha t_{\text{sep}}\bigr)^{-(q_1+2q_2)/2}
    |x-x_c|^{q_1}|t-t_c|^{q_2}\\
    &= |G_{2\alpha}(\zeta_x-y,\zeta_t-t_y)|
    \sum_{\b q\in\mathbb{N}_0^2,\, q_1+q_2=p+1}
    \frac{1}{q_1 ! q_2 !} \frac{1}{(2\alpha t_{\text{sep}})^{q_1/2} (2t_{\text{sep}})^{q_2}}
    |x-x_c|^{q_1}|t-t_c|^{q_2}\\
    &= |G_{2\alpha}(\zeta_x-y,\zeta_t-t_y)|
    \sum_{\b q\in\mathbb{N}_0^2,\, q_1+q_2=p+1}
    \frac{1}{q_1 ! q_2 !}\left(\frac{|x-x_c|}{\sqrt{2\alpha t_{\text{sep}}}}\right)^{q_1} \left(\frac{|t-t_c|}{2t_{\text{sep}}}\right)^{q_2}.
\end{align*}
Define
\begin{align}
    L=\sqrt{4\alpha t_{\text{sep}}}, \qquad
    \xi=\frac{|x-x_c|}{L}, \qquad
    \tau=\frac{|t-t_c|}{t_{\text{sep}}}
    \label{heat-scaling-parameters}.
\end{align}
Then,
\begin{align*}
    |R_p(x,t)|
    &\lesssim_p
    |G_{2\alpha}(\zeta_x-y,\zeta_t-t_y)|
    \sum_{\b q\in\mathbb{N}_0^2,\, q_1+q_2=p+1}
    \frac{ (\sqrt{2} \xi)^{q_1} (\tau / 2)^{q_2}}{q_1!q_2!} \\
    &=
    |G_{2\alpha}(\zeta_x-y,\zeta_t-t_y)|
    \frac{(\sqrt{2} \xi+\tau / 2)^{p+1}}{(p+1)!} \\
    &\lesssim_p
    |G_{2\alpha}(\zeta_x-y,\zeta_t-t_y)|(\xi^{2}+\tau^{2})^{(p+1)/2}.
\end{align*}

\bibliography{references}
\end{document}